\documentclass[a4paper, 12pt]{article}      
\usepackage{amsmath}
\usepackage{amssymb}
\usepackage{amsthm}
\usepackage{hyperref}
\usepackage{algorithm,algorithmic}

\makeatother

\newtheorem{theo}{Theorem}[section]
\newtheorem{prop}[theo]{Proposition}
\newtheorem{lemma}[theo]{Lemma}

\newtheorem{coro}[theo]{Corollary}
\newtheorem{claim}{Claim}

\theoremstyle{definition}\newtheorem{de}{Definition}[section]
\theoremstyle{definition}\newtheorem{rem}{Remark}[section]

\numberwithin{equation}{section}

\def\det{\mathrm{det}\/}

\def\R{\mathbb{R}}            
\def\eps{\varepsilon}
\def\aa{a}
\def\xb{b}
\def\ap{\mathrm{AppSteer}\/}
\def\ex{\mathrm{Exact}\/}
\def\gf{\mathrm{GlobalFree}\/}
\def\gfmod{\mathrm{GlobalFree\_Modified}\/}

\def\hX{\widehat X}
\def\hd{\hat d}
\def\hu{\hat u}
\def\bx{\overline{x}}

\def\bet{\overline{\eta}}

\def\cg{\mathcal{G}}
\def\cj{\mathcal{J}}
\def\ck{\mathcal{K}}
\def\ci{\mathcal{I}}
\def\cl{\mathcal{L}}
\def\ch{\mathcal{H}}
\def\cv{\mathcal{V}}

\def\tk{\widetilde{k}}
\def\tn{\widetilde{n}}

\def\tx{\widetilde{x}}

\begin{document}


\title{A Global Steering Method for Nonholonomic Systems \thanks{This work was supported by grants from Digiteo and R\'egion Ile-de-France, by the ANR project {GCM}, program ``Blanche'',
project number NT09\_504490, and by the Commission
of the European Communities under the 7th Framework Program Marie Curie Initial Training Network (FP7-PEOPLE-2010-ITN), project SADCO, contract number 264735. }
}
\author{Yacine Chitour
\thanks{Y. Chitour is with L2S, Universit\'e Paris-Sud XI, CNRS and
Sup\'elec, Gif-sur-Yvette, and Team  GECO, INRIA Saclay -- Ile-de-France, France.
       {\tt\small yacine.chitour@lss.supelec.fr}}
, Fr\'ed\'eric Jean
\thanks{F. Jean is with UMA, ENSTA ParisTech, and Team  GECO, INRIA Saclay -- Ile-de-France, France.
        {\tt\small frederic.jean@ensta-paristech.fr}}
~and Ruixing Long
\thanks{R. Long is with Department of Chemistry, Princeton University, Princeton, NJ 08544, USA.
        {\tt\small rlong@princeton.edu}}
 }

\date{}




\maketitle

\begin{abstract}
  In this paper, we present an iterative steering algorithm for
  nonholonomic systems (also called driftless control-affine systems)
  and we prove its global convergence under the sole assumption that
  the Lie Algebraic Rank Condition (LARC) holds true everywhere.  That
  algorithm is an extension of the one introduced in \cite{Jean1} for
  regular systems. The first novelty here consists in the explicit
  algebraic construction, starting from the original control system,
  of a lifted control system which is regular. The second contribution
  of the paper is an exact motion planning method for nilpotent
  systems, which makes use of sinusoidal control laws and which is a
  generalization of the algorithm described in \cite{Murray} for
  chained-form systems.
\end{abstract}


\tableofcontents


\section{Introduction}
Nonholonomic systems have been attracting the attention of the
scientific community for several years, due to the theoretical
challenges they offer and the numerous important applications they cover.
From the point of view of control theory, a nonholonomic system
is a driftless
control-affine system and is written as
\begin{equation}\label{CS0}
(\Sigma)\quad \dot{x}=\sum_{i=1}^{m}u_iX_i(x),\  x\in\Omega,\ \ u=(u_1,\dots,u_m)\in\mathbb{R}^m,
\end{equation}
where $\Omega$ is an open connected subset of $\mathbb{R}^n$, and $X_1,\dots,X_m$
are $C^\infty$ vector fields on $\Omega$. Admissible inputs are $\R^m$-valued measurable functions $u(\cdot)$ defined on some interval $[0,T]$ and a trajectory of $(\Sigma)$, corresponding to some $x_0\in \Omega$ and to an admissible input $u(\cdot)$,
is the (maximal) solution  $x(\cdot)$ in $\Omega$ of the Cauchy problem defined by
$\dot{x}(t)=\sum_{i=1}^{m}u_i(t)X_i(x(t))$, $t\in [0,T]$, and $x(0)=x_0$.

In this paper, we address the \emph{motion planning problem} (MPP for short) for $(\Sigma)$,
namely determine a procedure which associates with every pair of points $(p,q)\in\Omega\times\Omega$ an admissible input $u(\cdot)$ defined on some interval $[0,T]$, such that the corresponding trajectory of $(\Sigma)$ starting from $p$ at $t=0$ reaches $q$ at $t=T$.
As for the existence of a solution to MPP,
this is equivalent to the
complete controllability of $(\Sigma)$. After the works of Chow and
Rashevsky in the thirties~\cite{cho40,ras38}, and that of Sussmann
and Stefan in the seventies~\cite{Sussmann73,stefan}, the issue of
complete controllability for nonholomic systems is well-understood and
it is usually guaranteed by assuming that the Lie Algebraic Rank
Condition (also known as the H\"ormander condition) holds for
$(\Sigma)$. This easily checkable condition is not only sufficient for
complete controllability but also necessary when the vector fields are
analytic.  From a practical viewpoint, assuming the LARC is, in a
sense, the minimal requirement to ensure complete controllability for
$(\Sigma)$ and this is what we will do for all the control systems
considered hereafter.


As for the construction of the solutions of the MPP, we present, in this paper, a complete procedure solving the
MPP for a nonholonomic system subject to the sole
LARC. By ``complete procedure'', we mean that
the following properties must be guaranteed by the proposed procedure.
\begin{enumerate}
\item Global character of the algorithm: for every pair of points $(p,q)$ in $\Omega$, the algorithm must produce a steering control. (Note that the core of many algorithms consists in
a local procedure and turning the latter into a global one is not always a trivial issue.)
\item Proof of convergence of the algorithm.
\item Regarding numerical implementations, no prohibitive limitation on the state
dimension $n$.
\item Usefulness for practical applications, e.g., robustness with
  respect to the dynamics, ``nice'' trajectories produced by the
  algorithm, (no cusps neither large oscillations), and possibility of
  localizing the algorithm in order to handle obstacles (i.e.,
  reducing the working space $\Omega$ to any smaller open and connected
  subset of $\R^n$).
\end{enumerate}

There exist of course several algorithms addressing the MPP
in different contexts but most of them fail to verify all the aforementioned properties.

At first,  in the case of specific classes of driftless
nonholonomic systems (i.e. where more is known than the sole LARC), effective techniques have been proposed, among which a Lie bracket
method for steering nilpotentizable systems (see \cite{Lafferriere}
and \cite{Lafferriere2}), sinusoidal controls for chained-form
systems (see \cite{Murray}), averaging techniques for left-invariant systems defined on a Lie group (see \cite{Leonard, Bullo-Lewis-Leonard}), and a trajectory generation method for
flat systems (see \cite{Fliess}). Depending on the applications, these methods
turn out to be extremely efficient, especially when the system to be steered is shown to be flat with an explicit flat output.


However, the class of systems considered previously is rather restrictive: for $2$-input nonholonomic systems (i.e. $m=2$), under suitable regularity assumptions, a flat system admits a feedback chained-form transformation (cf. \cite{Murray2, Martin}) and thus is exactly nilpotentizable; on the other hand,  when the dimension of the state space is large enough,
exact nilpotentizability is clearly a non generic property among $2$-input nonholonomic systems. Moreover, there exist standard nonholonomic
systems whose kinematic model does not fall into any of the
aforementioned categories. For instance,  mobile robots with more than
one trailer cannot be transformed in chained-form unless each trailer
is hinged to the midpoint of the previous wheel axle, an unusual
situation in real vehicles. Another similar example is the
rolling-body problem: even the simplest model in this category, the
so-called plate-ball system, does not allow any chained-form
transformation and is not flat.

Regarding general nonholonomic systems, various steering techniques have been proposed in the literature, and we only mention three of them: the iterative method, the generic loop method, and the continuation method.
The first one, introduced in~\cite{Lafferriere} and improved in~\cite{Lafferriere2}, is an approximation procedure which is exact for nilpotent systems.
This method is proved to be convergent with the sole assumption of the LARC and actually meets most of the requirements to be a complete  procedure in the sense defined previously.
However, either the resulting trajectories in~\cite{Lafferriere2} contain a large number
of cusps (exponential with respect to the degree of nonholonomy), or the computation
of the steering control in~\cite{Lafferriere} requires the inversion of a system of algebraic
equations. The latter turns out to be numerically intractable as soon as the dimension of the state
is larger than six. Let us also mention a less important limitation for practical use.
The iterative method described in~\cite{Lafferriere,Lafferriere2} makes use of several nonlinear changes of coordinates, which must be performed by numerical integration of ODEs at each step of the iterative method, thus leading to spurious on-line computations.

The generic loop method, presented in~\cite{Sontag}, is based on a local deformation procedure and requires an a priori
estimate of some ``critical distance" which is, in general, an
unknown parameter in practice. That fact translates into a severe drawback for
constructing a globally valid algorithm. The continuation method of~\cite{Sussmann3} and~\cite{Chitour} belongs to the class of Newton-type methods.  Proving its convergence amounts to
show the global existence for the solution of a non linear differential equation,
 which relies on handling the
abnormal extremals associated to the control system. That latter issue
turns out to be a hard one, see~\cite{YJT1,YJT2} for instance. This is
why, in the current state of knowledge, the continuation method can be
proved to converge only under restrictive assumptions (see
\cite{Chitour1,Chitour2,Chitour3}).

The algorithm considered in the present paper takes as starting point the
globally convergent algorithm for steering {\it regular} nonholonomic systems
discussed in
\cite{Jean1}.  As the iterative method of~\cite{Lafferriere,Lafferriere2}, that algorithm
can be casted in the realm of nonlinear
geometric control and roughly works as follows:
one first solves the motion
planning problem for a control system which is nilpotent and ``approximates" system
$(\ref{CS0})$ in a suitable sense; then, one applies the resulting
input $\hat{u}$ to $(\ref{CS0})$ and iterates the procedure from the
current point. If we use $\hat{x}(t, \aa,\hat{u})$, $t\in[0,T]$ to
denote a trajectory of the ``approximate" control system starting from
$\aa$, a local version of this algorithm is summarized below,
where $d$ is an appropriate distance (to be defined in the next section)
and $e$ is a fixed positive real number.
\begin{algorithm}
\caption{Local Steering Algorithm}
\label{algo_rrt}
\begin{algorithmic}
\REQUIRE $x_0$, $x_1$, $e$
\STATE $k=0$;
\STATE $x^k=x_0$;
\WHILE {$d(x^k,x_1)>e$}
  \STATE Compute $\hat{u}^k$ such that $x_1=\hat{x}(T,x^k,\hat{u}^k)$;
   \STATE  $x^{k+1}=$ AppSteer $(x^k, x_1):=x(T,x^k,\hat{u}^k)$;
   \STATE  $k=k+1$;
\ENDWHILE
\end{algorithmic}
\end{algorithm}

We note that Algorithm \ref{algo_rrt} converges {\em locally} provided
that the function AppSteer is {\em locally contractive} with respect to the
distance $d$, i.e., for $x_1\in\Omega$, there exists
$\varepsilon_{x_1}>0$ and $c_{x_1} <1$ such that
\begin{equation}
d(\textrm{AppSteer}(x,x_1),x_1)\leq c_{x_1} d(x_1,x),
\end{equation} for $x\in\Omega$ and $d(x_1,x)<\varepsilon_{x_1}$.

Assume now that we have a {\em uniformly locally contractive} function $\ap$
on a connected compact set $K \subset \Omega$, i.e. there exists
$\varepsilon_K>0$ and $c_K\in (0,1)<1$ such that
\begin{equation}
d(\textrm{AppSteer}(x,x_1),x_1)\leq c_Kd(x_1,x),
\end{equation} for $x, ~x_1\in K$ and $d(x_1,x)<\varepsilon_K$.

Based on the local algorithm, a global
approximate steering algorithm on $K$ can be built along the line of the
following idea (a similar procedure is proposed
in~\cite{Lafferriere2}): consider a parameterized
path $\gamma \subset K$ connecting $x_0$ to $x_1$. Then pick a
finite sequence of intermediate goals $\{x^d_0=x_0,x^d_1, \dots,
x^d_j=x_1\}$ on $\gamma$ such that $d(x^d_{i-1},x^d_{i}) <
\varepsilon_K/2$,
$i=0,\dots, j$. One can prove that the iterated application
of a uniformly locally
contractive $\ap(x^{i-1}, x^d_{i})$ from the current
state to the next subgoal (having set $x^d_{i}=x_1$, for $ i \geq
j$) yields a sequence $x^i$  converging to $x_1$.

To turn the above idea into a practically efficient algorithm, three
issues must be successfully addressed:
\begin{itemize}
\item[(P-1)]\label{contraction}Construct a {\em uniformly} locally contractive
  local approximate steering method.
\item[(P-2)]\label{u-hat}The ``approximate" control $\hat{u}^k$ must
  be {\em exact} for steering the ``approximate system" from the
  current point $x^k$ to the final point $x_1$. As this computation
  occurs at each iteration, it must be performed in a reasonable time.
\item[(P-3)]\label{critical}Since the knowledge of the ``critical
  distance" $\varepsilon_K$ is not available in practice, the
  algorithm should achieve global convergence without explicit
  knowledge of $\varepsilon_K$.
\end{itemize}

Issues (P-1) and (P-3) are solved in~\cite{Jean1} under the assumption that the control
system is {\it regular} (cf. Definition~\ref{de:regular-point} below). The solution proposed therein relies on the understanding of the
geometry defined by the nonholonomic system (cf.~\cite{bjr}). That geometry is a sub-Riemannian one and it
endows the working space $\Omega$ with a sub-Riemannian metric $d$ for which the aforementioned
function AppSteer is contractive. Moreover, the approximation of the original system adapted to the motion
planning turns out to be the approximation at the first order with respect to $d$ (cf.\cite{Jean1}).
However, the regularity assumption for the control system is rather
restrictive: in general, nonholonomic systems do exhibit
singularities (cf. \cite{Vershik}). A solution also exists in the case of a non regular control
system~\cite{Vendittelli}, but only when the state dimension $n$ is less than or
equal to 5. In the present paper, we completely remove the
regularity assumption and solve Issues (P-1) and (P-3) for every nonholonomic control
system. The solution is based on an explicit desingularization procedure: adding
new variables (thus augmenting the dimension of the state space), we
construct  a ``lifted" control system which is regular and whose
projection is the original control system. The solution
of Issue (P-1) described in~\cite{Jean1} can thus be applied to the ``lifted''
control system, as well as the globally convergent motion planning
algorithm solving Issue (P-3) proposed therein. Note that other
desingularization procedures already
exist~\cite{Bellaiche1,goo76,Jea:01b,rot76}, but we insist on the fact
that the one we propose here involves only  {\it explicit } polynomial transformations.
It numerically translates to the fact that these changes of variables can be performed off-line
once each local procedure is identified.

As regards Issue (P-2), several algorithms were proposed for computing
$\hat{u}$, i.e. for controlling nilpotent
systems. In~\cite{Lafferriere}, the authors make use of piecewise
constant controls and obtain smooth controls by imposing some special
parameterization (namely by requiring the control system to stop
during the control process). In that case, the smoothness of the
inputs is recovered by using a reparameterization of the time, which
cannot prevent in general the occurrence of cusps or corners for the
corresponding trajectories. However, smoothness of the {\em
trajectories} is generally mandatory for robotic applications.
Therefore, the method proposed in \cite{Lafferriere} is not adapted to
such applications. In~\cite{Lafferriere2}, the proposed controls are
polynomial (in time), but an algebraic system must be inverted in
order to access to these inputs.  The size and the degree of this
algebraic system increase exponentially with respect to the dimension
of state space, and there does not exist a general efficient exact
method to solve it. Even the existence of solutions is a non trivial
issue. Furthermore, the methods~\cite{Lafferriere}
and~\cite{Lafferriere2} both make use of exponential coordinates which
are not explicit and thus require in general numerical integrations
of nonlinear differential equations. That prevents the use of these
methods in an iterative scheme such as Algorithm~\ref{algo_rrt}. Let
us also mention the path approximation method by Liu and
Sussmann~\cite{Liu}, which uses unbounded sequences of sinusoids. Even
though this method bears similar theoretical aspects with our method
(see especially the argument strategy in order to prove
Lemma~\ref{Free_Family}, which is borrowed from~\cite{Liu}), it is not
adapted from a numerical point of view to the motion planning issue
since it relies on a limit process of highly oscillating inputs.  In
the present paper, we present an {\it exact} steering algorithm 
 for general nilpotent systems is provided, which uses
sinusoidal inputs and which can be applied for controlling the
approximate (nilpotent) system used in~\cite{Jean1}. Our method
generalizes the one proposed in~\cite{Murray} for controlling
chained-form systems, which is briefly recalled next: after having brought
the system under a ``canonical'' form, the authors of~\cite{Murray}
proceed by controlling component after component by using, for each
component, two sinusoids with suitably chosen frequencies. In the
present paper, we show for general nilpotent systems that, with more
frequencies for each component, one can steer an arbitrary component
independently on the other components. We are also able
to construct inputs which give rise to $C^1-$trajectories.

We now describe in a condensed manner the global motion
planning strategy developed in this paper. The latter is presented as an algorithmic procedure associated with a given nonholonomic system $(\Sigma)$ defined on $\Omega\subset \R^n$. The required inputs are initial and final points $x^{\textrm{initial}}$ and  $x^{\textrm{final}}$ belonging to $\Omega$, a tolerance $e>0$, and a compact convex set $K\subset\Omega$ (of appropriate size) equal to the closure of its interior which is a neighborhood of both $x^{\textrm{initial}}$ and $x^{\textrm{final}}$. For instance, $K$ can be chosen to be a large enough compact tubular neighborhood constructed around a curve joining $x^{\textrm{initial}}$ and  $x^{\textrm{final}}$. The global steering method is summarized in Algorithm~\ref{algo}.

\begin{algorithm}[h]
\caption{Global Approximate Steering Algorithm: Global$(x^{\textrm{initial}}, x^{\textrm{final}},e,K)$}
\label{algo}
\begin{algorithmic}[1]
\STATE Build a decomposition of $K$ into a finite number of compacts sets $\cv_{\cj_i}^c$, with
$i=1,\dots,M$ (Section~\ref{Sec-lifting}).\\
\STATE Construct the connectedness graph $\textsf{G}:=(\textsf{N},~\textsf{E})$ associated with this decomposition
and choose a simple path $\textsf{p}:=\{j_0,j_1\dots,j_{\bar{M}}\}$ in $\textsf{G}$ such that
$x^{\textrm{initial}}\in\cv_{\cj_{j_0}}^c$ and  $x^{\textrm{final}}\in\cv_{\cj_{j_{\bar{M}}}}^c$
 (Section~\ref{Sec-lifting}).
\STATE Choose a sequence $(x^{i})_{i=1,\dots,\bar{M}-1}$ such that
$x^{i}\in\cv_{\cj_{j_i}}^c\cap\cv_{\cj_{j_{i+1}}}^c$.
\STATE Set $x:=x^{\textrm{initial}}$.
\FOR{$i = 1,\dots,\bar{M}-1$}
  \STATE Apply the Desingularization Algorithm at $\aa:=x^{i}$ with $\cj:=\cj_{i}$ (Section \ref{Sec-lifting}).\\
   \COMMENT{\textsf{the output is an $m$-tuple of vector fields $\xi$ on $\cv_{\cj_i}\times \R^{\tn}$ which is free up to step $r$.}}
\STATE  Let $\ap$ be the LAS method associated to the approximation $\mathcal{A}^{\xi}$ of $\xi$ on $\cv_{\cj_i}\times \R^{\tn}$ defined in Section \ref{bellaiche-construction} and
 to its steering law $\ex_{m,r}$ constructed in Section \ref{sec:sub-opt}).
  \STATE Set $\tx_0:=(x,0)$, $\tx_1:=(x^i,0)$, and $\cv^c:=\cv_{\cj_i}^c \times \overline{B}_R(0)$ with $R>0$ large enough.
  \STATE Apply GlobalFree$(\widetilde{x}_0,\tx_1,e,\cv^c,\ap)$ to $\xi$ (Section~\ref{GASA}).\\
\COMMENT{\textsf{the algorithm stops at a point $\tx$ which is $e$-close to $\tx_1$;}}
\RETURN  $x:=\pi (\tx)$.  \\
\COMMENT{\textsf{$\pi:\cv_{\cj_i}\times \R^{\tn} \to \cv_{\cj_i}$ is the canonical projection.}}
\ENDFOR
\end{algorithmic}
\end{algorithm}

The paper is devoted to the construction of the various steps of this algorithm.  We will also show that each of these steps is conceived so that the overall construction is a complete procedure in the sense defined previously. In particular, the convergence issue is addressed in the following theorem.

\begin{theo}\label{th1}
Let $(\Sigma)$ be a nonholonomic system on $\Omega\subset \R^n$ satisfying the LARC. For every $e>0$, every connected compact set $K$ which is equal to the closure of its interior, and every pair of points $(x^{\textrm{\emph{initial}}}, x^{\textrm{\emph{final}}})$ in the interior of $K$, Algorithm~\ref{algo} steers, in a finite number of steps, the control system $(\Sigma)$ from $x^{\textrm{\emph{initial}}}$ to a point $x \in K$ such that $d(x,x^{\textrm{\emph{final}}}) <e$.
\end{theo}

Before providing the structure of the paper, we mention possible extensions of our
algorithm. The first one concerns the working space $\Omega$. Since it is an arbitrary open connected set of $\R^n$, one can extend the algorithm to the case where the working space is a smooth connected manifold of finite dimension. From a numerical point of view, there would be the additional burden of computing the charts. A second extension deals with the stabilization issue. Indeed, at the heart of the algorithm lies an iterative procedure such as Algorithm~\ref{algo_rrt}, which can be easily adapted for stabilization tasks (cf.~\cite{OriVen:05}). Another possible generalization takes advantage of devising from our algorithm a globally regular input, one can then address the motion planning of dynamical extensions of the nonholonomic control systems considered in the present paper. Finally, let us point out the modular nature of Algorithm \ref{algo}: one can propose other approaches to obtain uniformly contractive local methods (other desingularization methods or different ways of dealing with singular points), or replace $\ex_{m,r}(\cdot)$ by more efficient control strategies for general nilpotent systems.

The paper is organized as follows. In Section \ref{Sec-Back}, we
define properly the notion of first order approximation. We then
propose, in Section \ref{Desing}, a purely polynomial desingularization
procedure based on a {\em lifting} method. In Section \ref{regular-case}, we describe in detail the globally convergent
steering algorithm given in \cite{Jean1} for regular systems together
with a proof of convergence. In Section \ref{control-nilpotent}, we
present an exact steering procedure for general nilpotent systems using sinusoids, and we gather, in the appendix, the proof of Theorem \ref{th1} and some additional comments.





\section{Notations and Definitions}\label{Sec-Back}

Let $n$ and $m$ be two positive integers. Let $\Omega$ and $VF(\Omega)$ be respectively an open connected subset of $\mathbb{R}^n$ and the set of $C^\infty$ vector fields on $\Omega$. Consider $m$ vector fields $X_1,\dots,X_m$ of $VF(\Omega)$, and the associated driftless control-affine nonholonomic system given by
\begin{equation}\label{CS}
\dot{x}=\sum_{i=1}^{m}u_iX_i(x),\  x\in\Omega,
\end{equation}
where $u=(u_1,\cdots,u_m)\in\mathbb{R}^m$ and the input $u(\cdot)=(u_1(\cdot),\dots,u_m(\cdot))$ is an integrable vector-valued function defined on $[0,T]$, with $T$
a fixed positive real number.

We also assume that $(\ref{CS})$ is {\it complete}, i.e., for every $a\in \Omega$ and input
$u(\cdot)$, the Cauchy problem defined by $(\ref{CS})$ starting from $\aa$ at $t=0$ and corresponding to $u(\cdot)$ admits a unique (absolutely continuous) solution $x(\cdot,\aa,u)$ defined on $[0,T]$ and called the trajectory of $(\ref{CS})$ starting from $\aa$ at $t=0$ and corresponding to the
input $u(\cdot)$. A point $x\in\Omega$ is said to be {\em accessible} from $\aa$ if there exists an input $u:[0,T]\rightarrow\mathbb{R}^m$ and a time $t\in[0,T]$ such that $x=x(t,\aa,u)$. Then, System~(\ref{CS}) is said to be completely controllable if any two points
in $\Omega$ are accessible one from each other (see \cite{Sussmann1}).

We next provide a classical condition ensuring that System~(\ref{CS}) is controllable. We first need the following definition.

\begin{de}[\emph{Lie Algebraic Rank Condition (LARC)}]\label{larc}
Let $L(X)$ be the Lie algebra generated by the vectors fields $X_1,\dots,X_m$ and
$L(x)$ be the linear subspace of  $\mathbb{R}^n$ equal to the evaluation of $L(X)$ at every point $x\in\Omega$ (see \cite{Bellaiche1}). If $L(x)=\mathbb{R}^n$, we say that the {\it Lie Algebraic Rank Condition} (LARC for short) is verified at $x\in\Omega$. If this is the case at every point $x\in\Omega$, we say that System~(\ref{CS}) satisfies the LARC.
\end{de}
Chow's Theorem essentially asserts that, if System~(\ref{CS}) satisfies the
LARC then it is completely controllable (cf. \cite{cho40}).

\begin{rem}
For the sake of clarity, we assume through this paper that the control set is equal to $\R^m$. However, it is well-known that Chow's theorem only requires that the convex hull of the control set contains a neighborhood of the origin in $\R^m$ (see for instance \cite[Chapter 4, Theorem 2]{Jurd}). We will explain later how we can adapt our method to the case with constraints on the control set (see Appendix \ref{rem:control-set}). Moreover, it is worth recalling that complete controllability for $(\Sigma)$ does not imply that LARC holds true for $(\Sigma)$ if the vector fields $X_1,\dots, X_m$ are only smooth, but this is the case if $X_1,\dots, X_m$ are analytic (cf. \cite[Chapter 5]{Agrachev-CTGV}).
\end{rem}

Throughout this paper, we will only consider driftless control-affine nonholonomic systems of the type (\ref{CS}) verifying the LARC, and thus completely controllable. In
that context, the {\it motion planning problem}
will be defined as follows: find a procedure which furnishes, for every two points
$x_0$, $x_1\in \Omega$,  an input $u$ steering
(\ref{CS}) from
$x_0$ to $x_1$, i.e., $x(T,x_0,u)=x_1$. 

Our solution to this problem relies heavily on the underlying
geometry, which is a \emph{sub-Riemannian geometry}. We provide in Section \ref{recall-SRG}
the useful definitions and refer the reader to \cite{Bellaiche1} for more
details. We then introduce in Section \ref{def-ASM} a notion of approximate steering method related to this geometry.

\subsection{Basic facts on sub-Riemannian geometry}\label{recall-SRG}

\subsubsection{Sub-Riemannian distance and Nonholonomic order}

\begin{de}[\emph{Length of an input}]\label{length}
The {\it length of an input } $u$ is defined by
$$\ell(u)=\int_0^T\sqrt{u_1^2(t)+\dots+u_m^2(t)}dt,$$
and the {\it length of a trajectory} $x(\cdot,\aa,u)$ is defined
by $\ell(x(\cdot,\aa,u)):=\ell(u)$.
\end{de}

The appropriate notion of distance associated with the control system
$(\ref{CS})$ and closely related to the notion of accessibility is
that of {\it sub-Riemannian distance}, also called {\it control
  distance}.

\begin{de}[\emph{Sub-Riemannian distance}]\label{dis}
The vector fields $X_1,\dots,X_m$ induce a function $d$ on $\Omega$, defined by $d(x_1,x_2):=\inf_u\ell(x(\cdot,x_1,u)), \hbox{ for every points } x_1,x_2 \hbox{ in }\Omega$, where the infimum is taken over all the inputs $u$ such that $x(\cdot, x_1,u)$ is defined on $[0,T]$ and $x(T,x_1,u)=x_2$. We say that the function $d$ is the {\it sub-Riemannian distance} associated with $X_1,\dots,X_m$.
\end{de}%
\begin{rem}
The function $d$ defined above is a {\em distance} in the usual sense, i.e., it verifies (i) $d(x_1,x_2)\geq 0$ and $d(x_1,x_2)=0$ if and only if $x_1=x_2$; (ii) symmetry: $d(x_1,x_2)=d(x_2,x_1)$; (iii) triangular inequality: $d(x_1,x_3)\leq d(x_1,x_2)+d(x_2,x_3)$.
Notice that one always has  $d(x_1,x_2)<\infty$ since the control
system is assumed to be completely controllable.
\end{rem}

\begin{de}[\emph{Nonholonomic derivatives of a function}]\label{nd}
If $f:\Omega\rightarrow \mathbb{R}^n$ is a smooth function, the {\em first-order
nonholonomic derivatives} of $f$ are the Lie derivatives $X_i f$ of $f$
along $X_i$, $i=1,\ldots,m$. Similarly, $X_i(X_j f)$, $i,
j=1,\ldots,m$, are called the
{\em second-order nonholonomic derivatives} of $f$, and more generally, $X_{i_1}\cdots X_{i_k}f$, $i_1,\dots,i_k\in\{1,\dots,m\}$ are the {\em$k^{\textrm{th}}$-order nonholonomic derivatives} of $f$, where $k$ is any positive integer.
\end{de}

\begin{prop}[{\cite[Proposition 4.10, page 34]{Bellaiche1}}]\label{ndf-order}
Let $s$ be a non-negative integer. For a smooth function $f$ defined near $\aa\in\Omega$, the following conditions are equivalent:
\begin{itemize}
\item[{\em(i)}] $f(x)=O(d^s(x,\aa))$ for $x$ in a neighborhood of $\aa$;
\item[{\em(ii)}] all the nonholonomic derivatives of order $\leq s-1$ of $f$ vanish at $\aa$.
\end{itemize}
\end{prop}

\begin{de}[\emph{Nonholonomic order of a function}]\label{order-f}
Let $s$ and $f$ be respectively a non-negative integer and a smooth real-valued function defined on $\Omega$. If Condition (i) or (ii) of Proposition \ref{ndf-order} holds,
we say that $f$ is of {\em order} $\geq s$ at $\aa\in\Omega$. If $f$ is of order $\geq s$ but not of order $\geq s+1$ at $\aa$, we say that $f$ is of {\em order} $s$ at $\aa$. The order of $f$ at $\aa$ will be denoted by ord$_{\aa}(f)$.
\end{de}

\begin{de}[\emph{Nonholonomic order of a vector field}]\label{order-X}
Let $q$ be an integer. A vector field $Y\in VF(\Omega)$ is of {\em
  order} $\geq q$ at $\aa\in\Omega$ if, for every non-negative integer $s$ and
every smooth function $f$ of order $s$ at $\aa$, the Lie derivative $Yf$ is of order $\geq q+s$ at $\aa$. If $Y$ is of order $\geq q$ but not $\geq q+1$, it is of {\em order} $q$ at $\aa$. The order of $Y$ at $\aa$ will be denoted by ord$_{\aa}(Y)$.
\end{de}


\begin{de}[\emph{Nonholonomic first order approximation at $\aa$}]\label{approx-point}
An $m$-tuple $\widehat{X}^{\aa}:=\{\widehat{X}_1^{\aa},\dots,\widehat{X}_m^{\aa}\},$
defined on $B(\aa,\rho_{\aa}):=\{x\in\Omega,
\ d(x,\aa)\leq\rho_{\aa}\}$ with $\rho_{\aa}>0$, is said to be a {\it nonholonomic first order approximation of $X:=\{X_1,\dots,X_m\}$ at $\aa\in\Omega$}, if the vector fields $X_i-\widehat{X}_i^{\aa}$, for $i=1,\dots,m$, are of order $\geq 0$ at $\aa$. The positive number $\rho_{\aa}$ is called {\em the approximate radius at $\aa$}.
\end{de}

\begin{rem}\label{rem-nh-order}
As a consequence of Definition \ref{approx-point}, one gets that the nonholonomic order at $\aa$ defined by the vector fields $\hX_1^{\aa},\dots,\hX_m^{\aa}$ coincides with the one defined by $X_1,\dots,X_m$.
\end{rem}

\subsubsection{Privileged coordinates}

The changes of coordinates take an important place in this paper, whether it is to estimate the sub-Riemannian distance, or to compute the
order of functions and vector fields, or to transform a control system
into a normal form. To avoid heavy notations, we will need some
conventions and simplifications that we fix now for the rest of the
paper.

A point in $\Omega \subset \R^n$ is denoted by $x=(x_1,\dots,x_n)$ and
the canonical
basis of $\R^n$ by $(\partial_{x_1}, \dots,\partial_{x_n})$. Even
though $x$ is a point, we will sometimes refer to $(x_1,\dots,x_n)$ as
the original coordinates.
A \emph{system of local coordinates}
$y=(y_1,\dots,y_n)$ at a
point $\aa \in \Omega$ is defined as a diffeomorphism $\varphi$
between an open neighborhood $N_\aa \subset \Omega$ of $\aa$ and an
open neighborhood
$N_{\varphi(\aa)} \subset \R^n$ of $\varphi(\aa)$, $
\varphi :  x \mapsto y=(y_1,\dots, y_n)$. If the diffeomorphism $\varphi$ is defined on $\Omega$, then $y=(y_1,\dots,y_n)$ is said
to be a \emph{system of global coordinates} on $\Omega$. A system of global coordinates is said to be
\emph{affine} (resp. \emph{linear}) if the corresponding diffeomorphism $\varphi$ is affine (resp. linear).
If $f$ is a function  defined on $N_\aa$, the function $f \circ
\varphi^{-1}$ defined on $N_{\varphi(\aa)}$ will be called \emph{$f$
  (expressed) in coordinates} $(y_1,\dots,y_n)$. If $X \in VF
(\Omega)$ is a vector field, the push-forward $\varphi_*X=d\varphi \circ X \circ
\varphi^{-1} \in VF(N_{\varphi(\aa)})$ will be called \emph{$X$
  (expressed) in coordinates} $(y_1,\dots,y_n)$.

For the sake of simplicity, we will in general not introduce the
notation $\varphi$ and,
with a slight abuse of the notation, replace it by $y$. Thus we write
$y(x)$ or $(y_1(x),\dots, y_n(x))$ instead of $\varphi (x)$. The
function $f \circ \varphi^{-1}$ will be denoted  by
$f(y)$, and the  vector field $\varphi_*X$ by $X(y)$
or $y_*X$. The values at a point $\bar y \in N_{\varphi(\aa)}$
will be denoted respectively by $f(y)_{|y=\bar y}$ and $X(y)_{|y=\bar
  y}$.

A special class of coordinates, called {\em privileged coordinates}
and defined below, turns out to be a useful tool to compute the order
of functions and vector fields, and to estimate the sub-Riemannian
distance $d$.

We will use $L^s(X)$ to denote the Lie sub-algebra of elements of
length (cf. Definitions \ref{de:length-bracket} and \ref{de:length-bracket2}) not greater than $s\in\mathbb{N}$. Take $x \in {\Omega}$ and
let $L^s(x)$ be the vector space generated by the values at $x$ of
elements belonging to $L^s(X)$. Since System~(\ref{CS}) verifies the
LARC at every point $x\in\Omega$, there exists a smallest integer
$r:=r(x)$ such that $\dim L^{r}(x)=n$. This integer is called the {\em
  degree of nonholonomy} at $x$.

\begin{de}[\emph{Growth vector}]
\label{gro-vec}
For $\aa\in\Omega$, let
  $n_s(\aa):=\dim L^s(\aa)$, $s=1,\dots,r$. The
  sequence $(n_1(\aa),\dots,n_r(\aa))$ is the {\em growth vector} of
  $X$ at $\aa$.
\end{de}%

\begin{de}[\emph{Regular and singular points}]\label{de:regular-point}
A point $\aa\in\Omega$ is said to be {\em regular} if the growth
vector remains constant in a neighborhood of $\aa$ and, otherwise,
$\aa$ is said to be {\em singular}. The nonholonomic System~(\ref{CS}) (or the $m$-tuple $X$) is said to be \emph{regular} if every point in $\Omega$ is regular.
\end{de} 

\noindent Note that regular points form an open and dense set in ${\Omega}$.

\begin{de}[\emph{Weight}]\label{de-weight}
For $\aa\in\Omega$ and $j=1,\ldots,n$, let $w_j:=w_j(\aa)$ be the integer
defined by setting $w_j:=s$ if $n_{s-1} < j \leq n_s$, with
$n_s:=n_s(\aa)$ and $n_0:=0$. The integers $w_j$, for $j=1,\dots,n$
are called the {\em weight at }$\aa$.
\end{de}%
\begin{rem}\label{adapt_frame}
The meaning of Definition~\ref{de-weight} can be understood in another
way. Choose first some vector fields $W_1,\dots,W_{n_1}$
in $L^1(X)$ such that $W_1(\aa),\dots,W_{n_1}(\aa)$ form a basis of
$L^1(\aa)$. Choose then other vectors fields $W_{n_1+1},\dots,
W_{n_2}$ in $L^2(X)$ such that $W_{1}(\aa),\dots, W_{n_2}(\aa)$ form a
basis of $L^2(\aa)$ and, for every positive integer $s$, choose
$W_{n_{s-1}+1},\dots,W_{n_s}$ in $L^s(X)$ such that
$W_1(\aa),\dots,W_{n_s}(\aa)$ form a basis of $L^s(\aa)$. We obtain in
this way a sequence of vector fields $W_1,\dots,W_n$ such that
\begin{equation}\label{adapted-frame}
\left\{\begin{array}{l}
W_1(\aa),\dots,W_n(\aa) \textrm{ is a basis of } \R^n,\\
W_i\in L^{w_i}(X), i=1,\dots,n.
\end{array}\right.
\end{equation}
A sequence of vector fields verifying Eq. (\ref{adapted-frame}) is
called an {\em adapted frame at} $\aa$. The word ``adapted" means
``adapted to the flag $L^1(\aa)\subset L^2(\aa)\subset\cdots\subset
L^r(\aa)=\R^n$'', since the values at $\aa$ of an adapted
frame contain a basis $W_1(\aa),\dots,W_{n_s}(\aa)$ of every subspace
$L^s(\aa)$ of the flag. The values of $W_1,\dots,W_n$ at a point $b$
close to $\aa$ also form a basis of $\R^n$. However, if $\aa$
is singular, this basis may be not adapted to the flag
$L^1(\xb)\subset L^2(\xb)\subset\cdots\subset
L^{r(\xb)}(\xb)=\R^n$.
\end{rem} 

\begin{de}[\emph{Privileged coordinates at $\aa$}]A {\em system of privileged coordinates} at $\aa\in\Omega$ is a system of local coordinates $(z_1,\ldots,z_n)$
centered at $\aa$ (the image of $\aa$ is $0$) such that ord$_{\aa}(z_j(x))=w_j$, for
$j=1,\dots,n$.
\end{de}%
\begin{rem}\label{rem-weight-order}
For every system of local coordinates $(y_1,\dots,y_n)$ centered at $\aa$, we have, up to a re-ordering, ord$_{\aa}(y_j)\leq w_j$ or, without re-ordering, $\sum_{j=1}^n\textrm{ord}_{\aa}(y_j)\leq\sum_{j=1}^n w_j$.
\end{rem}

The order at $\aa\in\Omega$ of functions and vector fields expressed in a system of privileged
coordinates $(z_1,\dots,z_n)$ centered at $\aa$ can be evaluated algebraically as follows:

\begin{itemize}

\item the order of the monomial $z_1^{\alpha_1} \ldots z_n^{\alpha_n}$
is equal to its {\em weighted degree} $w(\alpha):= w_1\alpha_1+\dots+w_n\alpha_n;$

\item the order of a function $f(z)$ at $z=0$ is
the least weighted degree of the monomials occurring in the
Taylor expansion of $f(z)$ at 0;

\item the order of the monomial vector field $z_1^{\alpha_1}\dots
z_n^{\alpha_n} \partial_{z_j}$ is equal to its {\em weighted degree}
 $w(\alpha)-w_j,$ where one assigns the weight $-w_j$ to $\partial_{z_j}$ at $0$;

\item the order of a vector field $W(z) = \sum_{j=1}^n W_j(z)
\partial_{z_j}$ at $z=0$ is the least weighted degree of the monomials
occurring in the Taylor expansion of $W$ at 0.
\end{itemize}

\begin{de}[\emph{Continuously varying system of privileged coordinates}] A {\em continuously varying system of
privileged coordinates on $\Omega$} is a mapping $\Phi$ taking values
in $\R^n$, defined and continuous on a neighborhood of the set
$\{(x,x), x \in\Omega \} \subset \Omega \times \Omega$, and so that
the partial mapping $z:=\Phi
(\aa, \cdot)$ is a system of privileged coordinates at $\aa$. In
this case, there exists a continuous function $\bar{\rho} : \Omega
\rightarrow (0,+\infty)$ such that the coordinates $\Phi (\aa,
\cdot)$ are defined on $B(\aa,\bar{\rho}(\aa))$. We call $\bar{\rho}$ an
{\em injectivity radius function} of $\Phi$.
\end{de}%

\begin{de}[\emph{Pseudo-norm}]\label{de:pseudonorm} Let $\aa \in \Omega$ and $w_1,\dots,w_n$ the weights at $\aa$. The application from $\R^n$ to $\R$ defined by $\|z\|_{\aa}:=|z_1|^{1/w_1} + \dots + |z_n|^{1/w_n}, \ z=(z_1,\dots,z_n) \in \R^n,$
is called  {\em the pseudo-norm} at $\aa$.
\end{de}%

\subsubsection{Distance and error estimates}

Privileged coordinates provide estimates of the sub-Riemannian
distance $d$, according to the following result.

\begin{theo}[{Ball-Box Theorem \cite{Bellaiche1}}]
\label{le:bbox}
Consider $(X_1,\dots,X_m )\in V\!F (\Omega)^m$, a point $\aa \in \Omega$, and a
system of privileged coordinates $z$ at $\aa$. There exist positive
constants $C_d(\aa)$ and $\eps_d (\aa)$ such that, for every $x\in\Omega$ with
$d(\aa,x)<\eps_d (\aa)$, one has
\begin{equation}
\frac{1}{C_d(\aa)} \, \|z (x) \|_{\aa} \leq d(\aa,x) \leq C_d(\aa) \,
\|z (x) \|_{\aa}.
\label{eq:ballbox}
\end{equation}

If $\Omega$ contains only regular points and if $\Phi$ is a
continuously varying system of privileged coordinates on $\Omega$,
then there exist continuous positive functions $C_d(\cdot)$ and
$\eps_d (\cdot)$ on $\Omega$ such that Eq. (\ref{eq:ballbox})
holds true with $z=\Phi(\aa,\cdot)$ at all $(x,\aa)$ satisfying
$d(x,\aa)<\eps_d (\aa)$.
\end{theo}

\begin{coro}\label{theo:unif_contr}
Let $K$ be a compact subset of $\Omega$. Assume that $K$ only contains regular points and there exists a
continuously varying system of privileged coordinates $\Phi$ on $K$. Then, there exist positive constants $C_K$ and $\eps_K$ such
that, for every pair $(\aa,x)\in K\times K$ verifying $d(\aa,x) <
\eps_K$, one has
\begin{equation}\label{uniform-estimation}
\frac{1}{C_K} \Vert \Phi(\aa,x)\Vert_{\aa} \leq d(\aa,x) \leq C_K
\Vert \Phi(\aa,x)\Vert_{\aa}.
\end{equation}
\end{coro}

Privileged coordinates also allow one to measure the error obtained
when $X$ is replaced by an approximation $\widehat{X}$.

\begin{prop}[{\cite[Prop.~7.29]{Bellaiche1}}]
\label{le:error}
Consider a point $\aa \in \Omega$, a system of privileged
coordinates $z$ at $\aa$, and an approximation $\hX$ of
$X$ at $\aa$. Then, there exist positive constants $C_e(\aa)$ and
$\eps_e (\aa)$ such that, for every $x \in \Omega$ with
$d(\aa,x)<\eps_e (\aa)$ and every integrable input function
$u(\cdot)$ with $\ell (u) <\eps_e (\aa)$, one has
\begin{equation}
\| z (x(T,x,u)) - z (\hat x (T,x,u)) \|_{\aa} \leq C_e(\aa) \max \big(
\|z(x)\|_{\aa}, \ell (u) \big) \ \ell (u)^{1/r},
\label{eq:error}
\end{equation}
where $r$ is the degree of nonholonomy at $\aa$, $x(\cdot,x,u)$ and $\hat
x (\cdot,x,u)$ are respectively the trajectories of $\dot{x} = \sum_{i=1}^m u_iX_i(x),$
and $\dot{x} = \sum_{i=1}^m u_i\hX_i(x)$.
\end{prop}

\subsection{Approximate steering method} \label{def-ASM}

\begin{de}[\emph{Nonholonomic first order approximation on $\Omega$}]\label{approx}
A {\em nonholonomic first order approximation of $X$ on $\Omega$} is a mapping $\mathcal{A}$ which associates, with every $\aa\in\Omega$, a nonholonomic first order approximation of $X$ at $\aa$ defined on $B(\aa,\rho_{\aa})$, i.e., $\mathcal{A}(\aa):=\widehat{X}^{\aa}$ on $B(\aa,\rho_{\aa})$. The {\it approximation radius function} of $\mathcal{A}$ is the function $\rho:\Omega\rightarrow (0,\infty)$ which associates, with every $\aa$, its approximate radius $\rho_{\aa}$, i.e., $\rho(\aa):=\rho_{\aa}$.
\end{de}

\noindent In the sequel, {\em nonholonomic first-order approximations} will simply be called {\em approximations}. Useful properties of
approximations are \emph{continuity} and \emph{nilpotency}.

\begin{de}[\emph{Continuity and nilpotency of an approximation}]\label{conti-nilpo}
Let $\mathcal{A} : \aa \mapsto \widehat{X}^{\aa}$ be an approximation on $\Omega$.

\begin{itemize}

\item We say that $\mathcal{A}$ is {\em continuous} if
\begin{itemize}
\item[(i)] the mapping
$
(\aa,x) \mapsto \mathcal{A}(\aa)(x)
$
is well-defined and, for every $\aa\in\Omega$, is continuous on a neighborhood of $(\aa,\aa)\in\Omega \times \Omega$;
\item[(ii)]the approximation radius function $\rho$ of $A$
is continuous.
\end{itemize}

\item We say that $\mathcal{A}$ is {\em nilpotent
of step $s \in \mathbb{N}$} if, for every $\aa \in \Omega$, the Lie
  algebra generated
by $\widehat{X}^{\aa}$ is nilpotent of step $s$, i.e. every Lie bracket of length larger than $s$ is equal to zero. (For a definition of the length of a Lie bracket, see Definitions \ref{de:length-bracket} and \ref{de:length-bracket2}.)

\end{itemize}
\end{de}%

Consider a $m$-tuple of vector fields $X=\{X_1,\dots,X_m\}$ in $VF^m
(\Omega)$.

\begin{de}[\emph{Steering law of an approximation}]\label{de:steering-law}
 Let $\mathcal{A} :\aa \mapsto \hX^{\aa}$ be an approximation of $X$ on $\Omega$ and $\rho$
its approximation radius function. A {\em steering law}\/ of
$\mathcal{A}$ is a mapping
which, to every pair $(x,\aa) \in \Omega\times\Omega$ verifying $d(x,\aa) <
\rho(\aa)$, associates an integrable input function $\hat{u}:[0,t]\mapsto\mathbb{R}^m$,
henceforth called a {\em steering control}, such that
the trajectory $\hat x (\cdot, x, \hat u)$ of the {\em approximate control system}
\begin{equation}\label{ACS-de}
  \dot{x}=\sum_{i=1}^m u_i\hX_i^{\aa}(x),
\end{equation}
 is defined on $[0,T]$
and satisfies $\hat x (T, x, \hat u)=\aa$. In other words, $\hat
u(\cdot)$ steers (\ref{ACS-de}) from $x$ to $\aa$.
\end{de}%

A steering law of an approximation is intended to be used as an
approximate steering law for the original system. For that purpose, it
is important to have a continuity property of the steering control:
the closest are $x$ and $\aa$, the smaller is the length of
$\hat{u}$. We introduce the stronger notion of \emph{sub-optimality}
(which is a sort of Lipschitz continuity of the steering law).

\begin{de}[\emph{Sub-optimal steering law}]\label{de:sub-opt}
Let $\mathcal{A}$ be an approximation of $X$ on
$\Omega$ and, for every $\aa \in \Omega$,  let $\hd_{\aa}$ be the
sub-Riemannian distance associated to $\mathcal{A}( \aa)$. We say that a
steering law of $\mathcal{A}$ is {\em sub-optimal} if there exists a
constant $C_{\ell} >0$ and a continuous positive function
$\eps_{\ell}(\cdot)$ such that, for any $\aa,x \in \Omega$ with
$d(\aa,x) < \eps_{\ell}(\aa)$, the control $\hat u(\cdot)$ steering
(\ref{ACS-de}) from $x$ to $\aa$ satisfies: $
\ell (\hat u) \leq C_{\ell} \, \hd_{\aa} (x,\aa) = C_{\ell} \, \hd_{\aa}(
\hat x(0, x, \hat u), \hat
x(T, x, \hat u)).
$
\end{de}%
Due to the definition of the sub-Riemannian distance
$\hd_{\aa}$, sub-optimal steering laws always exist. Given an approximation $\mathcal{A}$ of $X$ and a steering law for
$\mathcal{A}$, we define a {\em local approximate steering} method for
$X$ as follows.

\begin{de}[\emph{Local approximate steering}]\label{de:ap}
The {\em local approximate steering}  (LAS for short)
method associated to $\mathcal{A}$ and its steering law is the mapping $\ap(\cdot,\cdot)$
which associates, with every pair $(x,\aa) \in \Omega\times\Omega$ verifying $d(x,\aa) <
\rho(\aa)$, the point $x(T,x,\hat{u})$, i.e.,
\[
\ap (x,\aa):= x(T, x,{\hu}),
\]
where $\hu (\cdot)$ is the steering control of
$\mathcal{A}(\aa)$ associated to $(x,\aa)$ and $\rho$ is
the approximation radius function of $\mathcal{A}$.
\end{de}%

\begin{de}[\emph{Local contractions and uniform local contractions}]
A LAS method is {\em locally contractive} if, for every $\aa \in \Omega$,
there exist $\varepsilon_{\aa}>0$ and $c_\aa <1$ such that one has:
$$
d(\aa,x) < \varepsilon_{\aa}\ \Longrightarrow \ d(\aa,\ap(x,\aa))
\leq c_\aa d(\aa,x).
$$
A LAS method is {\em uniformly locally contractive} on a compact set $K
\subset \Omega$ if it is locally contractive, and if $\varepsilon_{\aa}$ and
$c_\aa$ are independent of $\aa$, i.e., there exists
$\varepsilon_K>0$ and $c_K <1$ such that, for every pair
$(\aa,x)\in K\times K$, the following implication holds true:
$$
d(\aa,x)<\varepsilon_K\ \Longrightarrow\ d(\aa,\ap(x,\aa)) \leq
c_K d(\aa,x).
$$
\label{Def:LinComLAS}
\end{de}

\begin{rem}
We will show that if $\widehat{X}$ is an approximation of $X$ at
$\aa$, the corresponding $\ap$ function is locally contractive in a
neighborhood of $\aa$. By the Fixed Point Theorem, one gets local
convergence of Algorithm 1 (LAS).  However, in order to obtain a
globally convergent algorithm from LAS, one needs $\ap$ to be {\em
  uniformly} locally contractive. In other words, the mapping
$\mathcal{A}$ needs to be {\it continuous} in the sense of Definition
\ref{conti-nilpo}.
\end{rem}

As a direct consequence of Proposition \ref{le:error}, we obtain sufficient
conditions for a LAS method to be uniformly locally contractive.

\begin{coro}\label{coro:unif_contr}
Let $K$ be a compact subset of $\Omega$. Assume that:
\begin{itemize}
\item[$(i)$] all points in $K$ are regular;
\item[$(ii)$] there exists a
continuously varying system of privileged coordinates $\Phi$ on $K$;
\item[$(iii)$] there exists a
continuous approximation $\mathcal{A}$ of $X$ on $K$;
\item[$(iv)$]
$\mathcal{A}$ is provided with  a sub-optimal steering law.
\end{itemize}
Then, the LAS
method $\ap$ associated to  $\mathcal{A}$ and its steering law is
uniformly locally contractive. Moreover, up to
reducing the positive constant $\eps_K$ occurring in Corollary~\ref{theo:unif_contr}, one
has, for every pair $(\aa,x)\in K\times K$ verifying $d(\aa,x)<\eps_K$,
\begin{eqnarray}
d(\ap(x,\aa),\aa) &\leq& \frac{1}{2} d(x,\aa),
\label{eq:distcontr}\\[0.1cm]
\label{eq:pseudcontr}
\Vert z(\ap(x,\aa))\Vert_{\aa}&\leq& \frac{1}{2} \, \Vert z(x)\Vert_{\aa}.
\end{eqnarray}
\end{coro}

\begin{proof}[Proof of Corollary \ref{coro:unif_contr}]
Under the hypotheses $(i)-(iv)$, one immediately extends Proposition \ref{le:error} and obtains that there exist continuous
positive functions $C_e(\cdot)$ and $\eps_e (\cdot)$ such that
inequality~(\ref{eq:error}) holds true, with $z=\Phi(\aa,\cdot)$ and
$\hX=\mathcal{A}(\aa)$, for every pair $(x,\aa)\in\Omega\times\Omega$
with $d(x,\aa) < \eps_e (\aa)$ and
every integrable input function $u(\cdot)$ with $\ell (u) <
\eps_e (\aa)$. The remaining argument is standard and one conclude easily.

\end{proof}

\begin{rem}\label{rem-uniform1}
Since the growth vector and the weights do not remain constant in any
open neighborhood of a singular point, privileged coordinates $z$
cannot vary continuously in any open neighborhood of that singular
point. Therefore, around a singular point, the distance estimations provided in
Eqs. (\ref{uniform-estimation}) and (\ref{eq:pseudcontr}) and based on
privileged coordinates do not hold true uniformly. In particular, if
$({\aa}_n)$ is a sequence of regular points converging to a singular
point $\aa$ (this is possible since regular points are dense in
$\Omega$), the sequences $\varepsilon_d(\aa_n)$ and
$\varepsilon_e(\aa_n)$ tend to zero whereas $\varepsilon_d(\aa)$ and
$\varepsilon_e(\aa)$ are not equal to zero.
\end{rem}

\begin{rem}\label{rem-uniform2}
A similar discontinuity issue occurs of course for the approximate
system. Indeed, if $\aa$ is a singular point, the growth vector and
the weights of the associated privileged coordinates  at $\aa$ change
around $\aa$, implying a change of the truncation order in the Taylor
expansion of the vector fields. Therefore, the approximate vector
fields cannot vary continuously in any neighborhood of a singular
point.
\end{rem}



\section{Desingularization by Lifting}\label{Desing}

As it appears in Corollary~\ref{coro:unif_contr}, the absence of
singular points
is one of the key features in order to construct uniformly locally
contractive LAS method.  As a matter of fact, we will show in
Section~\ref{GASA} how to construct a globally convergent motion planning
algorithm  for a regular nonholonomic system (i.e., when all points in
$\Omega$ are regular).

However, in general, nonholonomic systems do have singular
points. For such systems, attempts have been made to construct
specific LAS methods (see~\cite{Vendittelli,Jean2}), but additional
conditions on the structure of the singularities are required. Our
approach here is different: we present in this section a desingularization procedure
of the system, in such a way to replace a MPP for a
non regular system by a MPP for a regular one.

The strategy consists in ``{\em {lifting}}" the vector fields $\{X_1,\dots,X_m\} \in VF^m(\Omega)$ defining the control system to some extended domain $\widetilde{\Omega}:=\Omega\times\R^{\tilde{n}},$ with
$\tn\in\mathbb{N}$ to be defined later. The lifted vector fields  $\{\xi_1,\dots,\xi_m\}\in VF^m(\widetilde{\Omega})$ are constructed so that:
\begin{itemize}
 \item[(i)] for $i=1,\dots,m$, $\xi_i$ has the following form,
$$
\xi_i(x,y)=X_i(x)+\sum_{j=1}^{\tn}b_{ij}(x,y)\partial_{y_j}, \qquad (x,y) \in \Omega\times\R^{\tilde{n}},
$$
where the functions $b_{ij}$, $j=1,\dots,\tilde{n}$, are smooth;
\item[(ii)] the Lie algebra generated by $\{\xi_1,\dots,\xi_m\}$ is free up to step $r$ (see Def.~\ref{free-s} below).
\end{itemize}

Point (ii) guarantees that the nonholonomic system defined by $\{\xi_1,\dots,\xi_m\}$ is regular, since its growth vector is constant
on $\widetilde{\Omega}$. Point (i) guarantees that one obtains
$X_1,\dots, X_m$ by {\em projecting} $\xi_1,\dots,\xi_m$ on
$\R^n$. Indeed, let $\pi$ be the canonical projector from
$\widetilde{\Omega}$ to $\Omega$ defined by $\pi(\widetilde{x})=x$, where
$\widetilde{x}=(x,y)\in\widetilde{\Omega}$. Then, denoting
$d\pi_{\widetilde{x}}$ the differential of $\pi$ at $\widetilde{x}$,
one has
$$
d\pi_{\widetilde{x}} (\xi_i(\widetilde{x}))=X_i(\pi(\widetilde{x})).
$$
As a consequence, the
projection by $\pi$ of a trajectory
$\widetilde{x}(\cdot,\widetilde{x}_0,u)$ of the control system
\begin{equation}
\label{ECSS}
\dot{\tx}=\sum_{i=1}^mu_i\xi_i(\tx), \tx\in\widetilde{\Omega},
\end{equation}
is a trajectory  of (\ref{CS}) associated to the \emph{same} input, i.e.,
$
\pi \big(\widetilde{x}(\cdot,\widetilde{x}_0,u)\big) =
x(\cdot,\pi(\widetilde{x}_0),u).
$

Therefore, any control $u$ steering System~(\ref{ECSS}) from a point
$\widetilde{x}_0:=(x_0,0)$ to a point $\widetilde{x}_1:=(x_1,0)$ also
steers System~(\ref{CS}) from $x_0$ to $x_1$. It then suffices to
solve the MPP for the regular System~(\ref{ECSS}).

Note that distinguished desingularization procedures already exist, cf. \cite{rot76,Bellaiche1,Jea:01b}.
However, an important property of the desingularization procedure presented here is that all
the changes of
coordinates and intermediate constructions involved in it
are explicit and purely algebraic. Note also that, during the lifting
process, we obtain, as a byproduct, a nonholonomic first
order approximation of $\{\xi_1,\dots,\xi_m\}$  in a ``canonical" form, which
can be exactly controlled by sinusoids (see Section
\ref{control-nilpotent}). 

We start this section by presenting some general facts on free Lie
algebras, namely the {\em P. Hall basis} in Section \ref{Sec-Hall},
and the {\em canonical form} of a nilpotent Lie algebra of step $r$ in
Section \ref{Sec-cano}. We then give one {\em desingularization
  procedure} in Section \ref{Sec-lifting}. The proofs of the
results stated in Section \ref{Sec-lifting} will be gathered in
Section \ref{Sec-proof}.

\subsection{P. Hall basis on a free Lie algebra and evaluation map}
\label{Sec-Hall}
In this section, we present some general facts on free Lie
algebras. The reader is referred to \cite{Bourbaki} for more
details. Consider $\ci:=\{1,\dots,m\}$, and the free Lie algebra
$\cl(\ci)$ generated by the elements of $\ci$. Recall that $\cl(\ci)$
is the $\mathbb{R}$-vector space generated by the elements of $\ci$
and their formal brackets $[~,~]$, together with the relations of
skew-symmetry and the Jacobi identity enforced (see \cite{Bourbaki}
for more details). We note that, by construction, for every $I\in \cl(\ci)$, there exists $(I_1,I_2)\in \cl(\ci)\times \cl(\ci)$ such that $I=[I_1,I_2]$.

\begin{de}[Length of the elements of $\cl(\ci)$] \label{de:length-bracket}
The {\em length} of an element $I$ of a free Lie
algebra $\cl(\ci)$, denoted by $\vert I\vert$, is defined inductively by
\begin{eqnarray}
\vert I\vert&:=&1, \textrm{ for }
I=1,\dots,m;\label{length1}\\ \vert I\vert&:=&\vert I_1\vert+\vert I_2\vert,
\textrm{ for } I=[I_1,I_2], \textrm{ with } I_1, I_2\in L(X).\label{length2}
\end{eqnarray}
\end{de}
We use $\cl^s(\ci)$
to denote the subspace generated by elements of $\cl(\ci)$ of length
not greater than $s$. Let $\tn_s$ be the dimension of
$\cl^s(\ci)$.

The {\em P. Hall basis} of $\cl(\ci)$ is a totally ordered subset of $\cl(\ci)$ defined as follows.
\begin{de}[\emph{P. Hall basis}]A subset $\ch=\{I_j\}_{j\in\mathbb{N}}$ of $\cl(\ci)$ is the {\em P. Hall basis} of $\cl(\ci)$ if (H1), (H2), (H3), and (H4) are verified.
\begin{itemize}
\item[(H$1$)]If $\vert I_i\vert<\vert I_j\vert$, then $I_i\prec I_j$;
\item[(H$2$)]$\{1,\dots,m\}\subset\ch$, and we impose that $1\prec 2\prec\dots\prec m$;
\item[(H$3$)]every element of length $2$ in $\ch$ is in the form $[I_i,I_j]$ with $(I_i, I_j)\in\ci\times\ci$ and $I_i\prec I_j$;
\item[(H$4$)]an element $I_k\in\cl(\ci)$ of length greater than $3$ belongs to $\ch$ if $I_k=[I_{k_1},[I_{k_2},I_{k_3}]]$ with $I_{k_1},I_{k_2},I_{k_3}, \textrm{ and }[I_{k_2},I_{k_3}]$ belonging to $\ch$, $I_{k_2}\prec I_{k_3}$, $I_{k_2}\prec I_{k_1}$ or $I_{k_2}=I_{k_1}$, and $I_{k_1}\prec[I_{k_2},I_{k_3}]$.
\end{itemize}
\end{de}

\noindent The elements of $\ch$ form a basis of $\cl(\ci)$, and
$``\prec"$ defines a strict and total order over the set $\ch$. In the
sequel, we use $I_k$ to denote the $k^{\textrm{th}}$ element of $\ch$
with respect to the order $``\prec"$. Let $\ch^s$ be the subset of
$\ch$ of all the elements of length not greater than $s$. The elements
of $\ch^s$ form a basis of $\cl^s(\ci)$ and
$\textrm{Card}(\ch^s)=\tn_s$. The set $\cg^s:=\ch^s\setminus\ch^{s-1}$ contains
the elements in $\ch$ of length equal to $s$. Its cardinal will be
denoted by $\tk_s=\textrm{Card}(\cg^s)$.  

By (H1)$-$(H4), every element $I_j\in\mathcal{H}$ can be expanded in a unique way as
\begin{equation}\label{Hall_expansion}
I_j=[I_{k_1},[I_{k_2},\cdots,[I_{k_{i}},I_{k}]\cdots]],
\end{equation}
with $k_1\geq \cdots\geq k_{i}$, $k_{i}<k$, and $k\in\{1,\dots,\tilde{n}_1\}$. In that case, the element $I_j$ is said to be a {\em direct descendent} of $I_k$, and we write $\phi(j):=k$. For $I_j\in\ch^r$, the expansion $(\ref{Hall_expansion})$ also associates with $I_j\in\ch$ a sequence $\alpha_j=(\alpha_j^1,\dots,\alpha_j^{\tn_r})$ in $\mathbb{Z}^{\tn_r}$ defined by $$\alpha_j^{\ell}:=\textrm{ Card }\{s\in\{1,\dots,i\}, k_s=\ell\}.$$ By construction, one has $\alpha_j^{\ell}=0$ for $\ell\geq j$, and $\alpha_j=(0,\dots,0)$  for $1\leq j\leq\tilde{n}_1$.  

Consider now a family of $m$ vector fields $X=\{X_1,\dots, X_m\}$ and
the Lie algebra $L(X)$ they generate.
The P. Hall basis $\ch$ induces, via the {\em evaluation map}, a
family of vector fields spanning  $L(X)$ as a linear space.

\begin{de}[\emph{Evaluation map}]\label{de:eval-map}
The {\em evaluation map} $E_X$ defined on $\cl(\ci)$, with values in $L(X)$, assigns to every $I\in\cl(\ci)$ the vector field $X_I=E_X(I)$ obtained by plugging in $X_i$, $i=1,\dots,m$, for the corresponding letter $i$.
\end{de}

\begin{de}[\emph{Length of the elements of $L(X)$}]\label{de:length-bracket2}
With the notations of Definition \ref{de:eval-map}, if $X_I=E_X(I)$, the length of $X_I$, denoted by $\Delta(X_I)$, is set to be equal to $\vert I\vert$, .
\end{de}

\begin{de}[\emph{P. Hall family}]
The {\em P. Hall family} $H_X$ associated with the vector fields $X=\{X_1,\dots, X_m\}$ is defined by
$
H_X:=\{E_X(I), I\in\mathcal{H}\},
$
where $E_X$ is the evaluation map and $\mathcal{H}$ is the P. Hall
basis of the free Lie algebra $\cl(\ci)$. The family $H_X$ inherits
the ordering and the numbering of the elements in $\ch$ induced by
{(H1)--(H4)}.
\end{de}

\noindent Note that $H_X$ is a spanning set of $L(X)$, but not always
a basis.

\begin{de}[\emph{Free up to step $s$}]\label{free-s}
Let $s$ be a positive integer such that $1\leq s\leq r$. A family of vector fields $\xi=\{\xi_1,\dots,\xi_m\}$ defined on a subset $\widetilde{\Omega}$ of $\R^{\tn_r}$ is said to be {\em free up to step $s$} if, for every $\tx\in\widetilde{\Omega}$, the growth vector $(n_1(\tx),\dots,n_s(\tx))$ is equal to $(\tn_1,\dots,\tn_s)$.
\end{de}
\begin{rem}
If $\xi$ defined on $\widetilde{\Omega}\subset\R^{\tn_r}$ is free up to step $r$, then every point of $\widetilde{\Omega}$ is {\em regular}.
\end{rem}

\begin{de}[\emph{Free weights}]\label{de:free-weight}
Let $\xi=\{\xi_1,\dots,\xi_m\}$ be free up to step $r$ on $\widetilde{\Omega}\subset\mathbb{R}^{\tn_r}$. The integers $\widetilde{w}_1,\dots,\widetilde{w}_{\tn_r}$, where $\widetilde{w}_j=s$ if $n_{s-1}(\tx)<j\leq n_s(\tx)$ for every $\tx\in\widetilde{\Omega}$, are called the {\em free weights} of step $r$.
\end{de}

\subsection{Canonical form}\label{Sec-cano}
We present in this section the construction of a canonical form
for nilpotent systems proposed by Grayson and Grossman
in~\cite{Grayson1} and~\cite{Grayson2}. Similar results were also
obtained by Sussmann in~\cite{Sussmann1}.

The construction takes place in $\R^{\tn_r}$, where $r$ is a positive
integer and $\tn_r$ the dimension of $\cl^r(\ci)$. We denote by
$v=(v_1,\dots,v_{\tn_r})$  the points in $\R^{\tn_r}$, and by
$(\partial_{v_1}, \dots,\partial_{v_{\tn_r}})$ the
canonical basis of $\R^{\tn_r}$.  For $j=1,\dots,\tn_r$, we assign to
the coordinate function $v_j$ the weight $\widetilde{w}_j$ and to the vector
$\partial_{v_j}$ the weight $-\widetilde{w}_j$.  The {\em weighted degree}
of a monomial  $v_1^{\alpha_1} \cdots v_{\tn_r}^{\alpha_{\tn_r}}$
is then defined as $
\widetilde{w}(\alpha):=
\widetilde{w}_1\alpha_1+\cdots+\widetilde{w}_{\tn_r}\alpha_{\tn_r},
\qquad \hbox{where }  \alpha=(\alpha_1,\dots,\alpha_{\tn_r}),
$
and the weighted degree of a monomial vector field
$v_1^{\alpha_1}\dots v_{\tn_r}^{\alpha_{\tn_r}} \partial_{v_j}$ is
defined as  $\widetilde{w}(\alpha)-\widetilde{w}_j$.

For every $I_j\in\ch^r$, let $\alpha_j$ be the sequence associated
with $I_j$ (see Section \ref{Sec-Hall}). Define the monomial $P_{j}(v)$ associated with $I_j$ by
\begin{equation}
P_{j}(v):=\frac{v^{\alpha_j}}{\alpha_j!},
\end{equation}
where $\displaystyle
v^{\alpha_j}:=\prod_{\ell}v_{\ell}^{\alpha_j^{\ell}}$, and
$\displaystyle\alpha_j ! :=\prod_{\ell}\alpha_j^{\ell}!$.
Note that  $P_{j}$  satisfies the following inductive formulas.
\begin{equation}
\label{induction-P}
\begin{array}{llll}
P_{j}(v)&=&1 &\textrm{ if }I_j\in\ch^1;\\
P_{j}(v)&=&\displaystyle\frac{v_{j_1}}{\alpha_{j_2}^{j_1}+1}P_{j_2}(v)&\textrm{ if }I_j=[I_{j_1},I_{j_2}].
\end{array}
\end{equation}

\begin{theo}[\cite{Grayson1,Grayson2}]
\label{th-grossman}
We define the vector fields $D_1,\dots, D_m$
on $\R^{\tn_r}$ as follows:
\begin{eqnarray*}
D_1&=&{\partial_{v_1}},\\
D_2&=&{\partial_{v_2}}+\sum_{\substack{2\leq\vert I_j\vert\leq
    r\\ \phi(j)=2}}P_{j}(v){\partial_{v_j}},\\
&\vdots&\\
D_m&=&{\partial_{v_m}}+\sum_{\substack{2\leq\vert I_j\vert\leq
    r\\ \phi(j)=m}}P_{j}(v){\partial_{v_j}}.\\
\end{eqnarray*}
Then, the Lie algebra $L(D)$ generated by $(D_1,\dots,D_m)$ is free up
to
step $r$, and one has
$$
D_{I_{j}}(0)={\partial_{v_j}},\ \textrm{ for }I_j\in\ch^r,
$$
where $D_{I_j}:=E_D(I_j)$ is defined through the
evaluation map $E_D$  with values in $L(D)$.
\end{theo}
\noindent The proof of Theorem \ref{th-grossman} goes by induction on
the length of elements in $L(D)$. The reader is referred to
\cite{Grayson2} for a complete development.

\begin{coro}\label{coro1-grossman}
For all $I_k\in\ch^r$, $D_{I_k}$ has the following form
\begin{equation}\label{eq:coro-grossman}
D_{I_k}=\partial_{v_k}+\sum_{{I_j\in\ch^r,\  \vert I_j\vert>\vert
    I_k\vert}}P_j^{k}(v)\partial_{v_j},
\end{equation}
where every non zero polynomial $P_j^{k}$ is homogeneous
of weighted degree equal to $\vert I_j\vert-\vert I_k\vert$.
\end{coro}

\begin{rem}
The explicit expression of the polynomials $P_j^k(v)$ as functions of
the monomials $P_{j}(v)$ is obtained through an induction formula.
\end{rem}

\begin{coro}\label{coro2-grossman}
For $i=1,\dots, m$, we define $m$ vector fields $\check{D}_i$ as follows:
\begin{equation*}
\check{D}_i:=\partial_{v_i}+\sum_{\substack{2\leq\vert
    I_k\vert\leq\ch^{r-1}\\\phi(k)=i}}P_{k}(v)\partial_{v_k}+\sum_{\substack{I_j\in
    S\\\phi(j)=i}}P_{j}(v)\partial_{v_j},
\end{equation*}
where $S$ is an arbitrary non-empty subset of $\cg^{r}$. Then,
\begin{itemize}
\item if $I_k\in\ch^{r-1}\cup S$, we have
$
\check{D}_{I_k}=\partial_{v_k}+\sum_{{I_j\in\ch^{r-1}\cup S,\  \vert
    I_j\vert>\vert I_k\vert}}P_j^{k}(v)\partial_{v_j};
$
\item if $I_k\in\cg^{r}\setminus S$, we have $\check{D}_{I_k}=0$.
\end{itemize}
\end{coro}

\begin{de}[\emph{Canonical form}]
Let $X_1,\dots,X_m$ be $m$ vector fields on an open subset $\Omega$ of
$\R^{\tn_r}$ and $v$ a local system of coordinates on $\Omega$.
The control system associated to
$\{X_1,\dots,X_m\}$ is said to be in {\em canonical form in the
  coordinates $v$} if one has
$
v_*X_i=D_i, \quad \hbox{for } i=1,\dots,m,
$
where we use $v_*X_i$ to denote the {push-forward} of $X_i$ by $v$.
\end{de}

\noindent Consider now the control system given by
\begin{equation}\label{cano-f1}
\dot{v}=\sum_{i=1}^{m}u_iD_i(v), \quad v\in\R^{\tn_r}.
\end{equation}Writing $(\ref{cano-f1})$ component by component, one has, for $j=1,\dots,\tn_r$,
\begin{equation}\label{cano-f2}
\dot{v}_j=P_{j}(v_1,\dots,v_{j-1})u_i,\ \textrm{ where }i=\phi(j),
\end{equation}
or inductively,
\begin{equation}\label{cano-f3}
\dot{v}_j=\frac{v_{j_1}}{\alpha_{j_2}^{j_1}+1}\dot{v}_{j_2},\ \ \textrm{ where }I_j=[I_{j_1},I_{j_2}].
\end{equation}
More explicitly, one has
\begin{equation}\label{cano-f4}
\dot{v}_j=\frac{1}{k!}v_{j_1}^k\dot{v}_{j_2},\ \textrm{ if }X_{I_j}=\textrm{ad}^k_{X_{I_{j_1}}}X_{I_{j_2}},
\end{equation}where ad$^k_{X_{I_{j_1}}}X_{I_{j_2}}:=[\underbrace{ X_{I_{j_1}},[ X_{I_{j_1}},\cdots, [X_{I_{j_1}}}_{k\textrm{ times }},X_{I_{j_2}}]$, with $X_{I_{j_2}}=[X_{I_{j_3}},X_{I_{j_4}}]$ and ${I_{j_3}}\prec{I_{j_1}}$. The inductive formula (\ref{cano-f4}) will be used in Chapter \ref{control-nilpotent}. 

A particular system of coordinates, called {\em canonical coordinates} (a terminology arising from Lie group theory), allows one to obtain canonical forms. Consider $m$ vector fields $X_1,\dots,X_m $ on $\Omega \subset \R^{\tn_r}$, let $v \in \Omega$, and $W=\{W_1,\dots,W_n\}$ be a set of vector fields in $L(X)$ such that $W_1(v),\dots,W_n(v)$ is a basis of $\R^{\tn_r}$. The {\em canonical coordinates of the second kind} at $v$ associated with $W$ are the system of local coordinates at $v$ defined as the inverse of the local diffeomorphism
\begin{equation}\label{exp-coordinates}
(z_1,\dots, z_{\tn_r})\longmapsto e^{z_{\tn_r}W_{\tn_r}}\circ\cdots\circ e^{z_1W_1}(v),
\end{equation}
where we use $e^{zW_i}$ to denote the flow of $W_i$. When the system $(X_1,\dots, X_m)$ is nilpotent, the above diffeomorphism defines global coordinates on $\Omega$ for every $v\in \Omega$.

\begin{theo}[\cite{Sussmann1}]\label{theo-sussmann}
Assume that the vector fields $(X_1,\dots, X_m)$ generate a Lie algebra which is both nilpotent of step $r$ and free up to step $r$. Then, in the canonical coordinates of the second kind  associated with the P. Hall basis $H^r_X$, the control system defined by $(X_1,\dots, X_m)$ is in canonical form.
\end{theo}

\begin{rem}
The canonical coordinates of the second kind require to determine the flow of the control vector fields i.e., to integrate some differential equations. In general, there does not exist algebraic change of coordinates between an arbitrary system of coordinates and the canonical coordinates of the second kind.
\end{rem}

\subsection{Desingularization algorithm}\label{Sec-lifting}
Let $X=\{X_1,\dots,X_m\}\subset VF(\Omega)$ be a family of $m$ vector
fields on $\Omega\subset\mathbb{R}^n$, and $K$ be a compact subset of
$\Omega$. We assume that the LARC is satisfied at every point of $K$. Therefore, the degree of nonholonomy of $X$ is \emph{bounded} on $K$ and we
denote by $r$ its maximal value.

Recall that $\mathcal{H}^r$ denote the elements of the P. Hall basis
of length smaller or equal to $r$. For every $n$-tuple
$\cj=(I_1,\dots,I_n)$ of elements of $\ch^r$, we define the domain
$\cv_{\cj}\subset\Omega$ by
\begin{equation}\label{VJ}
\cv_{\cj}:=\{\ p\in\Omega\textrm{ such that }\det
(X_{I_1}(p),\dots,X_{I_n}(p))\neq 0\  \},
\end{equation}
where $X_{I_j}=E_X(I_j)$. Such a set $\cv_{\cj}$ is open in $\Omega$ (possibly empty) and for every $p\in\cv_{\cj}$, the vectors
$X_{I_1}(p),\dots, X_{I_n}(p)$ form a basis of $\R^n$.

Since $K$ is compact, there exist a finite family of $n$-tuples
$\cj_1,\dots, \cj_M$ of elements of $\ch^r$ such that
\begin{equation}\label{K-covering}
K\subset\bigcup_{i=1}^M\cv_{\cj_i}.
\end{equation}One deduces from (\ref{K-covering}) a compact covering of $K$ in the form
\begin{equation}\label{K-covering-c}
K\subset\bigcup_{i=1}^M\cv_{\cj_i}^c,
\end{equation}where, for $i=1,\dots,M$, the set $\cv_{\cj_i}^c\subset \cv_{\cj_i}$ is compact.

\begin{de}\label{de:graph}
Let $(S_i)_{i\in I}$ be a finite set of subsets of $\Omega$. The \emph{connectedness graph} $\textsf{G}:=(\textsf{N},\textsf{E})$ associated with $(S_i)_{i\in I}$ is defined as
as follows:
\begin{itemize}
\item[-] the set of nodes $\textsf{N}:=I$;  
\item[-] a pair $(i,j)$ with $i$ and $j$ in $\textsf{N}$ belongs to the set of edges $\textsf{E}$ if $S_i\cap S_j\neq\emptyset$.
\end{itemize}
A \emph{simple path} on $\textsf{G}$ is a subset $\textsf{p}:=\{i_1,\dots,i_L\}$ of two by two distinct elements of $\textsf{N}$ such that, for $j=1,\dots,L-1$, the pair $(i_j,i_{j+1})$ belongs to $\textsf{E}$.
\end{de}

\begin{rem}
With the notations of Definition \ref{de:graph}, if we assume that all the sets $S_i$ are open or, all of them are closed, and the set $\displaystyle S:=\cup_{i\in I} S_i$ is connected, then, for every $(x_0,x_1)\in S\times S$, there exists a simple path on $\textsf{G}$ denoted by $p:=\{i_1,\dots, i_L\}$ such that $x_0\in S_{i_1}$ and $x_1\in S_{i_L}$.
\end{rem}

Take $\cj=(I_1,\dots,I_n)$ among $\cj_1,\dots,\cj_M$, and pick a point $a$ in $\cv_{\cj}$. In the sequel, we construct, by induction on the length of elements in a free Lie algebra, a family of $m$ vector fields $\xi=\{\xi_1,\dots, \xi_m\}$ defined on $\cv_{\cj}\times\R^{\tilde{n}_r-n}$, which is free up to step $r$ and has its projection on $\cv_{\cj}$ equal to $X$. At the same time, we give an approximation of $\xi$ at $\tilde{a}:=(a,0)\in\cv_{\cj}\times\R^{\tn_r-n}$  in canonical form.  

We define $\cj^s:=\{I_j\in\cj, \textrm{ with }\vert I_j\vert=s\}, \textrm{ for }s\geq 1$, and $\cg^s:=\ch^s\setminus\ch^{s-1}, \textrm{ for }s\geq 2$. We denote by $k_s$ the cardinal of $\cj^s$, and by $\tilde{k}_s$ the cardinal of $\cg^s$. We are now ready to describe in details our desingularization algorithm. \bigskip

\noindent \fbox{\bf Desingularization Algorithm (DA)}

\begin{itemize}
\item {\bf Step 1:}
\begin{enumerate}
\item[(1-1)] Set
$
\cv^1:=\cv_{\cj}\times\R^{\tilde{k}_1-k_1}, \ a^1:=(a,0)\in\cv^1,\ \ck^1:=\ch^1\cup(\cj\setminus\cj^1).
$
\item[(1-2)] Define $\{\xi^1_1,\dots, \xi^1_m\}$ on $\cv^1$ as follows:
\begin{equation*}
\forall (x,v^1) \in \cv^1, \qquad \xi_i^1(x,v^1) :=X_i(x) +\left\{\displaystyle\begin{array}{ll}0,&\textrm{ for } i\in \cj^1,\\ {\partial_{v^1_i}},&\textrm{ for }i\in\cg^1\setminus\cj^1. \end{array}\right.
\end{equation*}
\item[(1-3)] Compute the coordinates $y^1$ on $\cv^1$ defined as the unique affine system of coordinates on $\cv^1$ such that
$
{\partial_{y_j^1}}=\xi^1_{I_j}(a^1), \textrm{ for }I_j\in\ck^1, \textrm{ and } y^1(a^1)=0.
$
\item[(1-4)] Construct the system of global coordinates $z^1$ on $\cv^1$ by
\begin{eqnarray*}
z^1_{j}&:=&y^1_{j}, \ \ \textrm{ for }j\in\ch^1,\\
z^1_{j}&:=& y^1_{j}-\sum_{k=1}^{\tilde{n}_1}(\xi^1_k\cdot y^1_k)(y^1)_{|y^1=0} \ \ y^1_k,\ \ \textrm{ for }I_j\in\ck^1\setminus\ch^1,
\end{eqnarray*}where $I_j$ denotes the $j^{\textrm{th}}$ element in $\ck^1$.
\end{enumerate}

\item {\bf Step s, 2 $\leq$ s $\leq$ r:}
\begin{enumerate}
\item[(s-1)] Set
$
\cv^s:=\cv^{s-1}\times\R^{\tilde{k}_s-k_s}, \ a^s:=(a,0)\in\cv^s,
$
and
$
\ck^s:=\ck^{s-1}\cup(\cg^s\setminus\cj^s).
$
Denote by $v^s$ the points in $\R^{\tilde{k}_s-k_s}$.
\item[(s-2)] Define $\{\xi^s_1,\dots, \xi^s_m\}$ as the vector fields on $\cv^s$ which write in coordinates $(z^{s-1},v^s)$ as:
\begin{equation*}
\xi_i^s(z^{s-1},v^s)=\xi_i^{s-1}(z^{s-1}) +\sum_{\substack{I_k \in \cg^s \setminus\cj^s \\ \phi(k)=i}} P_{k}(z^{s-1}){\partial_{v_{k}^s}}\ .
\end{equation*}
\item[(s-3)] Compute the system of global coordinates $y^s$ on $\cv^s$ as the unique isomorphism $(z^{s-1},v^s) \mapsto y^s$ such that
$
{\partial _{y_{\phi(I)}^s}}=\xi^s_I(a^s) \hbox{ for every } I\in\ck^s\ .
$
\item[(s-4)] Construct the system of global coordinates $\tilde{z}^s$ on $\cv^s$ by the following recursive formulas:
\begin{itemize}
\item[(s-4)-(a)] for $I_j\in\ch^s$,
\begin{equation}
\tilde{z}^s_{j}:=y^s_{j}+\sum_{k=2}^{\vert I_j\vert-1}r_k(y^s_1,\dots,y^s_{j-1}),
\end{equation}
where, for $k=2,\dots,\vert I_{j}\vert-1$,
\begin{eqnarray*}
&&r_k(y^s_1,\dots,y^s_{j-1})\\
&=&-\sum_{\substack{\vert\beta\vert=k\\ \omega(\beta)<\vert I_j\vert}} \big[(\xi_{I_1}^s)^{\beta_1}\cdots(\xi_{I_{j-1}}^s)^{\beta_{j-1}}\cdot(y^s_{j}+ \sum_{q=2}^{k-1}r_q)\big](y^s)_{|y^s = 0}\ \frac{(y_1^s)^{\beta_1}}{\beta_1!}\\
&&\cdots\frac{(y_{j-1}^s)^{\beta_{j-1}}}{\beta_{j-1}!}\ ;
\end{eqnarray*} 
\item[(s-4)-(b)] for $I_j\in\ck^s\setminus\ch^s$,
\begin{equation}
\tilde{z}^s_{j}:=y^s_{j}+\sum_{k=2}^{s}r_k(y^s_1,\dots,y^s_{\tn_s}),
\end{equation}where, for $k=2,\dots,s$,
\begin{eqnarray*}
r_k(y^s_1,\dots,y^s_{\tn_s})&=&-\sum_{\substack{\vert\beta\vert=k\\ \omega(\beta)\leq s}}\big[(\xi_{I_1}^s)^{\beta_1}\cdots(\xi_{I_{\tn_s}}^s)^{\beta_{\tn_s}}\cdot(y^s_{j}+ \sum_{q=2}^{s}r_q)\big](y^s)_{|y^s = 0} \\
&& \frac{(y_1^s)^{\beta_1}}{\beta_1!}\cdots\frac{(y_{j-1}^s)^{\beta_{\tn_s}}}{\beta_{\tn_s}!}.
\end{eqnarray*} 
\end{itemize}
\item[(s-5)] Construct the system of global coordinates $z^s$ as follows: 
\begin{itemize}
\item[(s-5)-(a)] for $j>\tn_s$, set $z_j^s:=\tilde{z}_j^s$;
\item[(s-5)-(b)] for $j=1,\dots,\tn_s$, set
$
z_j^s:=\Psi_j^s(\tilde{z}_1^s,\dots,\tilde{z}_{j-1}^s),
$
where all $\Psi_j^s$ are polynomials such that the two following conditions are satisfied:
\begin{itemize}
  \item[-] if we impose the weight of $z_j^s$ to be $\widetilde{w}_j$ for $j=1,\dots,\tn_s$, then every $\Psi_j^s$ is homogeneous of weighted degree equal to $\widetilde{w}_j$;
  \item[-] denote by ord$_{a^s}^s(\cdot)$ the nonholonomic order defined by $(\xi_1^s,\dots,\xi_m^s)$ at $a^s$, and by $\xi_{i,j}^{s}(z^s)$ the $j^{\textrm{th}}$ component of $\xi_{i}^{s}(z^s)$; then one has
 \begin{equation}\label{change-coordinates}
 \xi_{i,j}^{s}(z^s)= \delta_{i,\phi(j)} P_{j} (z^s_1,\dots,z^s_{j-1})+R_{i,j}(z^s), \quad j=1,\dots,\tn_s,
 \end{equation}
where $\textrm{ord}_{a^s}^s(R_{i,j})\geq \widetilde{w}_j$ ($\delta_{i,k}$ denotes the Kronecker symbol). Note that $\textrm{ord}_{a^s}^s(P_{j})=\widetilde{w}_j -1$.
\end{itemize}
 \end{itemize}
\end{enumerate}
\end{itemize}

\begin{theo}\label{theo-desingularization}
Let $\xi_i:=\xi_i^r$, for $i=1,\dots,m$, and $z:=z^r$,  where $\xi_i^r$ and $z^r$ are given by the {desingularization algorithm}. Then,
the family of vector fields $\xi=\{\xi_1,\dots,\xi_m\}$ defined on $\Omega \times\R^{\tn_r-n}$ is free up to step $r$. Moreover, the system of coordinates $z=(z_1,\dots,z_{\tn_r})$ is a system of privileged coordinates at $\tilde{a}$ for $\xi$, and the family of vector fields $\widehat{\xi}=\{\widehat{\xi}_1,\dots,\widehat{\xi}_m\}$ defined in the coordinates $z$ by the canonical form:
\begin{equation}\label{cano-desingularized}
\widehat{\xi}_i=\partial_{z_i}+\sum_{\substack{2\leq\vert I_j\vert\leq \tn_{r}\\ \phi(j)=i}}P_{j}(z_1,\dots,z_{j-1})\partial_{z_j}, \ \textrm{ for }i=1,\dots,m,
\end{equation}
is a nonholonomic first order approximation of $\xi$ at $\tilde{a}$.
\end{theo}

\begin{rem}
We note that the desingularization procedure does not a priori require that
\begin{itemize}
\item[(a)]the coordinates $z$ are privileged coordinates;
\item[(b)]the system $\widehat{\xi}$ is a first order approximation of $\xi$ at $a$.
\end{itemize}However, (a) and (b) can be used directly at the first step of the motion planning algorithm presented in Section \ref{regular-case}.
\end{rem}

\begin{rem}\label{rem:nilpotent}
If we assume that the original system $X=\{X_1,\dots, X_m\}$ is nilpotent, then, by adapting the proof of Theorem \ref{theo-desingularization} presented in Section \ref{Sec-proof}, one can show that the corresponding ``lifted" system $\xi=\{\xi_1,\dots,\xi_m\}$ given by the Desingularization Algorithm proposed in this section remains \emph{nilpotent} with the same order of nilpotency. Moreover, when expressed in the privileged coordinates $z$, the system $\xi$ is equal to its own first order approximation in the canonical form. In other words, for any nilpotent systems of step $k$, the Desingularization Algorithm constructs a nilpotent system of step $k$  and free up to step $k$ which is in the canonical form in coordinates $z$.
\end{rem}

\subsection{Proof of Theorem 3.5}\label{Sec-proof}

The proof of Theorem \ref{theo-desingularization} is based on the following proposition.

\begin{prop}\label{prop-desingularization}
The desingularization algorithm is feasible from $s=1$ to $s=r$. At each step $s$ of the construction $(s=1,\dots,r)$, the following properties hold true:
\begin{enumerate}
\item[{\em (A1)}]\label{linear} the vectors $\{\xi_I^s(a^s) \ : \ I\in\ck^s\}$ are linearly independent;
\item[{\em (A2)}]\label{order-equal} if $\vert I_j\vert\leq s$, then {\em ord}$_{a^s}^s(\tilde{z}^s_j)=\vert I_j\vert$, and {\em ord}$_{a^s}^s(z^s_j)=\vert I_j\vert$;
\item[{\em (A3)}]\label{order-greater} if $\vert I_j\vert> s$, then {\em ord}$_{a^s}^s(z^s_j)>s$;
\item[{\em (A4)}]\label{change-exist} the change of coordinates $(\Psi_j^s)_{j=1,\dots,\tn_s}$ is well defined;
\item[{\em (A5)}]\label{bracket-form} for $I_k\in\ck^{s}$, the vector fields $\xi_{I_k}^s$ has the following form in coordinates $z^s$,
\begin{equation}\label{eq:bracket-form}
\xi_{I_k}^s (z^s)=\sum_{I_j\in\ch^s}(P_j^k(z^s)+R_j^k(z^s))\partial_{z_j^s} +\sum_{I_{\ell}\in\ck^s\setminus\ch^s}Q_{\ell}^k(z^s) \partial_{z_{\ell}^s},
\end{equation}
with {\em ord}$_{a^s}^s(R_j^k)>\vert I_j\vert-\vert I_k\vert$,  {\em ord}$_{a^s}^s(Q_{\ell}^k)>s-\vert I_k\vert$, and $P_j^k$ given by Eq. (\ref{eq:coro-grossman}).

More precisely, if one defines $
\check{\xi}_i^s:=\sum_{\substack{I_j\in\ch^s\\\phi(j)=i}}P_{j}(z^s)\partial_{z_j^s},
$
then, one has
$
\check{\xi}^s_{I_k}=\sum_{I_j\in\ch^s}P_j^k(z^s)\partial_{z_j^s},
$
where the polynomials $P_j^k$ verify the following properties:
\begin{itemize}
\item if $I_k\in\ch^s$, then
\begin{itemize}
\item for $\vert I_j\vert<\vert I_k\vert$, $P_j^k=0$;
\item for $\vert I_j\vert=\vert I_k\vert$, $P_j^j=1$, and $P_j^k=0$ if $k\neq j$;
\item for $\vert I_j\vert>\vert I_k\vert$, {\em ord}$_{a^s}^s(P_j^k)=\vert I_j\vert-\vert I_k\vert$;
\end{itemize}
\item if $I_k\in\ck^s\setminus\ch^s$, $P_j^k=0$ for all $j=1,\dots, \tn_s$.
\end{itemize}
\end{enumerate}
\end{prop}

\begin{rem}\label{rem-desing}
Property (A1) implies that Step (s-3) is feasible, which, in turn, guarantees that Steps (s-4)-(a) and (s-4)-(b) are well defined, and $\tilde{z}^s$ is a system of coordinates because the differential of the application $y^s\mapsto\tilde{z}^s$ at $0$ is equal to the identity map. Property (A4) guarantees that Step (s-5)-(b) is feasible. Property (A2) ensures that, at the end of the algorithm, the system of coordinates $z^r$ is a system of privileged coordinates. Property (A5) finally ensures that for $s=r$, the approximation $\widehat{\xi}$ of $\xi$ is in canonical form.
\end{rem}

\noindent By Remark \ref{rem-desing}, Theorem \ref{theo-desingularization} is a consequence of Proposition \ref{prop-desingularization} by induction on $s$.

\begin{proof}[{Proof of Proposition \ref{prop-desingularization}}]
We begin by showing that Properties (A1)-(A5) hold true for $s=1$.

\begin{claim}\label{linear1}
The family of vectors $\{\xi_I^1(a^1)\}_{I\in\ck^1}$ is linearly independent, i.e., Property {\em (A1)} holds true for $s=1$.
\end{claim}

\begin{proof}[Proof of Claim \ref{linear1}]

By construction, for every $I\in\cj$, one has $\xi_I^1(a^1)=X_I(a)$, which belongs to $\R^{n}\times\{0\}$. For $i\in\cg^1\setminus\cj^1$, the vector $\xi_i^1(a^1)$ belongs to $\R^{n}\times\R^{\tk_1-k_1}$, and the family of vectors $\{\xi_i^1(a^1)\}_{i\in\cg^1\setminus\cj^1}$ is linearly independent. Therefore, the family of vectors $\{\xi_I^1(a^1)\}_{I\in\ck^1}$ is linearly independent and Claim \ref{linear1} holds true.
\end{proof}

\begin{claim}\label{order-1}
For $j=1,\dots,\tn_1$, one has {\em ord}$_{a^1}^1(z^1_j)=1$, i.e., Property {\em (A2)} holds true for $s=1$.
\end{claim}

\begin{proof}[Proof of Claim \ref{order-1}]

For $j=1,\dots,\tn_1$, one has by construction $\xi_j^1\cdot z_j^1(a^1)=1$. Thus, one has $\textrm{ord}_{a^1}^1(z_j^1)\leq 1.$ Since $z^1$ is a system of coordinates centered at $a^1$, one has $z_j^1(a^1)=0$, which implies that $\textrm{ord}_{a^1}^1(z_j^1)>0.$ Therefore, one gets  ord$_{a^1}^1(z_j^1)=1$ and Claim \ref{order-1} holds true.
\end{proof}

\begin{claim}\label{order-2}
For $I_j\in\ck^1$ with $\vert I_j\vert> 1$, one has {\em ord}$_{a^1}^1(z^1_j)>1$, i.e., Property {\em (A3)} holds true for $s=1$.
\end{claim}

\begin{proof}[Proof of Claim \ref{order-2}]

For $\vert I_j\vert\geq 2$, i.e. $I_j\in\ck^1\setminus\cj^1$, one computes $\xi_k^1\cdot z_j^1$ at $a^1$ for every $k\in\{1,\dots,\tn_1\}$.
 \begin{eqnarray*}
 \xi_k^1\cdot z_j^1(a^1)&=&\xi_k^1\cdot y_j^1(a^1)-\sum_{i=1}^{\tn_1}(\xi_i^1\cdot y_j^1)(a^1)(\xi_k^1\cdot y_i^1)(a^1)\\
 &=&\xi_k^1\cdot y_j^1(a^1)-\xi_k^1\cdot y_j^1(a^1)=0.
 \end{eqnarray*}Then, by definition, one has ord$_{a^1}^1(z_j^1)>1$ for $\vert I_j\vert>1$ and Claim \ref{order-2} holds true.
 \end{proof}

\begin{claim}\label{linear-change}
For $i=1,\dots,m$, and $j=1,\dots, \tn_1$, the $j^{\textrm{th}}-$component of $\xi_i^1$ in coordinates $z^1$ is equal to $1$ if $i=j$, and equal to $0$ otherwise. In other words, for $i=1,\dots,m$, the $\tn_1$ first components of $\xi_i^1$ verify Eq. (\ref{change-coordinates}). Properties {\em (A4)} and {\em (A5)} hold true for $s=1$.
\end{claim}

\begin{proof}[Proof of Claim \ref{linear-change}]
By Claim \ref{linear1}, $\xi_1^1(a^1),\dots, \xi_{\tn_1}^1(a^1)$ is a basis of $\R^{\tn_1}$, and thus the linear change of coordinates $y^1$ exists. As $\partial_{y_j^1}=\xi_j^1(a^1)$, and $z_j^1=y_j^1$ for $j=1,\dots, \tn_1$, Claim \ref{linear-change} holds true.
\end{proof}

\noindent Therefore, Properties (A1)-(A5) hold true for $s=1$. Let $1\leq s\leq r$. Let us now assume that Properties (A1)-(A5) hold true for $1\leq s'\leq s$. We will show that they still hold true for $s+1$.

\begin{claim}\label{def-P}
The vector fields $\{\xi_i^{s+1}\}_{i=1,\dots,m}$ are well defined. Moreover, one has {\em ord}$_{a^{s+1}}^s(P_{k})=s$.
\end{claim}

\begin{proof}[Proof of Claim \ref{def-P}]
Consider $I_k\in\cg^{s+1}\setminus\cj^{s+1}$, then one has $I_k=[I_{k_1},I_{k_2}]$. By Eq. (\ref{induction-P}), one has
$$
P_{k}(z^{s})=\frac{z_{k_1}^{s}}{\alpha_{k_2}^{k_1}+1}P_{k_2}(z^{s}).
$$
Since $\vert I_{k_1}\vert \leq s$ and $\vert I_{k_2}\vert \leq s$, we have $k_1\leq \tn_{s}$ and $k_2\leq \tn_{s}$, thus the right-hand side of the above equation is well defined in coordinates $z^{s}=(z_1^{s},\dots, z_{\tn_{s}}^{s})$.   Therefore, the new vector fields $\{\xi_i^{s+1}\}_{i=1,\dots,m}$ are also well defined. Since ord$_{a^{s+1}}^s({z_{k_1}^{s}}P_{k_2})=$ ord$_{a^{s+1}}^s(z_{k_1}^s)+$ord$_{a^{s+1}}^s(P_{k_2})$, and by inductive hypothesis (namely (A2) holds true at step s), one has $\textrm{ord}_{a^{s+1}}^s(z_{k_1}^s)=\vert I_{k_1}\vert, \textrm{ and }\ \textrm{ord}_{a^{s+1}}^s(P_{k_2})=\vert I_{k_2}\vert-1.$ Therefore, one has $
\textrm{ord}_{a^{s+1}}^s(P_{k})=\vert I_{k_1}\vert+\vert I_{k_2}\vert-1=s.
$\end{proof}


\begin{claim}\label{general-form-I}
For $I_k\in\ck^{s+1}$ with $\vert I_k\vert\leq s+1$, one has
\begin{equation}\label{eq:general-form-I}
\xi_{I_k}^{s+1} (z^s,v^{s+1})=\xi_{I_k}^s(z^s)+ \sum_{I_j\in\cg^{s+1}\setminus\cj^{s+1}}\tilde{P}_j^k(z^s)\partial_{v_j^{s+1}},
\end{equation}where $
\tilde{P}_j^k(z^s)=P_j^k(z_1^s,\dots,z_{\tn_s}^s)+\tilde{R}_j^k(z^s),$ with  {\em ord}$_{a^{s+1}}^s(P_j^k)=\vert I_j\vert-\vert I_k\vert$ and
{\em ord}$_{a^{s+1}}^s(\tilde{R}_j^k)>\vert I_j\vert-\vert I_k\vert$.
\end{claim}

\begin{proof}[Proof of Claim \ref{general-form-I}]
The proof goes by induction on the length $\vert I_k\vert$. For $\vert I_k\vert=1$, one has
$$
\xi_k^{s+1}(z^s,v^{s+1})=\xi_k^s(z^s)+\sum_{\substack{I_j\in\cg^{s+1}\setminus\cj^{s+1} \\ \phi(j)=k}}P_{j}(z^s)\partial_{v_j^{s+1}}.
$$
By Claim \ref{def-P}, if $\phi(j)=k$, then ord$_{a^{s+1}}^s(P_{j})=s=\vert I_j\vert-\vert I_k\vert$. Claim \ref{general-form-I} holds true for $\vert I_k\vert=1$.

Assume that Claim \ref{general-form-I} holds true for every $I\in\ck^{s+1}$ of length less than or equal to $s_1$. Consider $I_k\in\ck^{s+1}$ with $\vert I_k\vert=s_1+1$. In coordinates $(z^s,v^{s+1})$, one has
\begin{eqnarray*}
\xi_{I_k}^{s+1}&=&[\xi_{I_{k_1}}^{s+1}, \xi_{I_{k_2}}^{s+1}]=[\xi_{I_{k_1}}^s+ \sum_{I_i\in\cg^{s+1}\setminus\cj^{s+1}}(P_i^{k_1}+\tilde{R}_i^{k_1}) \partial_{v_i^{s+1}},
\xi_{I_{k_2}}^s+\sum_{I_i\in\cg^{s+1} \setminus\cj^{s+1}}(P_i^{k_2}+\tilde{R}_i^{k_2})\partial_{v_i^{s+1}}]\\
&=&[\xi_{I_{k_1}}^{s}, \xi_{I_{k_2}}^{s}]+\sum_{I_i\in\cg^{s+1}\setminus\cj^{s+1}} \{\xi_{I_{k_2}}^s\cdot(P_i^{k_1}+\tilde{R}_i^{k_1})-\xi_{I_{k_1}}^s \cdot(P_i^{k_2}+\tilde{R}_i^{k_2})\} \partial_{v_i^{s+1}}.
\end{eqnarray*}
Since (A5) holds true up to step $s$, one has
\begin{eqnarray*}
&& \xi_{I_{k_1}}^s\cdot(P_i^{k_2}+\tilde{R}_i^{k_2}) = \left[\sum_{I_j\in\ch^s}(P_j^{k_1}+R_j^{k_1})\partial_{z_j^s}+\sum_{I_{\ell}\in\ck^s\setminus\ch^s} Q_{\ell}^{k_1}\partial_{z_{\ell}^s}\right]\cdot(P_i^{k_2}+\tilde{R}_i^{k_2})\\
&=&\sum_{I_j\in\ch^s}(P_j^{k_1}+R_j^{k_1})\partial_{z_j^s}P_i^{k_2}+ \sum_{I_{\ell}\in\ck^s\setminus\ch^s}Q_{\ell}^{k_1}\partial_{z_{\ell}^s}P_i^{k_2}+\sum_{I_j\in\ch^s}(P_j^{k_1}+R_j^{k_1})\partial_{z_j^s}\tilde{R}_i^{k_2}+ \sum_{I_{\ell}\in\ck^s\setminus\ch^s}Q_{\ell}^{k_1}\partial_{z_{\ell}^s}\tilde{R}_i^{k_2}\\
&=&\sum_{I_j\in\ch^s}P_j^{k_1}\partial_{z_j^s}P_i^{k_2}+ \left[\sum_{I_j\in\ch^s}R_j^{k_1}\partial_{z_j^s}P_i^{k_2}+ \sum_{I_j\in\ch^s}(P_j^{k_1}+R_j^{k_1})\partial_{z_j^s}\tilde{R}_i^{k_2} + \sum_{I_{\ell}\in\ck^s\setminus\ch^s}Q_{\ell}^{k_1}\partial_{z_{\ell}^s}\tilde{R}_i^{k_2}\right]\\
&:=&\sum_{I_j\in\ch^s}P_j^{k_1}\partial_{z_j^s}P_i^{k_2}+\mathcal{R}_{i,1}.
\end{eqnarray*}

We first show that every term in $\mathcal{R}_{i,1}$ has, at $a^{s+1}$,  an order strictly greater than $s+1-\vert I_k\vert$.\smallskip

Indeed, for $I_j\in\ch^s$, since $\textrm{ord}_{a^{s+1}}^s(z_j)=\vert I_j\vert,\ \textrm{ord}_{a^{s+1}}^s(P_i^{k_2})=\vert I_i\vert-\vert I_{k_2}\vert,$ $\textrm{ and }\textrm{ord}_{a^{s+1}}^s(R_j^{k_1})>\vert I_j\vert-\vert I_{k_1}\vert,$ one has $\textrm{ord}_{a^{s+1}}^s(R_j^{k_1}\partial_{z_j^s}P_i^{k_2})>\vert I_j\vert-\vert I_{k_1}\vert+(\vert I_i\vert-\vert I_{k_2}\vert)-\vert I_j\vert=\vert I_i\vert-\vert I_k\vert,$ with $\vert I_i\vert=s+1.$ Note that ord$_{a^{s+1}}^s((P_j^{k_1}+R_j^{k_1})\partial_{z_j^s}\tilde{R}_i^{k_2})=$ ord$_{a^{s+1}}^s(P_j^{k_1}\partial_{z_j^s}\tilde{R}_i^{k_2})$. Since $\textrm{ord}_{a^{s+1}}^s(P_j^{k_1})=\vert I_j\vert-\vert I_{k_1}\vert,\ \textrm{ and } \textrm{ord}_{a^{s+1}}^s(\tilde{R}_i^{k_2})>\vert I_i\vert-\vert I_{k_1}\vert,$ then, one has
$\textrm{ord}_{a^{s+1}}^s(P_j^{k_1}\partial_{z_j^s}\tilde{R}_i^{k_2})>\vert I_j\vert-\vert I_{k_2}\vert+\vert I_i\vert-\vert I_{k_1}\vert-\vert I_j\vert=\vert I_i\vert-\vert I_k\vert.$ Recall that, by definition, all the functions have positive order. Therefore,  one gets
$\textrm{ord}_{a^{s+1}}^s(Q_{\ell}^{k_1}\partial_{z_{\ell}^s}\tilde{R}_i^{k_2})\geq \textrm{ord}_{a^{s+1}}^s(Q_{\ell}^{k_1})>s-\vert I_{k_1}\vert\geq s+1-\vert I_k\vert.$ In conclusion, one gets $\textrm{ord}_{a^{s+1}}^s(\mathcal{R}_{i,1})>s+1-\vert I_k\vert.$

A similar computation shows that
\begin{eqnarray*}
&&\xi_{I_{k_2}}^s\cdot(P_i^{k_1}+\tilde{R}_i^{k_1})\\
&=&\sum_{I_j\in\ch^s}P_j^{k_2} \partial_{z_j^s}P_i^{k_1}+\left[\sum_{I_j\in\ch^s}R_j^{k_2} \partial_{z_j^s}P_i^{k_1}+\sum_{I_j\in\ch^s}(P_j^{k_2}+R_j^{k_2}) \partial_{z_j^s}\tilde{R}_i^{k_1}+\sum_{I_{\ell}\in\ck^s\setminus\ch^s} Q_{\ell}^{k_2}\partial_{z_{\ell}^s}\tilde{R}_i^{k_1}\right]\\
&:=&\sum_{I_j\in\ch^s}P_j^{k_2}\partial_{z_j^s}P_i^{k_1}+\mathcal{R}_{i,2}, \textrm{ with ord}_{a^{s+1}}^s(\mathcal{R}_{i,2})>s+1-\vert I_k\vert.
\end{eqnarray*}

\noindent Therefore, one gets
\begin{eqnarray*}
\xi_{I_k}^{s+1}&=&\xi_{I_k}^s+\sum_{I_i\in\cg^{s+1}\setminus \cj^{s+1}}\{\sum_{I_j\in\ch^s}(P_j^{k_1}\partial_{z_j^s}P_i^{k_2}-P_j^{k_2} \partial_{z_j^s}P_i^{k_1})\}\partial_{v_i}+\sum_{I_i\in\cg^{s+1}\setminus\cj^{s+1}}(\mathcal{R}_{i,1}+\mathcal{R}_{i,2})\partial_{v_i},
\end{eqnarray*}with $\textrm{ord}_{a^{s+1}}^s(\mathcal{R}_{i,1}+\mathcal{R}_{i,2})\geq\min\left(\textrm{ord}_{a^{s+1}}^s(\mathcal{R}_{i,1}),\ \textrm{ord}_{a^{s+1}}^s(\mathcal{R}_{i,2})\right)>s+1-\vert I_k\vert.$ 

Since Corollary \ref{coro2-grossman} implies that $\displaystyle\sum_{I_j\in\ch^s}(P_j^{k_1}\partial_{z_j^s}P_i^{k_2}-P_j^{k_2}\partial_{z_j^s}P_i^{k_1})=P_i^k,$ and ord$_{a^{s+1}}^s(P_i^k)=\vert I_i\vert-\vert I_k\vert$, one gets
$$
\xi_{I_k}^{s+1}(z^s,v^{s+1})=\xi_{I_k}^s (z^s)+\sum_{I_i\in\cg^{s+1}\setminus \cj^{s+1}}(P_i^k(z^s)+\tilde{R}_i^k(z^s))\partial_{v_i},
$$
with
$
\textrm{ord}_{a^{s+1}}^s(P_i^k)=s+1-\vert I_k\vert \quad \hbox{and} \quad \textrm{ord}_{a^{s+1}}^s(\tilde{R}_i^k)>s+1-\vert I_k\vert.$ Therefore, Claim \ref{general-form-I} still holds true for $I_k\in\ck^{s+1}$ with $\vert I_k\vert=s_1+1$. This terminates the induction, and Claim \ref{general-form-I} is now proved.
\end{proof}

\begin{claim}\label{linearS}
The family of vectors $\{\xi_{I_k}^{s+1}(a^{s+1})\}_{I_k\in\ck^{s+1}}$ is linearly independent, i.e., Property {\em (A1)} holds true at Step $s+1$.
\end{claim}

\begin{proof}[Proof of Claim \ref{linearS}]
Claim \ref{general-form-I} implies that for every $I_k\in\ck^{s}$, one has $\xi_{I_k}^{s+1}(a^{s+1})=\xi_{I_k}^{s}(a^{s})\in\mathbb{R}^{\tn_s}\times\{0\}.$ Corollary \ref{coro2-grossman} implies that for every $I_k\in\cg^{s+1}\setminus\cj^{s+1}$, one has $$\xi_{I_k}^{s+1}(a^{s+1})=\xi_{I_k}^{s}(a^{s})+\partial_{v_k}\in\mathbb{R}^{\tn_s}\times\mathbb{R}^{\tk_{s+1}-k_{s+1}}.$$ Therefore, by (A1) at step s, the vectors $\{\xi_{I_k}^{s+1}(a^{s+1})\}_{I_k\in\ck^{s+1}}$ are linearly independent.
\end{proof}

\begin{claim}\label{claim-order}
After performing {\em (s+1)-4-(a)} and {\em (s+1)-4-(b)} in the Desingularization Algorithm, one has, for every $I_j\in\ch^{s+1}$, $\textrm{\em ord}_{a^{s+1}}^{s+1}(\tilde{z}^{s+1}_j)=\vert I_j\vert,$ and for every $I_j\in\ck^{s+1}\setminus\ch^{s+1}$, $\textrm{\em ord}_{a^{s+1}}^{s+1}(z_j^{s+1})>s+1$.
\end{claim}

The proof of Claim \ref{claim-order} is based on the following result due to Bella\"iche \cite[Lemma 4.12]{Bellaiche1}.

\begin{lemma}\label{lemma-bellaiche}
Let $\{X_1,\dots,X_m\}$ be a family vector fields defined on
$\Omega$. Consider $\{W_1,\dots,W_n\}$ a frame adapted to the flag
$L^1(\aa)\subset\cdots\subset L^r(\aa)=\R^n$ at $\aa\in\Omega$
(\emph{Remark \ref{adapt_frame}}). A
function $f$ is of order strictly greater than $s$ at $\aa$ is and
only if $(W_1^{\alpha_1}\cdots W_n^{\alpha_n}f)(\aa)=0,$ for all $\alpha=(\alpha_1,\dots,\alpha_n)$ such that $w(\alpha)\leq s$.
\end{lemma} 

\begin{proof}[Proof of Claim \ref{claim-order}]
Claim \ref{linearS} guarantees that $\{\xi_{I_k}^{s+1}\}_{I_k\in\ch^{s+1}}$ is a basis adapted to the flag $L^1(a^{s+1})\subset\cdots\subset L^{s+1}(a^{s+1}).$ Complete $\{\xi_{I_k}^{s+1}\}_{I_k\in\ch^{s+1}}$ by other elements of the Lie algebra generated by $\{\xi_i^{s+1}\}_{i=1,\dots,m}$ in order to get a basis adapted to the flag $L^1(a^{s+1})\subset\cdots \subset L^{s+1}(a^{s+1})\subset\dots\subset L^r(a^{s+1}).$

For $I_j\in\ck^{s+1}\setminus\ch^{s+1}$, Formula (s+1)-4-(b) ensures that $((\xi_{I_1}^{s+1})^{\beta_1}\cdots(\xi_{I_{\tn_{s+1}}}^{s+1})^{\beta_{\tn_{s+1}}}\cdot \tilde{z}_j^{s+1})(a^{s+1})=0,$ for all $\beta=(\beta_1,\dots,\beta_{\tn_{s+1}})$ such that $w(\beta)\leq s+1$. By Lemma \ref{lemma-bellaiche}, one has $\textrm{ord}_{a^{s+1}}^{s+1}(\tilde{z}_j^{s+1})>s+1,\ \textrm{ for }I_j\in\ck^{s+1}\setminus\ch^{s+1}.$

For $I_j\in\ch^{s+1}$, Formula (s+1)-4-(a) implies that $((\xi_{I_1}^{s+1})^{\beta_1}\cdots(\xi_{I_{j-1}}^{s+1})^{\beta_{j-1}}\cdot \tilde{z}_j^{s+1})({a^{s+1}})=0,$ for all $\beta=(\beta_1,\dots,\beta_{j-1})$ such that $w(\beta)\leq \vert I_j\vert-1$. Using again Lemma \ref{lemma-bellaiche}, one has $\textrm{ord}_{a^{s+1}}^{s+1}(\tilde{z}_j^{s+1})>\vert I_j\vert-1,\ \textrm{ for }I_j\in\ch^{s+1}.$ By construction, one already has that ord$_{a^{s+1}}^{s+1}(\tilde{z}_j^{s+1})\leq \widetilde{w}_j=\vert I_j\vert$. Therefore, one finally gets $\textrm{ord}_{a^{s+1}}^{s+1}(\tilde{z}_j^{s+1})=\vert I_j\vert, \ \textrm{ for } I_j\in\ch^{s+1}.$ Claim \ref{claim-order} is now proved.\end{proof}

\begin{claim}\label{claim-change-coordinates}
The change of coordinates $(\Psi_j^{s+1})_{j=1,\dots,\tn_{s+1}}$ is well defined, i.e., Property {\em (A4)} holds true.
\end{claim}

\begin{proof}[Proof of Claim \ref{claim-change-coordinates}]
After performing Steps (s+1)-4-(a) and (s+1)-4-(b), one obtains a new system of coordinates $\tilde{z}^{s+1}$. In this system of coordinates, one can write $\xi_{i}^{s+1}$ as
\begin{eqnarray*}
\xi_i^{s+1}(\tilde{z}^{s+1})&=&\partial_{\tilde{z}_i^{s+1}}+\sum_{\substack{I_j\in\ch^{s+1}\\\vert I_j\vert \geq 2}} (\tilde{P}_{i,j}(\tilde{z}^{s+1}) +\tilde{R}_{i,j}(\tilde{z}^{s+1})) \partial_{\tilde{z}_j^{s+1}}+ \sum_{I_{\ell}\in\ck^{s+1}\setminus\ch^{s+1}} \tilde{Q}_{i,\ell}(\tilde{z}^{s+1}) \partial_{\tilde{z}_{\ell}^{s+1}},
\end{eqnarray*}
where $\tilde{P}_{i,j}$, $\tilde{R}_{i,j}$, and $\tilde{Q}_{i,\ell}$ are polynomials with
$\textrm{ord}_{a^{s+1}}^{s+1}(\tilde{P}_{i,j})=\widetilde{w}_j-1$, $\textrm{ord}_{a^{s+1}}^{s+1}(\tilde{R}_{i,j})\geq \widetilde{w}_j$, and
$\textrm{ord}_{a^{s+1}}^{s+1}(\tilde{Q}_{i,\ell})>s$. Since $\textrm{ord}_{a^{s+1}}^{s+1}(\tilde{z}^{s+1}_j)=\widetilde{w}_j,\ \hbox{  for }I_j\in\ch^{s+1}$, and $
\textrm{ord}_{a^{s+1}}^{s+1}(\tilde{z}^{s+1}_j)>s+1,\quad \hbox{ for }I_j\in\ck^{s+1}\setminus\ch^{s+1},$ the polynomials $\tilde{P}_{i,j}$ contain only variables $\tilde{z}^{s+1}_k$ with $\widetilde{w}_k \leq \widetilde{w}_j-1$. 

Let us now show that there exists a change of coordinates $\Psi^{s+1}$ which transforms coordinates $\tilde{z}^{s+1}$ into new coordinates $z^{s+1}$ such that $$\textrm{ord}_{a^{s+1}}^{s+1}(z^{s+1}_j)=\widetilde{w}_j,\ \textrm{ for } I_j\in\ch^{s+1},$$ $$\textrm{ord}_{a^{s+1}}^{s+1}(z^{s+1}_j)>s+1, \ \textrm{ for } I_j\in\ck^{s+1}\setminus\ch^{s+1},$$ and in the new coordinates, the $\tn_{s+1}$ first components $\xi_{i,j}^{s+1}(z^{s+1})$ of $\xi_i^{s+1}(z^{s+1})$ are in the form
$$
\xi_{i,j}^{s+1}(z^{s+1})= \delta_{i,\phi(j)}
                             P_{j}(z_1^{s+1},\dots,z_{j-1}^{s+1})+R_{i,j}(z^{s+1}), \quad j=1,\dots,\tn_{s+1},
$$
with ord$_{a^{s+1}}^{s+1}(R_{i,j})\geq \widetilde{w}_j$.

We first note that, once one has $\textrm{ord}_{a^{s+1}}^{s+1}(z^{s+1}_j)=\widetilde{w}_j$ for $I_j\in\ch^{s+1},$ and $\textrm{ord}_{a^{s+1}}^{s+1}(z^{s+1}_j)>s+1$ for $I_j\in\ck^{s+1}\setminus\ch^{s+1},$ then, the order of $P_{i,j}$ will be equal to its weighted degree, and thus automatically equal to $\widetilde{w}_j-1$ by construction of these polynomials.

Consider now $\check{\xi_i}^{s+1}$ defined in coordinates $\tilde{z}^{s+1}$ by
$$
\check{\xi}_i^{s+1}(\tilde{z}^{s+1})=\partial_{\tilde{z}_i^{s+1}}+ \sum_{I_j\in\ch^{s+1}} \tilde{P}_{i,j}(\tilde{z}^{s+1}) \partial_{\tilde{z}_j^{s+1}}.
$$
Recall that, by construction, the vector fields $\displaystyle\{\check{\xi}_i\}_{i=1,\dots,m}$ generate a free nilpotent Lie algebra of step $s+1$. Moreover, in the canonical coordinates of the second kind $(z_1^{s+1},\dots,z_{\tn_{s+1}}^{s+1})$ associated with $\{\check{\xi}^{s+1}_{I_k}\}_{I_k\in\ch^{s+1}}$, the vector fields $\check{\xi}_i^{s+1}$ are in the canonical form, i.e.
$$
\check{\xi}_i^{s+1}=\partial_{\tilde{z}_i^{s+1}}+ \sum_{\substack{I_j\in\ch^{s+1}\\\phi(j)=i}}P_{j}(z^{s+1})\partial_{z_j^{s+1}}.
$$

\noindent By definition of a system of coordinates, there exist $\tn_{s+1}$ smooth functions $(\Psi_{1}^{s+1},\dots, \Psi_{\tn_{s+1}}^{s+1})$ such that, for $j=1,\dots,\tn_{s+1}$, one has
$$z_j^{s+1}=\Psi_j^{s+1}(\tilde{z}_1^{s+1},\dots,\tilde{z}_{\tn_{s+1}}^{s+1}).$$
Expand now $\Psi_j^{s+1}$ in Taylor series. Since $\textrm{ord}_{a^{s+1}}^{s+1}(z^{s+1})=\widetilde{w}_j$, the Taylor expansion of $\Psi_j^{s+1}$ is a \emph{polynomial} of weighted degree equal to $\widetilde{w}_j$. Claim \ref{claim-change-coordinates} is now proved.\end{proof}

\begin{rem}\label{rem-change-coordinates}
The change of coordinates $(\Psi_j^{s+1})_{j=1,\dots,\tn_{s+1}}$ is computed by identification. Indeed, since ord$_{a^{s+1}}^{s+1}(z_j^{s+1})=\widetilde{w}_j$, and the nonholonomic order does not depend on any system of coordinates, then $\Psi_j^{s+1}$ is a function of order $\widetilde{w}_j$ at $a^{s+1}$, i.e., the Taylor expansion of $\Psi_j^{s+1}$ at $a^{s+1}$ contains only monomials of weighted degree equal to $\widetilde{w}_j$, and there is a finite number of such monomials. Therefore, the function $\Psi_j^{s+1}$ is necessarily in the following form
\begin{equation}\label{eq:change-coordinates}
\Psi_j^{s+1}(\tilde{z}^{s+1})=\sum_{w(\alpha)=\widetilde{w}_j} \beta_j^{\alpha}(\tilde{z}^{s+1}_1)^{\alpha_1} \dots (\tilde{z}^{s+1}_{\tn_{s+1}})^{\alpha_{\tn_{s+1}}},
\end{equation}where $\beta_j^{\alpha}$ are real numbers. Eq. (\ref{eq:change-coordinates}) is a finite sum and therefore the scalar coefficients $(\varphi_j^{\alpha})$ can be obtained by identification. Claim \ref{claim-change-coordinates} guarantees that such a set of real numbers $(\varphi_j^{\alpha})$ exists. Note also that, due to the constraint on the weight, Eq. (\ref{eq:change-coordinates}) only involves variables $\tilde{z}_k^{s+1}$ of weight less than $\widetilde{w}_j$, implying that the change of coordinates $(\Psi_j^{s+1})_{j=1,\dots,\tn_{s+1}}$ is naturally triangular.
\end{rem}

\begin{rem}\label{rem-change-example}
Let us now illustrate Remark \ref{rem-change-coordinates} with a simple example. Consider here a nilpotent system of step $2$ generated by two vector fields $(\xi_1, \xi_2)$. We have $\xi_{I_1}=\xi_1$, $\xi_{I_2}=\xi_2$ and $\xi_{I_3}=[\xi_1,\xi_2]$.  In coordinates $\tilde{z}=(\tilde{z}_1,\tilde{z}_2,\tilde{z}_3)$, $\xi_1$ and $\xi_2$ are necessarily in the form $\xi_1=(1,0,\alpha_1 \tilde{z}_1+\alpha_2\tilde{z}_2)$, and $\xi_2=(0,1,\beta_1 \tilde{z}_1+\beta_2\tilde{z}_2)$, where $\alpha_1,~\alpha_2,~\beta_1\textrm{ and }\beta_2$ are real numbers verifying $\beta_1-\alpha_2=1$. As mentioned in Remark \ref{rem-change-coordinates}, in the change of coordinates $(\Psi_1,\Psi_2,\Psi_3)$, every $\Psi_j$ is a homogeneous polynomial of weighted degree equal to $\widetilde{w}_j$. Set $z=(\Psi_1(\tilde{z}),\Psi_2(\tilde{z}),\Psi_3(\tilde{z}))=:(\tilde{z}_1,\tilde{z}_2,\tilde{z}_3+a\tilde{z}_1\tilde{z}_2+b\tilde{z}_1^2+c\tilde{z}_2^2),$ with $a,~b, \textrm{ and }c$ to be determined. One imposes
 that $\xi_2(z)=(0,1,z_1)$. After computation, one gets
\begin{eqnarray*}
(\alpha_1+2b)\tilde{z}_1+(\alpha_2+a)\tilde{z}_2&=&0,\\
(\beta_1+a)\tilde{z}_1+(\beta_2+2c)\tilde{z}_2&=&z_1=\tilde{z}_1.
\end{eqnarray*}
By identification, one gets $a=-\alpha_2$, $b=-\frac{\alpha_1}{2}$, $c=-\frac{\beta_2}{2}$, and in that case, $\beta_1+a=\beta_1-\alpha_2=1$ is automatically verified.
Then, the triangular change of coordinates $$(z_1,z_2,z_3)=(\tilde{z}_1,\tilde{z}_2,\tilde{z}_3-\alpha_2\tilde{z}_1\tilde{z}_2-\frac{\alpha_1}{2}\tilde{z}_1^2-\frac{\beta_2}{2}\tilde{z}^2)$$ puts $\xi_1$ and $\xi_2$ into the canonical form.
\end{rem}

\begin{claim}\label{prop5}
Property {\em (A5)} holds true at step $s+1$.
\end{claim}

\begin{proof}[Proof of Claim \ref{prop5}]
The proof goes by induction on the length of $I_k\in\ck^{s+1}$. It is similar to the one of Claim \ref{general-form-I}. For $\vert I_k\vert=1$, one has
\begin{eqnarray*}
\xi_i^{s+1}(z^{s+1})&=&\sum_{\substack{I_j\in\ch^{s+1}\\\phi(j)=i}}(P_{j}(z^{s+1})+R_{i,j}(z^{s+1})) \partial_{z_j^{s+1}}+\sum_{I_{\ell}\in\ck^{s+1}\setminus\ch^{s+1}}Q_{i,\ell}(z^{s+1})\partial_{z_{\ell}^{s+1}},
\end{eqnarray*}
with
$
\textrm{ord}_{a^{s+1}}^{s+1}(P_{j})=\vert I_j\vert-1,\ \textrm{ord}_{a^{s+1}}^{s+1}(R_{i,j})>\vert I_j\vert-1,\ \textrm{ord}_{a^{s+1}}^{s+1}(Q_{i,\ell})>s.$
Claim \ref{prop5} holds true for $\vert I_k\vert=1$.

Assume that Claim \ref{prop5} holds for brackets of length less than $s_1$. We show that it still holds true for brackets of length $s_1+1$. Consider $I_k\in\ck^{s+1}$ with $\vert I_k\vert=s_1+1$. Then, one has
\begin{eqnarray*}
&&\xi_{I_k}^{s+1}=[\xi_{I_{k_1}}^{s+1}, \xi_{I_{k_2}}^{s+1}]\\
&=&[\sum_{I_j\in\ch^{s+1}}(P_j^{k_1}+R_j^{k_1})\partial_{z_j^{s+1}}+\sum_{I_{\ell}\in\ck^{s+1}\setminus\ch^{s+1}}Q_{\ell}^{k_1}\partial_{z_{\ell}^{s+1}} , \sum_{I_j\in\ch^{s+1}}(P_j^{k_2}+R_j^{k_2})\partial_{z_j^{s+1}}+\sum_{I_{\ell}\in\ck^{s+1}\setminus\ch^{s+1}}Q_{\ell}^{k_2}\partial_{z_{\ell}^{s+1}}]\\
&=&\sum_{I_j\in\ch^{s+1}} [\sum_{I_i\in\ch^{s+1}}P_i^{k_1}\partial_{z_i^{s+1}}P_j^{k_2}- P_i^{k_2}\partial_{z_i^{s+1}}P_j^{k_1}]\partial_{z_j^{s+1}}\\
&+&\sum_{I_j\in\ch^{s+1}}[\sum_{I_i\in\ch^{s+1}}\{R_i^{k_1} \partial_{z_i^{s+1}}(P_j^{k_2}+R_j^{k_2})-R_i^{k_2}\partial_{z_i^{s+1}} (P_j^{k_1}+R_j^{k_1})\}+\{P_i^{k_1}\partial_{z_i^{s+1}}R_j^{k_2}-P_i^{k_2}\partial_{z_i^{s+1}} R_j^{k_1}\}\\
&&+\sum_{I_{\ell}\in\ck^{s+1}\setminus\ch^{s+1}}Q_{\ell}^{k_1} \partial_{z_{\ell}^{s+1}}(P_j^{k_2}+R_j^{k_2})-Q_{\ell}^{k_2}\partial_{z_{\ell}^{s+1}} (P_j^{k_1}+R_j^{k_1})]\partial_{z_j^{s+1}}\\
&+&\sum_{I_j\in\ck^{s+1}\setminus\ch^{s+1}}[\sum_{I_i\in\ch^{s+1}}(P_i^{k_1}+R_i^{k_1}) \partial_{z_i^{s+1}}Q_j^{k_2}-(P_i^{k_2}+R_i^{k_2})\partial_{z_i^{s+1}}Q_j^{k_1}\\
&&+\sum_{I_{\ell}\in\ck^{s+1}\setminus\ch^{s+1}}Q_{\ell}^{k_1}\partial_{z_{\ell}^{s+1}} Q_j^{k_2}-Q_{\ell}^{k_2}\partial_{z_{\ell}^{s+1}}Q_j^{k_1}]\partial_{z_j^{s+1}}.
\end{eqnarray*}

\noindent By the inductive hypothesis, one can proceed as follows.
\begin{itemize}
\item Taking into account the relation
$$\textrm{ord}_{a^{s+1}}^{s+1}(R_i^{k_1})>\vert I_i\vert-\vert I_{k_1}\vert,\hbox{ and }
\textrm{ord}_{a^{s+1}}^{s+1}(\partial_{z_i^{s+1}}(P_j^{k_2}+R_j^{k_2}))\geq \vert I_j\vert-\vert I_{k_2}\vert-\vert I_i\vert,$$
one deduces that $$\textrm{ord}_{a^{s+1}}^{s+1}R_i^{k_1}\partial_{z_i^{s+1}}(P_j^{k_2}+R_j^{k_2})>\vert I_j\vert-\vert I_k\vert.$$ By a similar argument, ord$_{a^{s+1}}^{s+1}R_i^{k_2}\partial_{z_i^{s+1}}(P_j^{k_1}+R_j^{k_1})>\vert I_j\vert-\vert I_k\vert$. Therefore, one gets
$$\textrm{ord}_{a^{s+1}}^{s+1}(R_i^{k_1}\partial_{z_i^{s+1}}(P_j^{k_2}+R_j^{k_2})- R_i^{k_2}\partial_{z_i^{s+1}}(P_j^{k_1}+R_j^{k_1}))>\vert I_j\vert-\vert I_k\vert.$$

\item Since $\textrm{ord}_{a^{s+1}}^{s+1}(P_i^{k_1})=\vert I_i\vert-\vert I_{k_1}\vert\hbox{ and } \textrm{ord}_{a^{s+1}}^{s+1}(\partial_{z_i^{s+1}}R_j^{k_2})>\vert I_j\vert-\vert I_{k_2}\vert-\vert I_i\vert,$ then $\textrm{ord}_{a^{s+1}}^{s+1}(P_i^{k_1}\partial_{z_i^{s+1}}R_j^{k_2})>\vert I_j\vert-\vert I_k\vert.$ By a similar argument, ord$_{a^{s+1}}^{s+1}(P_i^{k_2}\partial_{z_i^{s+1}}R_j^{k_1})>\vert I_j\vert-\vert I_k\vert$. One thus obtains $$\textrm{ord}_{a^{s+1}}^{s+1}(P_i^{k_1}\partial_{z_i^{s+1}}R_j^{k_2}-P_i^{k_2}\partial_{z_i^{s+1}}R_j^{k_1})>\vert I_j\vert-\vert I_k\vert.$$

\item Using the fact that $\textrm{ord}_{a^{s+1}}^{s+1}(\partial_{z_{\ell}^{s+1}}(P_j^{k_2}+R_j^{k_2}))>\vert I_j\vert-\vert I_{k_2}\vert-(s+1)\hbox{ and }\textrm{ord}_{a^{s+1}}^{s+1}(Q_{\ell}^{k_1})>s+1-\vert I_{k_1}\vert,$ then $\textrm{ord}_{a^{s+1}}^{s+1}(Q_{\ell}^{k_1}\partial_{z_{\ell}^{s+1}}(P_j^{k_2}+R_j^{k_2}))>\vert I_j\vert-\vert I_{k}\vert.$ By a similar argument, ord$_{a^{s+1}}^{s+1}Q_{\ell}^{k_2}\partial_{z_{\ell}^{s+1}}(P_j^{k_1}+R_j^{k_1})>\vert I_j\vert-\vert I_k\vert$. One deduces $\textrm{ord}_{a^{s+1}}^{s+1}(Q_{\ell}^{k_1}\partial_{z_{\ell}^{s+1}}(P_j^{k_2}+R_j^{k_2})- Q_{\ell}^{k_2}\partial_{z_{\ell}^{s+1}}(P_j^{k_1}+R_j^{k_1}))>\vert I_j\vert-\vert I_k\vert.$

\item Recall that $\textrm{ord}_{a^{s+1}}^{s+1}(P_i^{k_1}+R_i^{k_1})=\vert I_i\vert-\vert I_{k_1}\vert,\hbox{ and }\textrm{ord}_{a^{s+1}}^{s+1}(\partial_{z_i^{s+1}}Q_j^{k_2})>s+1-\vert I_{k_2}\vert-\vert I_i\vert,$ then $\textrm{ord}_{a^{s+1}}^{s+1}((P_i^{k_1}+R_i^{k_1})\partial_{z_i^{s+1}}Q_j^{k_2})>s+1-\vert I_k\vert.$ By a similar argument, $\textrm{ord}_{a^{s+1}}^{s+1}((P_i^{k_2}+R_i^{k_2})\partial_{z_i^{s+1}}Q_j^{k_1})>s+1-\vert I_k\vert.$
Therefore, it yields
$\textrm{ord}_{a^{s+1}}^{s+1}((P_i^{k_1}+R_i^{k_1})\partial_{z_i^{s+1}}Q_j^{k_2}- (P_i^{k_2}+R_i^{k_2})\partial_{z_i^{s+1}}Q_j^{k_1})>s+1-\vert I_k\vert.$

\item Since $\partial_{z_{\ell}^{s+1}}Q_j^{k_2}$ is a function, one knows by definition that ord$_{a^{s+1}}^{s+1}(\partial_{z_{\ell}^{s+1}}Q_j^{k_2})\geq 0$. As ord$_{a^{s+1}}^{s+1}(Q_{\ell}^{k_1})>s+1-\vert I_{k_1}\vert$, one has $\textrm{ord}_{a^{s+1}}^{s+1}(Q_{\ell}^{k_1}\partial_{z_{\ell}^{s+1}}Q_j^{k_2})>s+1-\vert I_{k_1}\vert=s+1-(\vert I_k\vert-\vert I_{k_2}\vert)>s+1-\vert I_k\vert.$ By a similar argument, ord$_{a^{s+1}}^{s+1}(Q_{\ell}^{k_2}\partial_{z_{\ell}^{s+1}}Q_j^{k_1})>s+1-\vert I_k\vert$.
One hence derives $\textrm{ord}_{a^{s+1}}^{s+1}(Q_{\ell}^{k_1}\partial_{z_{\ell}^{s+1}}Q_j^{k_2}-Q_{\ell}^{k_2}\partial_{z_{\ell}^{s+1}}Q_j^{k_1})>s+1-\vert I_k\vert.$
\end{itemize}

Summing up the above terms, one gets, for $I_k\in\ck^{s+1}$ of length $s_1+1$, that the bracket $\xi_{I_k}^{s+1}$ can be written in the form
\begin{eqnarray*}
\xi_{I_k}^{s+1}(z^{s+1})&=&\sum_{I_j\in\ch^{s+1}}(P_j^k(z^{s+1})+R_j^k(z^{s+1}))\partial_{z_j^{s+1}}+\sum_{I_{j}\in\ck^{s+1}\setminus\ch^{s+1}}Q_{j}^k(z^{s+1})\partial_{z_{j}^{s+1}},
\end{eqnarray*}with $\textrm{ord}_{a^{s+1}}^{s+1}(P_j^k)=\vert I_j\vert-\vert I_k\vert,\qquad \textrm{ord}_{a^{s+1}}^{s+1}(R_j^k)>\vert I_j\vert-\vert I_k\vert,$ and $\textrm{ord}_{a^{s+1}}^{s+1}(Q_{\ell}^{k})>s+1-\vert I_k\vert.$ Claim \ref{prop5} is now proved.
\end{proof}

In conclusion, Properties (A1)-(A5) still hold true at step $s+1$ in the Desingularization Algorithm. The induction step is established, which terminates the proof of Proposition \ref{prop-desingularization}.

\end{proof}


\section{Global Steering Method for Regular Systems}\label{regular-case}
By taking into account the Desingularization Algorithm presented in Chapter \ref{Desing}, we assume in this chapter and without loss of generality that the family of vectors fields $X=\{X_1,\dots, X_m\}$ is free up to step $r$ (cf. Definition \ref{free-s}). Recall that, in that case, every point $x\in\Omega$ is regular and the growth vector is constant on $\Omega$. We present in Section \ref{bellaiche-construction} an algebraic construction of privileged coordinates and a nonholonomic first order approximation of $X$ under canonical form. For regular systems, this construction also provides a continuously varying system of privileged coordinates. We then propose in Section \ref{GASA} a global motion planning algorithm for regular systems.

\subsection{Construction of the approximate system $\mathcal{A}^X$}\label{bellaiche-construction}


For every point $\aa$ in $\Omega$, we construct the first order approximate system $\mathcal{A}^X(a)$ of the system $X$ at $a$ (cf. Definition \ref{approx}) as follows: 
\begin{enumerate}
\item[Step (1)] Take $\{X_{I_j}\}_{I_j\in\ch^r}$. Set $w_j=\widetilde{w}_j$ for $j=1,\dots,n$.
\item[Step (2)] Construct the linear system of coordinates $y=(y_1,\dots,y_n)$ such that $\partial_{y_j}=X_{I_j}(\aa)$.
\item[Step (3)] Build the system of privileged coordinates $\widetilde{z}=(\widetilde{z}_1,\dots,\widetilde{z}_n)$ by the following iterative formula: for $j=1,\dots,n$,
\begin{equation}\label{regular-priv}
\widetilde{z}_j:=y_j+\sum_{k=2}^{w_j-1}h_k(y_1,\dots,y_{j-1}),
\end{equation}where, for $k=2,\dots,w_j-1$,
\begin{equation*}
h_k(y_1,\dots,y_{j-1})=-\sum_{\substack{ \vert \alpha\vert=k\\ w(\alpha)<w_j}}X_{I_1}^{\alpha_1}\dots X_{I_{j-1}}^{\alpha_{j-1}}\cdot(y_j+\sum_{q=2}^{k-1}h_q)(y)\arrowvert_{y=0}\ \frac{y_1^{\alpha_1}}{\alpha_1!}\cdots\frac{y_{j-1}^{\alpha_{j-1}}}{\alpha_{j-1}!},
\end{equation*}with $\vert\alpha\vert:=\alpha_1+\cdots+\alpha_{n}$.

\item[Step (4)] For $i=1,\dots,m$, compute the Taylor expansion of $X_i(\widetilde{z})$ at $0$, and express every vector field as a sum of vector fields which are homogeneous with respect to the weighted degree defined by the sequence $(w_j)_{j=1,\dots,n}$:
$$X_i(\widetilde{z})=X_i^{(-1)}(\widetilde{z})+X_i^{(0)}(\widetilde{z})+\cdots,$$
  where we use $X_i^{(k)}(\widetilde{z})$ to denote the sum of all the
  terms of weighted degree equal to $k$. Set
  $\widehat{X}_i^a(\widetilde{z}):=X_i^{(-1)}(\widetilde{z})$.

\item[Step (5)]  For $j=1,\dots,n$, identify homogeneous
  polynomials $\Psi_j$ of weighted degree equal to $w_j$ such that, in
  the  system of privileged coordinates
  $z:=(z_1,\dots,z_n)$ defined by
$$
z_j:=\Psi_j(\widetilde{z}_1,\dots,\widetilde{z}_{j-1}),
  \ \textrm{ for } j=1,\dots,n,
$$
the approximate system
$$
\widehat{X}^a(z)=\{z_*\widehat{X}^a_1(\widetilde{z}),\dots,
  z_*\widehat{X}^a_m(\widetilde{z})\}
$$
is in the canonical form.

\item[Step (6)] Set $\mathcal{A}^X(a):=\widehat{X}^a$ and
  $\Phi^X(a,\cdot):=$ the mapping $x \mapsto z$.
\end{enumerate}

\begin{rem}\label{proof-priv}
Steps (1)-(3) construct a system of privileged coordinates
$\widetilde{z}$.  The proof that $\widetilde{z}$ is a system of {privileged}
coordinates is essentially based on Lemma
\ref{lemma-bellaiche}. Roughly speaking, the idea to obtain $\widetilde{z}_j$ from
$y_j$ goes as follows: for every $\alpha=(\alpha_1,\dots,\alpha_n)$
with $w(\alpha)<w_j$ (so $\alpha_j=\cdots=\alpha_n=0$), compute
$X_{I_1}^{\alpha_1}\cdots X_{I_{j-1}}^{\alpha_{j-1}}\cdot
y_j(y)\vert_{y=0}$. If it is not equal to zero, then replace $y_j$
by
$$
y_j-(X_{I_1}^{\alpha_1}\cdots X_{I_{j-1}}^{\alpha_{j-1}}\cdot
y_j)(y)\vert_{y=0}\ \frac{y_1^{\alpha_1}}{\alpha_1!} \cdots
\frac{y_{j-1}^{\alpha_{j-1}}}{\alpha_{j-1}!}.
$$
With that new value of $y_j$, one gets $X_{I_1}^{\alpha_1}\cdots
X_{I_{j-1}}^{\alpha_{j-1}}\cdot y_j(y)\vert_{y=0}=0$. Therefore, by
Lemma \ref{lemma-bellaiche}, one has ord$_{\aa}(\widetilde{z}_j)\geq w_j$ for
$j=1,\dots,n$.
On the other hand, since Step (3) of the construction does not modify
the linear part, the system of coordinates $\widetilde{z}$ remains
adapted. By Remark \ref{rem-weight-order}, one also has
ord$_{\aa}(\widetilde{z}_j)\leq w_j$, and therefore,
ord$_{\aa}(\widetilde{z}_j)=w_j$.
\end{rem}

\begin{rem}
The existence of $\Psi_j$ in Step (5) is guaranteed by a simple
modification of Claim \ref{claim-change-coordinates},
page~\pageref{claim-change-coordinates}, see also Remarks
\ref{rem-change-coordinates} and \ref{rem-change-example}. The key point is, in the current case, the exponential coordinates are \emph{algebraic}.
\end{rem}

\begin{rem}
We will propose in Section \ref{control-nilpotent} an effective and
exact method for steering general nilpotent systems given in the
canonical form.
\end{rem}

It results from \cite{Bellaiche1} that, for {\em regular systems}, the
mapping $\Phi^X:(\aa,x)\rightarrow z$ is a continuously varying
system of privileged coordinates on $\Omega$. Note also that the
coordinates $z$ are obtained from $y$ by expressions of the form
\begin{eqnarray*}
z_1&=&y_1,\\
z_2&=&y_2+\textrm{pol}_2(y_1),\\
&\vdots&\\
z_n&=&y_n+\textrm{pol}_n(y_1,\dots,y_{n-1}),
\end{eqnarray*} where, for $j=1,\dots,n$, the function pol$_j(\cdot)$
is a polynomial which does not contain constant nor linear terms. Due
to the triangular form of this change of coordinates, the inverse
change of coordinates from $z$ to $y$ bears exactly the same
form. Therefore, the mapping $z=\Phi^X(\aa,\cdot)$ is defined on the
whole $\Omega$, i.e., $\Phi^X$ has an infinite injectivity radius.  We
also note that, by construction, $\mathcal{A}^X$ is a nonholonomic
first order approximation (cf. Definition \ref{approx}) and its
continuity results from the continuity of the mapping
$\Phi^X:(\aa,x)\mapsto z$. In summary, we have the following
proposition.

\begin{prop}\label{prop:continuity}
The mapping $\Phi^X$ is a continuously varying system of privileged
coordinates on $\Omega$ and the mapping $\mathcal{A}^X$ is a
continuous approximation of $X$ on $\Omega$.
\end{prop}

The following theorem is a consequence of Proposition
\ref{prop:continuity} and Corollary \ref{coro:unif_contr}.

\begin{theo}\label{regular-uniform}
Let $\cv^c$ be a compact subset of $\Omega$. If $\mathcal{A}^X$ is
provided with a sub-optimal steering law \emph{(cf. Definitions
  \ref{de:steering-law} and \ref{de:sub-opt})}, then the LAS method
$\ap$ associated with $\mathcal{A}^X$ and its steering law
\emph{(cf. Definition \ref{de:ap})} is uniformly locally contractive on $\cv^c$.
\end{theo}

\begin{rem}\label{rem:form-approx}
Due to Step (5) in the construction procedure, the approximate system
$\mathcal{A}^X(a)$ is under canonical form in a system of privileged
coordinates $z$. Therefore, $\mathcal{A}^X(a)$ has always the same
form, regardless of the control system $X$ or the approximate point
$a\in\Omega$. The specificity of each system or each approximate point
is hidden in the change of coordinates $\Phi^X$.
\end{rem}

\begin{rem}\label{rem:goal0}
It is important to notice that the approximate system used in the LAS
method is a nonholonomic first order approximation at the \emph{goal}
point $a$ (cf. Definition \ref{de:ap}). Therefore, the steering
control always displaces $\mathcal{A}^X(a)$ from some position (which
is the image by $\Phi^X(a,\cdot)$ of the current point of the original system) to
0 (which is $\Phi^X(a,a)$ by construction) in coordinates $z$. The latter
fact plays a crucial role in getting the \emph{sub-optimality} for the
steering law (see Section \ref{sec:sub-opt} for more details).
\end{rem} 

\subsection{Approximate steering algorithm}\label{GASA}

Let $\cv^c\subset\Omega$ be a connected compact set equal to the closure
of its interior and $(x^{\textrm{initial}},
x^{\textrm{final}})\in\cv^c\times\cv^c$. We devise, under the
assumptions of Theorem \ref{regular-uniform}, an algorithm (Algorithm
\ref{algo_gf} below) which steers
System~(\ref{CS}) from $x^{\textrm{initial}}$ to
$x^{\textrm{final}}$. That algorithm does not require
any a priori knowledge on the critical distance
$\varepsilon_{\cv^c}$. Note that this algorithm bears similarities
with {\em trust-region} methods in optimization (see \cite{Gilbert}
for more details). 

Recall first that the family of vectors fields $X=\{X_1,\dots, X_m\}$ is assumed to be free up to step $r$. As a consequence the weights $(w_1,\dots, w_n)$ are equal at every point $a \in\cv^c$  to  $(\widetilde{w}_1,\dots, \widetilde{w}_n)$, the free weights of step $r$. Hence the pseudo-norm $\|\cdot \|_a$ (see Definition~\ref{de:pseudonorm}) does not depend on $a \in\cv^c$ and will be denoted as $\| \cdot \|_r$.

The parameterized path $t \mapsto \delta_{t}(x)$ is defined by
$$
\delta_{t}(x) :=  (t^{w_1}
z_1(x), \dots, t^{w_n} z_n(x)), \ \textrm{ for }x\in\Omega,
$$
where $z:= \Phi^X(x^{\textrm{final}},\cdot)$. Note that
$\delta_{t}$ is the (weighted) {\em dilatation} in
privileged coordinates at $x^{\textrm{final}}$ with parameter $t$. In particular, $\| z(\delta_{t}(x))\|_r = |t| \, \| z(x)\|_r$. We
also define the function $\textrm{Subgoal}$ as follows.

\begin{center}
\begin{tabular}{ll}
\hline $\mathrm{Subgoal}(\bx,\eta,j)$ &\\ & \\
1. $t_j:=\max(0,1-\frac{j \eta}{\Vert z(\bx)\Vert_r})$; \\ [0.2cm]
2. $\mathrm{Subgoal}(\bx,\eta,j):=\delta_{t_j}(\bx)$ &  \\ \hline
\end{tabular}
\end{center}
We note that the formula for generating $t_j$ guarantees that
$$
\Vert z(\mathrm{Subgoal}(\bx,\eta,j))-z(\mathrm{Subgoal}(\bx,\eta,j-1))\Vert_r \leq \eta,
$$
 and that $x^d=x^{\textrm{final}}$ for $j$ large
enough. 

\begin{algorithm}
\caption{$\gf~(x^{\textrm{initial}}, x^{\textrm{final}}, e, {\cv^c}, \ap)$}
\label{algo_gf}
\begin{algorithmic}[1]
\STATE $i:=0$; $j:=1$; 
\STATE $x_{i}:=x^{\textrm{initial}}$; $\bx := x^{\textrm{initial}}$; 
\STATE $\eta:=\Vert z(x^{\textrm{initial}})\Vert_r$; \qquad\COMMENT{\textsf{initial choice of the maximum step size;}} 
\WHILE{$\Vert z(x_i)\Vert_r > e$} 
\STATE $x^d:=\mbox{Subgoal}~(\bx,\eta,j)$; 
\STATE $x:= \ap~(x_{i},x^d)$;  
\IF[\textsf{if the system is not approaching the subgoal,}]{$\Vert \Phi^X(x^d,x)\Vert_r > \frac{1}{2} \Vert \Phi^X(x^d,x_{i})\Vert_r$}  
\STATE $\eta:=\frac{\eta}{2}$;\qquad\COMMENT{\textsf{reduce the maximum step size,}}  
\STATE $\bx:=x_{i}$; $j:=1$;\qquad\COMMENT{\textsf{change the path $\delta_{0,t}(\bar{x})$.}} 
\ELSE 
\STATE $i:=i+1$; $j:=j+1$; 
\STATE $x_{i}:=x$;  $x^d_{i}:=x^d$; 
\ENDIF 
\ENDWHILE 
\RETURN $x_i$. 
\end{algorithmic}
\end{algorithm}

The global convergence of Algorithm \ref{algo_gf} is
established in the following theorem. For the sake of clarity, we
first assume that the sequences $(x_i)_{i\geq 0}$ and $(x_i^d)_{i \geq
0}$ constructed by Algorithm \ref{algo_gf} both stay within $\cv^c$. This assumption being of a purely numerical nature, we explain at the end of this section how we can remove it by adding suitable intermediate steps to Algorithm \ref{algo_gf}.

\begin{theo}\label{th:GlobConv}
Let $\cv^c\subset\Omega$ be a connected compact set equal to the closure of its interior. Assume that
\begin{itemize}
\item[(i)] the approximate system system $\mathcal{A}^X$ is provided with a sub-optimal steering law;
\item[(ii)] the LAS method $\ap$ is associated with $\mathcal{A}^X$ and its steering law;
\end{itemize}Then, $\forall~ (x^{\textrm{\emph{initial}}},x^{\textrm{\emph{final}}})\in \cv^c\times \cv^c$,
\emph{Algorithm \ref{algo_gf}} terminates in a finite number
of steps for any choice of the tolerance $e>0$ provided that the sequences $(x_i)_{i \geq 0}$ and $(x_i^d)_{i \geq
0}$ both belong to $\cv^c$.
\end{theo}

\begin{proof}[Proof of Theorem \ref{th:GlobConv}]
Note first that, if the conditional statement of Line 7 is not true
for every $i$ greater than some $i_0$, then $x^d_i=x^{\textrm{final}}$ after a finite
number of iterations. In this case, the error $\Vert z(x_i)\Vert_r$ is reduced
at each iteration and the algorithm stops when it becomes smaller
than the given tolerance $e$. This happens in particular if
$d(x_{i},x^d) < \eps_{\cv^c}$ for all $i$ greater than $i_0$ because
condition~(\ref{eq:pseudcontr}) is verified. Another preliminary remark is that, due to the continuity of the
control distance and of the function $\Vert z(\cdot)\Vert_r$,
there exists $\bet >0$ such that, for every pair $(x_1, x_2) \in {\cv^c}\times {\cv^c}$, one has
\begin{equation}
\Vert z(x_1)-z(x_2)\Vert_r < \bet \ \Longrightarrow \ d(x_1,x_2) < \frac{\eps_{\cv^c}}{2}.
\label{eq:distcont}
\end{equation}

In the following, we will prove by induction that if, at some step
$i_0$, one has $\eta < \bet$, then, for all  $i> i_0$, one has $d(x_{i-1},x_i^d) < (1/2+\cdots+(1/2)^{i-i_0})\eps_{\cv^c}< \eps_{\cv^c}$.

We assume without loss of generality that $i_0=0$ and $\bx =x_0$. For $i=1$, by construction,
$x^d = \mbox{Subgoal}(x_0,\eta,1)$ and $\Vert z(x_0)-z(x^d)\Vert_r \leq \eta < \bet$.

In view of~(\ref{eq:distcont}), one then has $d(x_0,x^d) <
\eps_{\cv^c}/2$. In view of~(\ref{eq:pseudcontr}), the conditional statement of Line
7 is not true, therefore  $x_1^d=x^d$ and one has
$
\displaystyle d(x_0,x_1^d) < \eps_{\cv^c}/2.
$

Assume now that for $i>1$ one has:
\begin{equation}
d(x_{i-2},x_{i-1}^d) < (1/2+\dots+(1/2)^{i-1}) \eps_{\cv^c}.
\label{eq:indhyp}
\end{equation}
The subgoal $x_{i-1}^d$ is of the form $\mathrm{Subgoal}(\bx,\eta,j)$. Let $x^d = \mathrm{Subgoal}(\bx,\eta,j+1)$. One can write:
\[
d(x_{i-1},x^d) \leq d(x_{i-1},x_{i-1}^d) +
d(x_{i-1}^d,x^d).
\]
By construction, it is $\Vert z(x_{i-1}^d)-z(x^d)\Vert_r \leq \eta < \bet$, which implies $d(x_{i-1}^d,x^d) < \eps_{\cv^c}/2$.
  The induction
hypothesis~(\ref{eq:indhyp}) implies that $d(x_{i-1},x_{i-1}^d) \leq \frac{1}{2} d(x_{i-2},x_{i-1}^d)$.

Finally, one gets $$d(x_{i-1},x^d) \leq \frac{1}{2} d(x_{i-2},x_{i-1}^d) +
d(x_{i-1}^d,x^d) \leq (1/2+\dots+(1/2)^{i})\eps_{\cv^c}.$$

In view of~(\ref{eq:pseudcontr}), the conditional statement of Line
7 is not true, and so $x_i^d = x^d$. This ends the induction. 

Notice that, at some step $i$, $\eta \geq \bet$, the conditional statement of
Line 7 could be false.  In this case, $\eta$ is decreased as in
Line 8.  The
updating law of $\eta$ guarantees that after a finite number of
iterations of Line 8, there holds $\eta < \bet$. This ends the proof.
\end{proof}

When the working space $\Omega$ is equal to the whole $\R^n$, the
assumption that the sequences $(x_i)_{i \geq 0}$ and $(x_i^d)_{i \geq
0}$ constructed by Algorithm \ref{algo_gf} both stay within a compact set ${\cv^c}$ can be removed.
This requires a simple modification of Lines $11$ and $12$ of Algorithm \ref{algo_gf}. 

We choose a real number $R$ close to one, precisely
$(\frac{1}{2})^{1/(r+1)^2} < R < 1,$ where $r$ is the maximum value of the degree of nonholonomy of System (\ref{CS}) on $\cv^c$. For every non-negative integer
$k$, we set $R_k = 1 + R + \cdots + R^k$. The algorithm is modified
as follows. Introduce first a new variable $k$, and add the
initialization $k:=0$. Replace then Lines $11$ and $12$ of Algorithm \ref{algo_gf} by the procedure below.

\begin{algorithm}[!h]
\begin{algorithmic}[1]
\IF{$\Vert z(x)\Vert_r \geq R_{k+1} \Vert z(x^{\textrm{initial}})\Vert_r$} 
\STATE $\eta:=\frac{\eta}{2}$; 
\ELSIF{$R_k \Vert z(x^{\textrm{initial}})\Vert_r \leq \Vert z(x)\Vert_r < R_{k+1} \Vert z(x^{\textrm{initial}})\Vert_r$} 
\STATE $i:=i+1$; $j:=j+1$; 
\STATE $x_{i}:=x$; $x^d_{i}:=x^d$;  
\STATE $\eta:=\frac{\eta}{2}$; 
\STATE $k:=k+1$; 
\ELSIF{$\Vert z(x)\Vert_r\leq R_k \Vert z(x^{\textrm{initial}})\Vert_r$} 
\STATE $i:=i+1$; $j:=j+1$; 
\STATE $x_{i}:=x$; $x^d_{i}:=x^d$; 
\ENDIF 
\end{algorithmic}
\end{algorithm}

This procedure guarantees that the sequences $(x_i)_{i \geq 0}$ and
$(x_i^d)_{i \geq 0}$ of the algorithm both belong to the compact set
$$
K= \{ x\in \R^n \ : \ \Vert z(x)\Vert_r \leq \frac{1}{1-R} \Vert z(x^{\textrm{initial}})\Vert_r \}.
$$
Moreover, at each iteration of the algorithm, the new variable $k$
is such that
$$
\Vert z(x_i)\Vert_r \geq R_k \Vert z(x^{\textrm{initial}})\Vert_r \quad \Rightarrow \quad \eta \leq
\frac{\Vert z(x^{\textrm{initial}})\Vert_r}{2^k}.
$$ 

For the sake of clarity, we state here the complete modified algorithm named as Algorithm \ref{algo_gfmod}.

\begin{algorithm}
\caption{$\gfmod~(x^{\textrm{initial}}, x^{\textrm{final}}, e, {\cv^c}, \ap)$}
\label{algo_gfmod}
\begin{algorithmic}[1]
\STATE $i:=0$; $j:=1$; 
\STATE $x_{i}:=x^{\textrm{initial}}$; $\bx := x^{\textrm{initial}}$; 
\STATE $\eta:=\Vert z(x^{\textrm{initial}})\Vert_r$; \qquad\COMMENT{\textsf{initial choice of the maximum step size;}} 
\WHILE{$\Vert z(x_i)\Vert_r > e$} 
\STATE $x^d:=\mbox{Subgoal}~(\bx,\eta,j)$; 
\STATE $x:= \ap~(x_{i},x^d)$;  
\IF[\textsf{if the system is not approaching the subgoal,}]{$\Vert \Phi^X(x^d,x)\Vert_r > \frac{1}{2} \Vert \Phi^X(x^d,x_{i})\Vert_r$}  
\STATE $\eta:=\frac{\eta}{2}$;\qquad\COMMENT{\textsf{reduce the maximum step size,}}  
\STATE $\bx:=x_{i}$; $j:=1$;\qquad\COMMENT{\textsf{change the path $\delta_{0,t}(\bar{x})$.}} 
\ELSIF{$\Vert z(x)\Vert_r \geq R_{k+1} \Vert z(x^{\textrm{initial}})\Vert_r$} 
\STATE $\eta:=\frac{\eta}{2}$; 
\ELSIF{$R_k \Vert z(x^{\textrm{initial}})\Vert_r \leq \Vert z(x)\Vert_r < R_{k+1} \Vert z(x^{\textrm{initial}})\Vert_r$} 
\STATE $i:=i+1$; $j:=j+1$; 
\STATE $x_{i}:=x$; $x^d_{i}:=x^d$;  
\STATE $\eta:=\frac{\eta}{2}$; 
\STATE $k:=k+1$; 
\ELSIF{$\Vert z(x)\Vert_r\leq R_k \Vert z(x^{\textrm{initial}})\Vert_r$} 
\STATE $i:=i+1$; $j:=j+1$; 
\STATE $x_{i}:=x$; $x^d_{i}:=x^d$; 
\ENDIF 
\ENDWHILE 
\RETURN $x_i$. 
\end{algorithmic}
\end{algorithm}

\begin{prop}\label{mod-global-algo}
Let ${\cv^c}\subset\Omega$ be a connected compact set equal to the closure of its interior. Under the assumptions \emph{(i)} and \emph{(ii)} of Theorem \ref{th:GlobConv}, $\forall~ (x^{\textrm{\emph{initial}}},x^{\textrm{\emph{final}}})\in {\cv^c}\times {\cv^c}$, \emph{Algorithm \ref{algo_gfmod}} terminates in a finite number of iterations for
any choice of the tolerance $e>0$.
\end{prop}

\begin{proof}[Proof of Proposition \ref{mod-global-algo}]
Notice that Lines $17$, $18$, and $19$ in Algorithm \ref{algo_gfmod} are identical to Lines $10$, $11$, and $12$ in Algorithm \ref{algo_gf}. Therefore, it is enough
to show that, after a finite number of iterations, the condition of Line $17$ in Algorithm \ref{algo_gfmod} holds true. Another preliminary remark is that the distance $\Vert z(x)- z(y)\Vert_r$ give a rough
estimate of the sub-Riemannian distance. Indeed it follows from
Theorem~\ref{theo:unif_contr} that, for every pair of close enough points $(x,y)\in {\cv^c}\times {\cv^c}$, one has
\begin{equation}
\label{eq:dfnofd0}
    \frac{1}{C_0} \Vert z(x)-z(y)\Vert_r^{r+1} \leq d(x,y) \leq C_0
    \Vert z(x)-z(y)\Vert_r^{1/(r+1)},
\end{equation}
where $C_0$ is a positive constant. As a consequence, Eq. (\ref{eq:distcont}) holds true if $\bet \leq
(\eps_{\cv^c}/(2C_0))^{r+1}$.

Let us choose a positive $\bet$ smaller than $(\eps_{\cv^c}/(2C_0))^{r+1}$. We next show that if, at
some step $i_0$, $\eta <\bet$, then the case of Line 10 and the one of Line 12 occur only
in a finite number of iterations. Recall first that, from the proof of Theorem~\ref{th:GlobConv}, one
gets, for every $i > i_0$,
$$
\Vert z(x_i^d)\Vert_r \leq \Vert z(x_{i_0})\Vert_r \quad \hbox{and} \quad d(x_{i},x_i^d)
\leq \eps_{\cv^c}.
$$
In view of Eq. (\ref{eq:dfnofd0}), an obvious adaptation of the latter
proof yields, for every $i>i_0$, $d(x_{i},x_i^d) \leq 2 C_0
\eta^{1/(r+1)}$, and thus
$$\Vert z(x_{i})-z(x_i^d)\Vert_r \leq (2 C_0^2)^{1/(r+1)}
\eta^{1/(r+1)^2}.$$ Finally one gets
\begin{eqnarray}
\label{eq:inegtriang}
\Vert z(x_i)\Vert_r &\leq&  \Vert z(x_{i}^d)\Vert_r + \Vert z(x_i)-z(x_{i}^d)\Vert_r\nonumber\\
& \leq& \Vert z(x_{i_0})\Vert_r+(2 C_0^2)^{1/(r+1)}\eta^{1/(r+1)^2}.
\end{eqnarray}

On the other hand, there exists an integer $k_0$ such that
$\eta \geq \frac{\Vert z(x^{\textrm{initial}})\Vert_r}{2^{k_0}}$. This implies that
$\Vert z(x_{i_0})\Vert_r \leq R_{k_0} \Vert z(x^{\textrm{initial}})\Vert_r$. Up to reducing $\bet$, and so
increasing $k_0$, assume
$$
(2 C_0^2)^{1/(r+1)} (\frac{\Vert z(x^{\textrm{initial}})\Vert_r}{2^{k_0}})^{1/(r+1)^2} \leq
R^{k_0+1} \Vert z(x^{\textrm{initial}})\Vert_r,
$$
since one has chosen $R> (\frac{1}{2})^{1/(r+1)^2}$. Using
Eq. (\ref{eq:inegtriang}), it holds, for every $i \geq i_0$,
$
\Vert z(x_i)\Vert_r \leq R_{k_0} \Vert z(x^{\textrm{initial}})\Vert_r + R^{k_0+1} \Vert z(x^{\textrm{initial}})\Vert_r =  R_{k_0+1}
\Vert z(x^{\textrm{initial}})\Vert_r.
$ Therefore, the case of Line 10 and the one of Line 12 occur in at most $k_0+1$
iterations. Applying again the arguments of the proof of Theorem~\ref{th:GlobConv}, the conclusion follows.
\end{proof}

\begin{rem}

It is worth pointing out that the additional steps involved in Algorithm \ref{algo_gfmod} are designed to prevent the sequences $(x_i)_{i \geq 0}$ and $(x_i^d)_{i \geq
0}$ from accumulating toward the boundary of the compact ${\cv^c}$. There exist other numerical artifacts of probabilistic nature which solve this problem. One also deduces from the proof of Proposition \ref{mod-global-algo} that if the points $x^{\textrm{initial}}$ and $x^{\textrm{final}}$ are \emph{far} enough from the boundary of ${\cv^c}$, the sequences $(x_i)_{i \geq 0}$ and $(x_i^d)_{i \geq}$ will remain in ${\cv^c}$.

\end{rem}



\section{Exact Steering Method for Nilpotent Systems}\label{control-nilpotent}

In this chapter, we devise an exact steering method for general nilpotent systems. Without loss of generality, we assume that the system $X=\{X_1,\dots,X_m\}$ is nilpotent of step $r$, free up to step $r$, and given in the canonical form in coordinates $x$. Recall that, under this assumption, the dynamics is written as follows
\begin{equation}\label{cano-f5}
\begin{array}{llll}
\dot{x}_i&=&u_i,&\textrm{ if }i=1,\dots,m;\\
\dot{x}_{I}&=&\frac{1}{k!}x_{I_L}\dot{x}_{I_R},&\textrm{ if }X_{I}=\textrm{ad}^k_{X_{I_L}}X_{I_R},\ \ I_L, I_R\in\ch^r,
\end{array}
\end{equation} where the components of $x$ are numbered by the elements of $\ch^r$, i.e., for $I\in\ch^r$, the component $x_I$ corresponds to the element $X_I$.
We also assume that we want to steer the System (\ref{cano-f5}) from any point $x\in\R^{\tn_r}$ to the origin $0$ of $\R^{\tn_r}$.  

\begin{rem}\label{rem:nilpotent-general}
Note that these two assumptions are not restrictive since, for general nilpotent systems, in order to steer from $x^{\textrm{initial}}$ to $x^{\textrm{final}}$, it suffices to apply the Desingularization Algorithm at the final point $x^{\textrm{final}}$ (see also Remark \ref{rem:nilpotent}).
\end{rem}

This method can also be applied for the construction of a sub-optimal steering law for the approximate system $\mathcal{A}^X$ defined in Section \ref{bellaiche-construction}. For practical uses, we require that the inputs give rise to regular trajectories (i.e., at least $C^1$), which are not too ``complex" in the sense that, during
the control process, we do not want the system to stop too many times or to make a large number of maneuvers.

Several methods were proposed in the literature for steering nilpotent systems. In \cite{Lafferriere}, the authors make use of piecewise constant controls and obtain smooth controls by imposing some special parameterization (namely by requiring the control system to stop during the control process). In that case, the regularity of the inputs is recovered by using a reparameterization of the time, which cannot prevent in general the occurrence of cusps or corners for the corresponding trajectories. However, regularity of the {\em trajectories} is generally mandatory for robotic applications.  Therefore, the method proposed in \cite{Lafferriere2} is not adapted to such applications. In \cite{Lafferriere}, the proposed inputs are polynomial functions in time, but an algebraic system must be inverted in order to access to these inputs. Moreover, the size and the degree of this algebraic system increase exponentially with respect to the dimension of state space, and there does not exist a
general efficient exact method to solve it. Even the existence of solutions is a non trivial issue. Furthermore, the methods \cite{Lafferriere} and \cite{Lafferriere2} both make use of exponential coordinates which are not explicit and thus require in general numerical integrations of nonlinear differential equations. That prevents the use of these methods in an iterative scheme such as Algorithm \ref{algo_rrt}. Let us also mention the path approximation method by Liu and Sussmann \cite{Liu}, which uses unbounded sequences of sinusoids. Even though this method bears similar theoretical aspects with our method, it is not adapted from a numerical point of view to the motion planning issue since it relies on a limit process of highly oscillating inputs.

\subsection{Steering by sinusoids}
We consider input functions in the form of linear combinations of sinusoids with integer frequencies. In \cite{Murray}, authors used this family of inputs to control the chained-form systems.

We first note that if every component of the input $u=(u_1,\dots,u_m)$ in Eq. (\ref{cano-f5}) is a linear combination of sinusoids with integer frequencies, then the dynamics of every component in Eq. (\ref{cano-f5}) is also a linear combination of sinusoids with integer frequencies, which are themselves linear combinations of frequencies involved in the input $u$. One may therefore expect to move some components during a $2\pi$ time-period without modifying others if the frequencies in $u$ are properly chosen. Due to the triangular form of Eq. (\ref{cano-f5}), it is reasonable to expect to move the components of $x$ one after another according to the order ``$\prec$" induced by the P. Hall basis. In that case, one must ensure that all the components already moved to their preassigned values return to the same values after each $2\pi-$period of control process, while the component under consideration arrives to its preassigned position. However, all the components cannot be moved independently by using sinusoids. For that purpose, we introduce the following notion of {\em equivalence}.

\begin{de}[\emph{Equivalence}]\label{Equivalence-Class}
Two elements $X_I$ and $X_J$ in the P. Hall family are said to be {\em equivalent} if $\Delta_i(X_I)=\Delta_i(X_J)$ for $i=1,\dots, m$, where we use $\Delta_i(X_I)$ to denote the number of times $X_i$ occurs in $X_I$. We write $X_I\sim X_J$ if $X_I$ and $X_J$ are equivalent and {\em equivalence classes} will be denoted by
\begin{equation*}
\mathcal{E}_X(\ell_1,\dots,\ell_m):=\{X_I \ \vert \ \Delta_i(X_I)=l_i, \textrm{ for }i=1,\dots,m\}.
\end{equation*}We say that the components $x_I$ and $x_J$ are {\em equivalent} if the corresponding brackets $X_I$ and $X_J$ are equivalent and {\em equivalent classes for components} are defined as follows,
\begin{equation*}
\mathcal{E}_x(\ell_1,\dots,\ell_m):=\{x_I\ \vert X_I\in\mathcal{E}_X(\ell_1,\dots,\ell_m) \}.
\end{equation*}
\end{de}%
\begin{rem}\label{rem-non-separate}
We will see in the following subsections that the frequencies occurring in the dynamics of $x_I$ only depend on the equivalence class of $x_I$, and not on the structure of the bracket $X_I$. Therefore, the equivalent components (in the sense of Definition \ref{Equivalence-Class}) cannot be moved separately by using sinusoids.
\end{rem}
\begin{de}[\emph{Ordering of equivalence classes}]\label{Order-equivalence-class}
Let $\mathcal{E}_x(\ell_1,\dots,\ell_m)$ and $\mathcal{E}_x(\tilde{\ell}_1,\dots,\tilde{\ell}_m)$ be two equivalence classes. $\mathcal{E}_x(\ell_1,\dots,\ell_m)$ is said to be smaller than $\mathcal{E}_x(\tilde{\ell}_1,\dots,\tilde{\ell}_m)$ if the smallest element (in the sense of ``$\prec$") in $\mathcal{E}_x(\ell_1,\dots,\ell_m)$ is smaller than the one in $\mathcal{E}_x(\tilde{\ell}_1,\dots,\tilde{\ell}_m)$, and we write (by abuse of notation) $\mathcal{E}_x(\ell_1,\dots,\ell_m)\prec\mathcal{E}_x(\tilde{\ell}_1,\dots,\tilde{\ell}_m)$.
\end{de} 

Let $\{\mathcal{E}_x^1,\mathcal{E}_x^2,\dots,\mathcal{E}_x^{\widetilde{N}}\}$ be the partition of the set of the components of $x$ induced by Definition \ref{Equivalence-Class}.  Assume that, for every pair $(i,j)\in\{1,\dots,\widetilde{N}\}^2$ with $i<j$, one has $\mathcal{E}_x^i\prec\mathcal{E}_x^j$. Our control strategy consists in displacing these equivalence classes one after another according to the ordering ``$\prec$" by using sinusoidal inputs. For every $j=1,\dots,\widetilde{N}$, the key point  is to determine how to construct an input $u^j$ defined on $[0,2\pi]$ such that the two following conditions are verified:
\begin{itemize}
\item[(C1)]under the action of $u^j$, every element of $\mathcal{E}_x^j$ reaches its preassigned value at $t=2\pi$;
\item[(C2)]under the action of $u^j$, for every $i<j$, every element of $\mathcal{E}_x^i$ returns at $t=2\pi$ to its value taken at $t=0$.
\end{itemize}

Once one knows how to construct an input $u^j$ verifying (C1) and (C2) for every $j=1,\dots,\widetilde{N}$, it suffices to {\em concatenate} them to control the complete system.

\begin{de}[\emph{Concatenation}]\label{de:conca}
The concatenation of $u^1,\dots, u^{\widetilde{N}}$ is defined on the interval $[0,2\widetilde{N}\pi]$ by
\begin{equation}\label{eq:conca}
u^1\ast\dots\ast u^{\widetilde{N}}(t):=u^{j}(t-2(j-1)\pi),
\end{equation} for $t\in[2(j-1)\pi,2j\pi]$ and $j\in\{1,\dots,\widetilde{N}\}$.
\end{de}

\begin{rem}\label{rem-concatenation}
As we will show later (see Remark \ref{rem:smooth-control}), for every positive integer $k$, it is possible to make $C^k$ concatenations such that the inputs are globally of class $C^k$ and the corresponding trajectories are not only piecewise smooth, but also globally of class $C^{k+1}$.
\end{rem} 

\subsection{Choice of frequencies}\label{control-by-sin}
In this section, we fix an equivalence class $\mathcal{E}_x^j$. We choose frequencies in $u^j$ such that Conditions (C1) and (C2) are verified. For sake of clarity, we first treat the case $m=2$ in Subsections \ref{special_case} and \ref{general_case}., and we show, in Subsection \ref{more-general}, how to adapt the method to greater values of $m$.

\subsubsection{A simple case: $m=2$ and Card $(\mathcal{E}_x^j)=1$}\label{special_case}

Let $x_I$ be the only element of $\mathcal{E}_x^j$, and $X_I$ the corresponding bracket. Let $m_1:=\Delta_1(X_I)$, and $m_2:=\Delta_2(X_I)$.

\begin{prop}\label{main1}
Consider three positive integers $\omega_1$, $\omega_2$, $\omega_3$, and $\varepsilon\in\{0,1\}$ such that
\begin{equation}\label{RM-Equation}
\left\lbrace \begin{array}{l}
        \omega_3=m_1\omega_1+(m_2-1)\omega_2,\\
        \varepsilon=m_1+m_2-1\pmod{2},
\end{array}\right.
\end{equation} and
\begin{equation}\label{NRC-Equation}
\omega_2>(m_1+m_2)m_1.
\end{equation}By choosing properly $\zeta$, the control
\begin{equation}\label{controls}
u_1=\cos \omega_1t,~ u_2=\cos\omega_2t+\zeta\cos(\omega_3t-\varepsilon\frac{\pi}{2}),
\end{equation} steers, during $[0,2\pi]$, the component $x_I$ from any initial value to any preassigned final value without modifying any component $x_J$, with $J\prec I$. Moreover, $x_I(2\pi)-x_I(0)$ gives rise to a non zero linear function of $\zeta$, where $\zeta$ is the coefficient in front of $\cos(\omega_3t-\varepsilon\frac{\pi}{2})$ in Eq. (\ref{controls}).
\end{prop}

\noindent The key point is to understand the frequencies occurring in the dynamics $\dot{x}_{I}$.

\begin{lemma}\label{F}
For $J\leq I$, the dynamics  $\dot{x}_{J}$ is a linear combination of cosine functions of the form
\begin{equation}\label{cos_F}
\cos\{(\ell_1\omega_1+\ell_2\omega_2+\ell_3\omega_3)t-(\ell_3\varepsilon+\ell_1+\ell_2+\ell_3-1)\frac{\pi}{2}\},
\end{equation}
where $\ell_1,\ell_2,\ell_3\in\mathbb{Z}$ satisfy
$\vert \ell_1\vert\leq m_1$, $\vert \ell_2\vert+\vert \ell_3\vert\leq m_2$.

In particular, the term $$\cos[(m_1\omega_1+(m_2-1)\omega_2-\omega_3)t-(-\varepsilon+m_1+m_2-1)\frac{\pi}{2}]$$
occurs in $\dot{x}_{I}$ with a zero coefficient depending linearly on $\zeta$.
\end{lemma}

\begin{proof}[Proof of Lemma \ref{F}]
The proof goes by induction on $\vert J\vert$.
\begin{itemize}
        \item $\vert J\vert=1$,  the result is true since $\dot{x}_{I_1}=u_1$ and $\dot{x}_{I_2}=u_2$.
        \item {\it Inductive step:}\\
        Assume that the result holds true for all $\tilde{J}$ such that $\vert\tilde{J}\vert<s$. We show that it remains true for $J$ such that $\vert J\vert=s$.

        By construction, we have $X_{J}=\textrm{ad}^k_{X_{J_1}}X_{J_2}$ with $\vert J_1\vert<s$ and $\vert J_2\vert<s$. Then,
        \begin{equation}\label{Chen-Fliess}
        \dot{x}_{J}=\frac{1}{k!}x_{J_1}^k\dot{x}_{J_2},
        \end{equation}
        $\dot{x}_{J_2}$ is given by the inductive hypothesis and $x_{I_1}$ is obtained by integration of Eq. $(\ref{cos_F})$. By using product formulas for cosine function, the result still holds true for $J$ of length $s$. This ends the proof of Lemma \ref{F}.
\end{itemize}
\end{proof}

\begin{proof}[Proof of Proposition $\ref{main1}$]
First note that integrating between $0$ and $2\pi$ a function of the form $\cos(\gamma t+\bar{\gamma}\frac{\pi}{2})$ with $(\gamma,\bar{\gamma})\in\mathbb{N}^2$ almost always gives $0$ except for $\gamma=0$ and $\bar{\gamma}=0\pmod {2}$  at the same time. Therefore, in order to obtain a non trivial contribution for $x_{I}$, $\dot{x}_{I}$ must contain some cosine functions verifying the following condition
\begin{equation}\label{RC}
\left\lbrace \begin{array}{l}
        \ell_1\omega_1+\ell_2\omega_2+\ell_3\omega_3=0,\\
        \ell_3\varepsilon+\ell_1+m_2+\ell_3-1\equiv 0\pmod{2},
\end{array}\right.
\end{equation}and this condition must not be satisfied by $J\prec I$ in order to avoid a change in the component $x_{J}$.

\noindent Under conditions $(\ref{RM-Equation})$ and $(\ref{NRC-Equation})$, we claim that
\begin{itemize}
\item[\bf (1)]\label{RM} $(m_1,m_2-1,-1,\varepsilon)$ is the only $4$-tuple verifying (\ref{RC}) for $x_{I}$, and $x_{I}(2\pi)-x_{I}(0)$ is a non zero {\it linear} function of $\zeta$;
\item[\bf (2)]\label{NRC}Eq. $(\ref{RC})$ is never satisfied for
$x_{J}$ with $J<I$.
\end{itemize}

\noindent Indeed, consider $(\ell_1,\ell_2,\ell_3)\in\mathbb{Z}^3$ verifying
$\vert \ell_1\vert\leq m_1$, $\vert \ell_2\vert+\vert \ell_3\vert\leq m_2$. One has
\begin{eqnarray}\label{Eq1}
\ell_1\omega_1+\ell_2\omega_2+\ell_3\omega_3&=&\ell_1\omega_1+\ell_2\omega_2+\ell_3((m_2-1)\omega_2+m_1\omega_1)\nonumber\\
&=&(\ell_3(m_2-1)+\ell_2)\omega_2+(\ell_1+\ell_3m_1)\omega_1.
\end{eqnarray}

\noindent Assume that $\omega_2>(m_1+m_2)m_1\omega_1$. Then, except for the $4-$tuple $(m_1,m_2,m_3,\varepsilon)$ verifying Eq. $(\ref{RM-Equation})$, the only possibility to have the right-hand side of Eq. $(\ref{Eq1})$ equal to $0$ is $\ell_1=\ell_2=\ell_3=0$. In that case, $$\ell_1+\ell_2+\ell_3\neq 1\pmod{2}.$$ Then, Eq. $(\ref{RC})$ is not satisfied, and {\bf (2)} is proved.

Due to Eq. $(\ref{controls})$, the power of $\zeta$ is equal to the number of times $\omega_3$ occurs in the resonance condition $(\ref{RM-Equation})$. The latter is clearly equal to $1$. Thus, $x_I(2\pi)-x_I(0)$ gives rise to a linear function of $\zeta$. It remains to show that the coefficient in front of $\zeta$ is not equal to zero. By Lemma \ref{F}, one knows that
\begin{eqnarray}
\dot{x}_I&=&g_I\cos\{(m_1\omega_1+m_2\omega_2)t-(m_1+m_2-1)\frac{\pi}{2}\}\\
&&+f_Ia\cos\{(m_1\omega_1+(m_2-1)\omega_2-\omega_3)t-(m_1+m_2-1-\varepsilon)\frac{\pi}{2}\}+\mathcal{R}\nonumber,
\end{eqnarray}
where we gathered all other terms into $\mathcal{R}$. Note that the numerical coefficients $f_I$ and $g_I$ depend on the frequencies $\omega_1$, $\omega_2$, and $\omega_3$. The goal is to show that $f_I$ is not equal to zero if we want to move the component $x_I$, i.e., when $\omega_3=(m_2-1)\omega_2+m_1\omega_1$. If we consider $f_I$ as a function of $\omega_1$, $\omega_2$, and
$\omega_3$, it suffices to show that this function is not identically equal to zero over the hyperplane of $\mathbb{R}^3$ defined by the resonance condition
$\omega_3=(m_2-1)\omega_2+m_1\omega_1$. We assume that the next lemma holds true, and we will provide an argument immediately after finishing the proof of Proposition \ref{main1}.

\begin{lemma}\label{alpha}
For all $J\leq I$, let $m_1^J:=\Delta_1(X_J)$ and $m_2^J:=\Delta_2(X_J)$.
If $f_J$ is the coefficient in front of the term $\cos\{(m_1^J\omega_1+(m_2^J-1)\omega_2-\omega_3)t-(m_1^J+m_2^J-1-\varepsilon)\frac{\pi}{2}\}$, and $g_J$ the one in front of the term $\cos\{(m_1^J\omega_1+m_2^J\omega_2)t-(m_1^J+m_2^J-1)\frac{\pi}{2}\}$. Then, the quotient $\alpha_J:=f_J/g_J$ verifies the following inductive formula.
\begin{itemize}
\item If $X_J=X_1$, $\alpha_J=0$; If $X_J=X_2$, $\alpha_J=1$;
\item If $X_J=[X_{J_1}, X_{J_2}]$, $\alpha_J$ is defined by
\begin{equation*}
\alpha_J=\frac{m_1^{J_1}\omega_1+m_2^{J_1}\omega_2}{m_1^{J_1}\omega_1+(m_2^{J_1}-1)\omega_2-\omega_3}\alpha_{J_1}+\alpha_{J_2}.
\end{equation*}where $m_i^{J_1}=\Delta_i(X_{J_1})$ for $i=1,\ 2$.
\end{itemize}
\end{lemma} 

Let us take $\omega_3=-\omega_2$. It results from Lemma \ref{alpha} that, for every $J\leq I$, one has
\begin{eqnarray*}
\alpha_J&=&\alpha_{J_1}+\alpha_{J_2}, \ \textrm{ if }X_J=[X_{J_1},X_{J_2}].
\end{eqnarray*}Since $\alpha_1=0$ and $\alpha_2=1$, then, over the hyperplane of $\mathbb{R}^3$ defined by $\omega_3=-\omega_2$, the function $\alpha_J$ is a strictly positive number independent of $\omega_1$ and $\omega_2$.

Let us show now that $\alpha_J(\omega_1,\omega_2,\omega_3)$ is not identically equal to zero over the hyperplane of $\mathbb{R}^3$ defined by $\omega_3=m_1\omega_1+(m_2-1)\omega_2$. Let $\hat{\omega}_2:=\displaystyle-{m_1\omega_1}/{m_2}$. One has $m_1\omega_1+(m_2-1)\hat{\omega}_2=-\hat{\omega}_2$. It implies that
\begin{equation*}
\alpha_I(\omega_1,\hat{\omega}_2,m_1\omega_1+(m_2-1)\hat{\omega}_2)=\alpha_I(\omega_1,\hat{\omega}_2,-\hat{\omega}_2).
\end{equation*}

\noindent Since the function $\alpha_I(\omega_1,\omega_2,-\omega_2)$ is never equal to zero, and it coincides with the function $\alpha_I(\omega_1,\omega_2,m_1\omega_1+(m_2-1)\omega_2)$ at the point $(\omega_1,\hat{\omega}_2)$, which is not identically equal to zero. Therefore, $f_I(\omega_1,\omega_2,\omega_3)$ is not identically equal to zero over the hyperplane $\omega_3=(m_2-1)\omega_2+m_1\omega_1$. Moreover, as it is a non trivial rational function, it eventually vanishes at a finite number of integer points. Then, we obtain a non zero linear function of $\zeta$, and {\bf (1)} is now proved. Proposition $\ref{main1}$ results from {\bf (1)} and {\bf (2)}.
\end{proof}

\begin{proof}[Proof of Lemma \ref{alpha}]
The proof goes by induction on $\vert I\vert$. Since $\dot{x}_1=u_1$ and $\dot{x}_2=u_2$, by Eq. (\ref{controls}), one has $\alpha_1=0$ and $\alpha_2=1$.

\noindent Assume that $\vert J\vert\geq 2$. By construction, one has $X_J=[X_{J_1}, X_{J_2}]$ with $\vert J_1\vert\leq \vert J_2\vert<\vert J\vert$. According to the inductive hypothesis, one has
\begin{eqnarray*}
\dot{x}_{J_1}&=&g_{J_1}\cos\{(m_1^{J_1}\omega_1+m_2^{J_1}\omega_2)t-(m_1^{J_1}+m_2^{J_1}-1)\frac{\pi}{2}\}\\
&&+f_{J_1}\cos\{(m_1^{J_1}\omega_1+(m_2^{J_1}-1)\omega_2-\omega_3)t-(m_1^{J_1}+m_2^{J_1}-1-\varepsilon)\frac{\pi}{2}\}+\mathcal{R}_{J_1},\\
\dot{x}_{J_2}&=&g_{J_2}\cos\{(m_1^{J_2}\omega_1+m_2^{J_2}\omega_2)t-(m_1^{J_2}+m_2^{J_2}-1)\frac{\pi}{2}\}\\
&&+f_{J_2}\cos\{(m_1^{J_2}\omega_1+(m_2^{J_2}-1)\omega_2-\omega_3)t-(m_1^{J_2}+m_2^{J_2}-1-\varepsilon)\frac{\pi}{2}\}+\mathcal{R}_{J_2}.
\end{eqnarray*}
This implies that
\begin{eqnarray*}
\dot{x}_J&=&\left(\frac{1}{m_1^{J_1}\omega_1+m_2^{J_1}\omega_2}g_{J_1}\cos\{(m_1^{J_1}\omega_1+m_2^{J_1}\omega_2)t-(m_1^{J_1}+m_2^{J_1})\frac{\pi}{2}\}\right.\\
&&+\frac{1}{m_1^{J_1}+(m_2^{J_1}-1)\omega_2-\omega_3}f_{J_1}a\cos\{(m_1^{J_1}\omega_1+(m_2^{J_1}-1)\omega_2-\omega_3)t\\
&&\left.-(m_1^{J_1}+m_2^{J_2}-\varepsilon)\frac{\pi}{2}\}+\mathcal{R}_{J_1}\right)\\
&&\left(g_{J_2}\cos\{(m_1^{J_2}\omega_1+m_2^{J_2}\omega_2)t-(m_1^{J_2}+m_2^{J_2}-1)\frac{\pi}{2}\}\right.\\
&&+f_{J_2}a\cos\{(m_1^{J_2}\omega_1+(m_2^{J_2}-1)\omega_2-\omega_3)t\\
&&\left.-(m_1^{J_2}+m_2^{J_2}-1-\varepsilon)\frac{\pi}{2}\}+\mathcal{R}_{J_2}\right)\\
&=&\frac{1}{2}\frac{g_{J_1}g_{J_2}}{m_1^{J_1}\omega_1+m_2^{J_1}\omega_2}\cos\{(m_1^J\omega_1+m_2^J\omega_2)t-(m_1^J+m_2^J-1)\frac{\pi}{2}\}\\
&&+\frac{1}{2}\left(\frac{g_{J_1}f_{J_2}}{m_1^{J_1}\omega_1+m_2^{J_1}\omega_2}+\frac{g_{J_2}f_{J_1}}{m_1^{J_1}\omega_1+(m_2^{J_1}-1)\omega_2-\omega_3}\right)\\
&&\cos\{(m_1^J\omega_1+m_2^J\omega_2-\omega_3)t-(m_1^J+m_2^J-1-\varepsilon)\frac{\pi}{2}\}+\mathcal{R}_J\\
&=&g_J\cos\{(m_1^J\omega_1+m_2^J\omega_2)t-(m_1^J+m_2^J-1)\frac{\pi}{2}\}\\
&&+f_J\cos\{(m_1^J\omega_1+m_2^J\omega_2-\omega_3)t-(m_1^J+m_2^J-1-\varepsilon)\frac{\pi}{2}\}+\mathcal{R}_{J}.
\end{eqnarray*}
Therefore, one obtains $\alpha_J=\frac{m_1^{J_1}\omega_1+m_2^{J_1}\omega_2}{m_1^{J_1}\omega_1+(m_2^{J_1}-1)\omega_2-\omega_3}\alpha_{J_1}+\alpha_{J_2}$.
\end{proof}

\subsubsection{A more general case: $m=2$ and Card $(\mathcal{E}_x^j)>1$}\label{general_case}

In general, given a pair $(m_1,m_2)$, the equivalence class $\mathcal{E}_x(m_1,m_2)$ contains more than one element. This situation first occurs for Lie brackets of length $5$. For instance, given the pair $(3,2)$, one has both $X_{I}=[X_2,[X_1,[X_1,[X_1,X_2]]]]$ and $X_{J}=[[X_1,X_2],[X_1,[X_1,X_2]]]$. By Lemma \ref{F}, if one chooses a 4-tuple verifying the resonance condition (\ref{RM-Equation}) for $\dot{x}_{I}$, the same resonance occurs in $\dot{x}_{J}$. Such two components cannot be independently steered by using resonance. The idea is to move simultaneously these components. For instance, one can choose $(u_1,u_2)$ as follows:
\begin{eqnarray*}
u_1(t)&=&\cos \omega_1t,\\
u_2(t)&=&\cos\omega_2t+a_I\cos\omega_3t+\cos\omega_4t+a_J\cos\omega_5t,
\end{eqnarray*}
where $\omega_1=1$, $\omega_2$ is chosen according to Eq. (\ref{NRC-Equation}), $\omega_3=(m_2-1)\omega_2+m_1\omega_1,$ and $\omega_5=(m_2-1)\omega_4+m_1\omega_1,$ with $\omega_4$ large enough to guarantee Condition (C2). After explicit integration of Eq. $(\ref{cano-f5})$, one obtains
\begin{displaymath}
\begin{pmatrix} f_{I}(\omega_1,\omega_2)&f_{I}(\omega_1,\omega_4)\\ f_{J}(\omega_1,\omega_2)&f_{J}(\omega_1,\omega_4)\end{pmatrix}\begin{pmatrix} a_I\\a_J\end{pmatrix}=A\begin{pmatrix} a_I\\a_J\end{pmatrix}=\begin{pmatrix}x_{I}(2\pi)-x_{I}(0)\\x_{J}(2\pi)-x_{J}(0)\end{pmatrix},
\end{displaymath}where $f_{I}$ and $f_{J}$ are two rational functions of frequencies. Thus, the pair $(u_1,u_2)$ controls exactly and simultaneously $x_{I}$ and $x_{J}$, provided that the matrix $A$ is invertible. We generalize this strategy in the following paragraphs. Assume that $\mathcal{E}_x^j(m_1,m_2)=\{x_{I_1},\dots, x_{I_N}\}$. The main result is given next.

\begin{prop}\label{Finitude}
Consider $$\{ \omega_{11}^1,\dots,\omega_{11}^{m_1}\},\dots,\{ \omega_{1N}^1,\dots,\omega_{1N}^{m_1}\},$$$$\{ \omega_{21}^1,\dots,\omega_{21}^{m_2-1},\omega_{21}^{*}\},\dots,\{\omega_{2N}^1,\dots,\omega_{2N}^{m_2-1},\omega_{2N}^{*}\}$$ belonging to $\mathbb{N}^{m_1N}\times\mathbb{N}^{m_2N}$ such that
\begin{equation}\label{RMC}
\left\lbrace\begin{array}{l}
        \forall~ j=1\dots N,~~\displaystyle\omega_{2j}^{*}=\sum_{i=1}^{m_1}\omega_{1j}^i+\sum_{i=1}^{m_2-1}\omega_{2j}^i,\\
        \varepsilon=m_1+m_2-1 \pmod{2},\end{array}\right.
\end{equation} and
\begin{equation}\label{NRMC}
\forall j=1\dots N-1,~\left\lbrace \begin{array}{llll}
        \omega_{11}^1&\in&\mathbb{N}~;\\
        \omega_{1j}^{i+1}&>&m_1\omega_{1j}^{i}~;&i=1\dots m_1,\\
        \omega_{2j}^1&>&m_1\omega_{1j}^{m_1}~;\\
        \omega_{2j}^{i}&>&m_2\omega_{2j}^{i-1}+m_1\omega_{1j}^{m_1}~;& i=2\dots m_2-1,\\
        \omega_{1j+1}^1&>&m_2\omega_{2j}^{m_2-1}+m_1\omega_{1j}^{m_1}~.
\end{array}\right.
\end{equation}
Then, the input $u:=(u_1,u_2)$ defined by
\begin{equation}\label{CF1}
\left\lbrace\begin{array}{lll}
        u_1&:=&\displaystyle\sum_{j=1}^N\sum_{i=1}^{m_1}\cos\omega_{1j}^it,\\
        u_2&:=&\displaystyle\sum_{j=1}^N\sum_{i=1}^{m_2-1}\cos\omega_{2j}^it+a_j\cos(\omega_{2j}^{*}t-\varepsilon\frac{\pi}{2}),
        \end{array}\right.
\end{equation}
steers the components $(x_{I_1},\dots,x_{I_N})$ from an arbitrary initial condition $(x_{I_1}(0),\dots,x_{I_N}(0))$ to an arbitrary final one $(x_{I_1}(2\pi),\dots,x_{I_N}(2\pi))$, without modifying any other component having been previously moved to its final value.
\end{prop}

\noindent This result generalizes Proposition $\ref{main1}$. The proof is decomposed in two parts as follows:
\begin{itemize}
\item[Part I:] we show that, if $(\ref{NRMC})$ holds and the control functions are of the form $(\ref{CF1})$, then $(\ref{RMC})$ is the \textit{only} resonance occurring in $(\dot{x}_{I_1},\dots,\dot{x}_{I_N})$;
\item[Part II:] as the resonance gives rise to a system of linear equations on $(a_1,\dots,a_N)$, we recover the invertibility of this system by choosing  suitable frequencies in the control function $(\ref{CF1})$.
\end{itemize}

\noindent \textbf{Part I \ Frequencies and Resonance} 

Consider inputs of the form $(\ref{CF1})$. Generalizing Lemma \ref{F}, we give a general form of frequencies involved in $\dot{x}_{J}$.

\begin{lemma}\label{FG}
The dynamics $\dot{x}_{J}$ is a linear combination of cosine functions of the form
\begin{equation}\label{FG1}
        (\ell_1\cdot\omega_1+\ell_2\cdot\omega_2+\ell_2^{*}\cdot\omega_2^{*})t-(\ell_1+\ell_2+m_2^{*}-1+\ell_2^{*}\varepsilon)\frac{\pi}{2},
\end{equation}where
\begin{eqnarray*}
        &&\ell_1\cdot\omega_1=\sum_{j=1}^N\sum_{i=1}^{m_1}\ell_{1j}^{i}\omega_{1j}^i,\
        \ell_2\cdot\omega_2=\sum_{j=1}^N\sum_{i=1}^{m_2-1}\ell_{2j}^{i}\omega_{2j}^i,\
        \ell_2^{*}\cdot\omega_2^{*}=\sum_{j=1}^N\ell_{2j}^{*}\omega_{2j}^{*},\label{FG2}\\
        &&\ell_1=\sum_{j=1}^N\sum_{i=1}^{m_1}\ell_{1j}^{i},\
        \ell_2=\sum_{j=1}^N\sum_{i=1}^{m_2-1}\ell_{2j}^{i},\
        \ell_2^{*}=\sum_{j=1}^Nm_{2j}^{*},\label{FG3}
\end{eqnarray*} with $(\ell_{1j}^{i}, \ell_{2j}^{i}, \ell_{2j}^{*})\in\mathbb{Z}^3$.

Let
\begin{equation*}
\vert\ell_1\vert=\sum_{j=1}^N\sum_{i=1}^{m_1}\vert \ell_{1j}^{i}\vert, \ \vert\ell_2\vert=\sum_{j=1}^N\sum_{i=1}^{m_2-1}\vert \ell_{2j}^{i}\vert, \ \textrm{and } \ \vert\ell_2^{*}\vert=\sum_{j=1}^N\vert \ell_{2j}^{*}\vert,
\end{equation*}
then, one has $\vert\ell_1\vert\leq \Delta_1(X_J)$, $\vert\ell_2\vert+\vert\ell_2^{*}\vert\leq \Delta_2(X_J)$.
\end{lemma}

\begin{proof}[Proof of Lemma \ref{FG}]
The proof goes by induction on $\vert J\vert$.
\begin{itemize}
        \item $\vert J\vert=1$: the result is true
        since $\dot{x}_{1}=u_1$ and $\dot{x}_{2}=u_2$.
        \item {\it Inductive step:}\\
        Assume that the result holds true for all $x_{\tilde{J}}$ such that $1\leq\vert \tilde{J}\vert<s$. We show that it remains true for $x_J$ with $\vert J\vert=s$. By construction, we have $X_{J}=\textrm{ad}^k_{X_{J_1}}X_{J_2}$, and
        \begin{equation}
        \dot{x}_{J}=\frac{1}{k!}x_{J_1}^k\dot{x}_{J_2},
        \end{equation}
        with $\vert J_1\vert<\vert J\vert$, $\vert J_2\vert<\vert J\vert$, and $k\vert J_1\vert+\vert J_2\vert=\vert J\vert$.

        Then, by the inductive hypothesis, we have
        \begin{eqnarray}
                \dot{x}_{J_1}&=&\textrm{LinCom}\left\{\cos\{(\ell_1\cdot\omega_1+\ell_2\cdot\omega_2+\ell_2^{*}\cdot\omega_2^{*})t-(\ell_1+\ell_2+\ell_2^{*}-1+\ell_2^{*}\varepsilon)\frac{\pi}{2}\}\right\},\label{j1}\\
                \dot{x}_{J_2}&=&\textrm{LinCom}\left\{\cos\{(\tilde{\ell}_1\cdot\omega_1+\tilde{\ell}_2\cdot\omega_2+\tilde{\ell}_2^{*}\cdot\omega_2^{*})t-(\tilde{\ell}_1+\tilde{\ell}_2+\tilde{\ell}_2^{*}-1+\tilde{\ell}_2^{*}\varepsilon)\frac{\pi}{2}\}\right\},\label{j2}
        \end{eqnarray}where $\textrm{LinCom}\{\cdot\}$ stands ``linear combination".

        Eq. $(\ref{j1})$ implies that
        \begin{eqnarray}
                x_{J_1}&=&\textrm{LinCom}\left\{\cos\{(\ell_1\cdot\omega_1+\ell_2\cdot\omega_2+\ell_2^{*}\cdot\omega_2^{*})t-(\ell_1+\ell_2+\ell_2^{*}-1+\ell_2^{*}\varepsilon)\frac{\pi}{2}-\frac{\pi}{2}\}\right\}\nonumber\\
                &=&\textrm{LinCom}\left\{\cos\{(\ell_1\cdot\omega_1+\ell_2\cdot\omega_2+\ell_2^{*}\cdot\omega_2^{*})t-(\ell_1+\ell_2+\ell_2^{*}+\ell_2^{*}\varepsilon)\frac{\pi}{2}\}\right\}.\label{J1}
        \end{eqnarray}
        For notational ease, we will only write down the case $\dot{x}_J=x_{J_1}\dot{x}_{J_2}.$ Using product formulas for cosine function, one has
        \begin{eqnarray}
                \dot{x}_{J}&=&\textrm{LinCom}\left\{\cos\{[(\ell_1\pm\tilde{\ell}_1)\cdot\omega_1+(\ell_2\pm\tilde{\ell}_2)\cdot\omega_2+(\ell_2^{*}\pm\tilde{\ell}_2^{*})\cdot\omega_2^{*}]t\right.\nonumber\\
                &&\left.-[(\ell_1\pm\tilde{\ell}_1)+(\ell_2\pm\tilde{\ell}_2)+(\ell_2^{*}\pm\tilde{\ell}_2^{*})-1+(\ell_2^{*}\pm\tilde{\ell}_2^{*})\varepsilon]\frac{\pi}{2}\}\right\}.
        \end{eqnarray}

        Moreover, according to the inductive hypothesis, one has $$\vert\ell_1\vert\leq \Delta_1(X_{J_1}),\ \vert\ell_2\vert+\vert\ell_2^{*}\vert\leq \Delta_2(X_{J_1}),$$ and $$\vert\tilde{\ell}_1\vert\leq \Delta_1(X_{J_2}),\ \vert\tilde{\ell}_2\vert+\vert\tilde{\ell}_2^{*}\vert\leq \Delta_2(X_{J_2}).$$ Then, one gets
        $$\vert\tilde{\ell}_1\pm\tilde{\ell}_1\vert\leq \Delta_1(X_J), \textrm{ and } \vert\tilde{\ell}_2\pm\tilde{\ell}_2\vert+\vert\ell_2^{*}\pm\tilde{\ell}_2^{*}\vert\leq \Delta_2(X_J).$$
        This concludes the proof of Lemma \ref{FG}.
\end{itemize}
\end{proof}

By Lemma $\ref{FG}$, one gets a non trivial contribution for $x_J$ if the resonance condition
\begin{equation}\label{Resonance-Condition1}
\left\lbrace \begin{array}{l}
        \ell_1\cdot\omega_1+\ell_2\cdot\omega_2+\ell_2^{*}\cdot\omega_2^{*}=0,\\
        \ell_2^{*}\varepsilon+\ell_1+\ell_2+\ell_2^{*}-1\equiv 0\pmod{2},
\end{array}\right.
\end{equation} is verified by the frequencies of some cosine functions involved in $\dot{x}_J$.

\begin{lemma}\label{Resonance}
Under conditions $(\ref{RMC})$ and $(\ref{NRMC})$ in Proposition \ref{Finitude}, one gets a non trivial contribution on $x_{I_j}$ depending linearly on $a_j$ for all $j=1\dots,N$.
\end{lemma}

\begin{proof}[Proof of Lemma \ref{Resonance}]
It is clear that the resonance condition $(\ref{Resonance-Condition1})$ holds for $$\{ \omega_{11}^1,\dots,\omega_{11}^{m_1}\},\dots,\{ \omega_{1N}^1,\dots,\omega_{1N}^{m_1}\},$$$$\{ \omega_{21}^1,\dots,\omega_{21}^{m_2-1},\omega_{21}^{*}\},\dots,\{\omega_{2N}^1,\dots,\omega_{2N}^{m_2-1},\omega_{2N}^{*}\}, $$ and $\varepsilon\in\{0,1\}$ verifying $(\ref{RMC})$. We show that it is the only resonance occurring in $\dot{x}_{I_j}$. Indeed, by Lemma $\ref{FG}$, the integer part of frequencies in $\dot{x}_{I_j}$ is in the following form
\begin{eqnarray}
        &&\ell_1\cdot\omega_1+\ell_2\cdot\omega_2+\ell_2^{*}\cdot\omega_2^{*}\nonumber\\
        &=&\sum_{j=1}^N\sum_{i=1}^{m_1}\ell_{1j}^{i}\omega_{1j}^i+\sum_{j=1}^N\sum_{i=1}^{m_2-1}\ell_{2j}^{i}\omega_{2j}^i+\sum_{j=1}^N\ell_{2j}^{*}\omega_{2j}^{*}\nonumber\\
        &=&\sum_{j=1}^N\sum_{i=1}^{m_1}\ell_{1j}^{i}\omega_{1j}^i+\sum_{j=1}^N\sum_{i=1}^{m_2-1}\ell_{2j}^{i}\omega_{2j}^i+\sum_{j=1}^N\ell_{2j}^{*}\left(\sum_{i=1}^{m_1}\omega_{1j}^i+\sum_{i=1}^{m_2-1}\omega_{2j}^i\right)\nonumber\\
        &=&\sum_{j=1}^N\sum_{i=1}^{m_1}(\ell_{1j}^{i}+\ell_{2j}^{*})\omega_{1j}^i+\sum_{j=1}^N\sum_{i=1}^{m_2-1}(\ell_{2j}^{i}+\ell_{2j}^{*})\omega_{2j}^i.\label{Resonance1}
\end{eqnarray}
By Condition $(\ref{NRMC})$, Eq. $(\ref{Resonance1})$ is equal to zero if and only if
\begin{eqnarray*}
        \ell_{1j}^{i}+\ell_{2j}^{*}&=&0,\ \textrm{ for } i=1,\dots,m_1,\\
        \ell_{2j}^{i}+\ell_{2j}^{*}&=&0,\ \textrm{ for } i=1,\dots,m_2-1.
\end{eqnarray*}
Then, one has
\begin{eqnarray*}
\vert\ell_1\vert&=&\sum_{j=1}^N\sum_{i=1}^{m_1}\vert \ell_{1j}^{i}\vert=\sum_{j=1}^N\sum_{i=1}^{m_1}\vert \ell_{2j}^{*}\vert=m_1\sum_{j=1}^N\vert \ell_{2j}^{*}\vert,\\
\vert\ell_2\vert&=&\sum_{j=1}^N\sum_{i=1}^{m_2-1}\vert \ell_{2j}^{i}\vert=\sum_{j=1}^N\sum_{i=1}^{m_2-1}\vert \ell_{2j}^{*}\vert=(m_2-1)\sum_{j=1}^N\vert \ell_{2j}^{*}\vert.
\end{eqnarray*}
However, by Lemma $\ref{FG}$, one knows that $\vert\ell_1\vert\leq m_1$ and $\vert\ell_2\vert+\vert \ell_2^{*}\vert\leq m_2$. Then, one necessarily has $m_{2j}^*=0$ for all $j=1,\dots,N$. In that case, one obtains $$\ell_2^{*}\varepsilon+\ell_1+\ell_2+\ell_2^{*}-1=-1\neq 0\pmod{2}.$$ In conclusion, the resonance condition $(\ref{Resonance-Condition1})$ does not hold for any $4-$tuple $(\ell_1,\ell_2,\ell_2^*,\varepsilon)$ different from $(m_1,m_2-1,-1,m_1+m_2-1\pmod 2)$.

By Eq. $(\ref{CF1})$, the power of $a_j$ is equal to the number of times $\omega_{2j}^*$ occurs in the resonance condition $(\ref{RM-Equation})$. Since the latter is equal to $1$, we obtain a linear function of $a_j$. This ends the proof of Lemma \ref{Resonance}.
\end{proof}

\begin{lemma}\label{Non-Resonance}
If $x_J\in\mathcal{E}^i_x$ and $i<j$, then $x_{J}(2\pi)-x_J(2\pi)=0$.
\end{lemma}

\begin{proof}[Proof of Lemma \ref{Non-Resonance}]
We first note that Eq. $(\ref{Resonance1})$ still holds true. Recall its expression here.
\begin{eqnarray}
        \ell_1\cdot\omega_1+\ell_2\cdot\omega_2+\ell_2^{*}\cdot\omega_2^{*}
        =\sum_{j=1}^N\sum_{i=1}^{m_1}(\ell_{1j}^{i}+\ell_{2j}^{*})\cdot\omega_{1j}^i+\sum_{j=1}^N\sum_{i=1}^{m_2-1}(\ell_{2j}^{i}+\ell_{2j}^{*})\cdot\omega_{2j}^i\label{Non-Resonance1}
\end{eqnarray}
By condition $(\ref{NRMC})$ in Proposition $\ref{Finitude}$, Eq. $(\ref{Non-Resonance1})$ is equal to zero if and only if $\ell_{1j}^{i}+\ell_{2j}^{*}=0$ for $i=1,\dots,m_1$, $j=1,\dots,N$ and $\ell_{2j}^{i}+\ell_{2j}^{*}=0$ for $i=1,\dots,m_2-1$, $j=1,\dots,N$. In that case, one has
\begin{eqnarray*}
\vert\ell_1\vert=m_1\sum_{j=1}^N\vert \ell_{2j}^{*}\vert,\qquad \vert\ell_2\vert+\vert \ell_2^{*}\vert=m_2\sum_{j=1}^N\vert \ell_{2j}^{*}\vert.
\end{eqnarray*}
One also knows that $\vert \ell_1\vert\leq \Delta_1(X_J),\quad\vert\ell_2\vert+\vert \ell_2^{*}\vert\leq\Delta_2(X_J),$ with $\Delta_1(X_J)<m_1$ or $\Delta_2(X_J)<m_2$. Therefore, one has $\ell_{2j}^*=0$ for all $j=1,\dots,N$. This implies that
$$\ell_2^{*}\varepsilon+\ell_1+\ell_2+\ell_2^{*}-1=-1\neq 0\pmod{2}.$$

In conclusion, the resonance condition $(\ref{Resonance-Condition1})$ does not hold true. This ends the proof of Lemma \ref{Non-Resonance}.
\end{proof}

\noindent\textbf{Part II \ Invertibility} 

As a consequence of Lemma $\ref{Resonance}$, one has
\begin{eqnarray}
&&\begin{pmatrix}
x_{I_1}(2\pi)-x_{I_1}(0)\\
\vdots\\
x_{I_N}(2\pi)-x_{I_N}(0)
\end{pmatrix}=A(\omega_{11}^1,\dots,\omega_{2N}^{m_2-1},\omega_{2N}^{*})
\begin{pmatrix}
a_1\\
\vdots\\
a_N
\end{pmatrix}\label{Systeme_Inverse}\\
&=&\begin{pmatrix}
f_{I_1}^X(\omega_{11}^1,\dots,\omega_{21}^{*}),&\dots&,f_{I_1}^X(\omega_{1N}^1,\dots,\omega_{2N}^{*})\\
\vdots&\ddots&\vdots\\
f_{I_N}^X(\omega_{11}^1,\dots,\omega_{21}^{*}),&\dots&,f_{I_N}^X(\omega_{1N}^1,\dots,\omega_{2N}^{*})
\end{pmatrix}
\begin{pmatrix}
a_1\\
\vdots\\
a_N
\end{pmatrix},\nonumber
\end{eqnarray} where  $f_{I_j}^X:~\mathbb{R}^{m}\rightarrow \mathbb{R}$ are rational functions of frequencies, and every $\omega_{2j}^{*}$ verifies Eq. $(\ref{RMC})$ for $j=1,\dots,N$.

\begin{de}[\emph{Control matrix and control vector}]\label{de:control-matrix}
The matrix $A$ and the vector $(a_1,\dots,a_N)$ occurring in Eq. (\ref{Systeme_Inverse}) are called respectively \emph{control matrix} and \emph{control vector} associated with the equivalence class $\mathcal{E}_x^j$.
\end{de}

We show in the sequel that it is possible to choose integer frequencies $$\{ \omega_{11}^1,\dots,\omega_{11}^{m_1}\},\dots,\{ \omega_{1N}^1,\dots,\omega_{1N}^{m_1}\},$$$$\{ \omega_{21}^1,\dots,\omega_{21}^{m_2-1},\omega_{21}^{*}\},\dots,\{\omega_{2N}^1,\dots,\omega_{2N}^{m_2-1},\omega_{2N}^{*}\},$$ so that the invertibility of the control matrix $A$ is guaranteed, as well as the non-resonance of every component $x_J$ belonging to a class smaller than $\mathcal{E}_x^j$.

For $j=1,\dots, N$, we use $P_j$ to denote the hyperplane in $\mathbb{R}^{M}$ with $M:=m_1+m_2$ defined by Eq. $(\ref{RMC})$, for which we recall the expression next,
$$\omega_{2j}^*=\sum_{i=1}^{m_1}\omega_{1j}^i+\sum_{i=1}^{m_2-1}\omega_{2j}^i.$$
We start by showing that the function $\textrm{det} A(\omega_{11}^1,\dots,\omega_{2N}^{*})$ is not identically equal to zero on $\displaystyle\cap_{j=1}^{N}P_j$. This is a consequence of the following lemma.

\begin{lemma}\label{Free_Family}
The family of functions
$$\{f_{I_1}^X(\omega_1^1,\dots,\omega_2^{m_2-1},\omega_2^{*}),\dots, f_{I_N}^X(\omega_1^1,\dots,\omega_2^{m_2-1},\omega_2^{*})\}$$ is linearly independent on the hyperplane $P$ in $\mathbb{R}^M$ defined by the equation $$\omega_{2}^{*}=\sum_{i=1}^{m_1}\omega_{1}^i+\sum_{i=1}^{m_2-1}\omega_{2}^i.$$
\end{lemma}

\begin{proof}[Proof of Lemma \ref{Free_Family}]
The first part of the argument consists in considering a family of $M$ indeterminates $Y=\{Y_1,\dots,Y_M\}$ and the associated control system
\begin{equation}\label{CSM}
        \dot{y}=\sum_{i=1}^M v_iY_i.
\end{equation}Let $H_Y$ be a \textit{P. Hall family} over $Y$. Consider the elements $\{Y_{J_1},\dots,Y_{J_{\tilde{N}}}\}$ in $H_Y$ of length $M$ such that $\Delta_i(Y_{J_j})=1$, for $i=1\dots M$, and $j=1,\dots,\tilde{N}$, and the corresponding components $\{y_{J_1},\dots,y_{J_{\tilde{N}}}\}$ in exponential coordinates.

If we apply one control of the form $\{v_i=\cos\nu_it\}_{i=1\dots M}$, with $\nu_m=\sum_{i=1}^{m-1}\nu_i$, to System $(\ref{CSM})$, then, by explicit integration, there exists, for each component $y_{J_j}$, a fractional function $f_{J_j}^Y:\mathbb{R}^m\rightarrow\mathbb{R}$ such that
\begin{equation}\label{D}
y_{J_j}(2\pi)-y_{J_j}(0)=f_{J_j}^Y(\nu_1,\dots,\nu_M),\ \ \textrm{ for } \nu_M=\sum_{i=1}^{M-1}\nu_i.
\end{equation}

\begin{claim}\label{Free_D}
The family of functions $\{f_{J_1}^Y,\dots,f_{I_{\tilde{N}}}^Y\}$ is linearly independent on the hyperplane in $\mathbb{R}^M$ defined by $ \nu_M=\sum_{i=1}^{M-1}\nu_i$.
\end{claim}

\begin{proof}[Proof of Claim \ref{Free_D}]

We first define $\tilde{f}_{J_j}^Y$ by
\begin{equation}\label{dd}
        \tilde{f}_{J_j}^Y(\nu_1,\dots,\nu_M)=f_{J_j}^Y(\nu_1,\dots,-\nu_M).
\end{equation}
Then, it is easy to see that $\tilde{f}_{J_j}^Y$ verifies the following inductive formula:
\begin{enumerate}
\item for $J=i=1\dots M$, $\tilde{f}_{J}^Y(\nu_i)=\frac{1}{\nu_i}$;
\item for $\vert J\vert>1$, $Y_{J}=[Y_{J_1},Y_{J_2}]$, there exists an injective function $$\sigma_{J}: \{1,\dots,m^{J}\}\rightarrow\{1,\dots, M\}$$ such that
\begin{eqnarray}\label{DD}
\tilde{f}_{J}^Y(\nu_{\sigma_{J}(1)},\dots,\nu_{\sigma_{J}(m^{J})})
=\frac{\tilde{f}_{J_1}^Y(\nu_{\sigma_{J}(1)},\dots,\nu_{\sigma_{J}(m^{{J_1}})})}{\sum_{i=1}^{m^{J_1}} \nu_{\sigma_{J}(i)}}\tilde{f}_{{J_2}}^Y(\nu_{\sigma_{J}(m^{{J_1}}+1)},\dots, \nu_{\sigma_{J}(m^{J})}),
\end{eqnarray}where $m^{J}:=\Delta(Y_{J})$, $m^{{J_1}}:=\Delta(Y_{{J_1}})$, and $m^{{J_2}}:=\Delta(Y_{{J_2}})$.
\end{enumerate}

\noindent We note that the family of rational functions $\tilde{f}^Y_J$ is well defined for all the Lie brackets $Y_J$ such that $\Delta_i(Y_J)\leq 1$, $i=1,\dots,M$. The algebraic construction could be extended to all the Lie brackets, but it is not necessary for our purpose. We also note that Claim \ref{Free_D} is equivalent to the fact that the family of rational functions $$\{\tilde{f}_{J_j}^Y(\nu_1,\dots,\nu_M)\}_{j=1,\dots,\tilde{N}}$$ is linearly independent over the hyperplane $\sum_{i=1}^M\nu_i=0$.

Recall that every element $Y_{J_j}$ in the family $\{Y_{J_1},\dots,Y_{J_{\tilde{N}}}\}$ writes uniquely as
\begin{equation}\label{PH}
        Y_{J_j}=[Y_{J_{j_1}},Y_{J_{j_2}}].
\end{equation}
\begin{de}[\emph{Left and right factors}]
For $J\in\{J_1,\dots, J_{\tilde{N}}\}$, the {\em left factor} $L(J)$ and the {\em right factor} $R(J)$ of $J$ are defined in such a way that $Y_J=[Y_{L(J)},Y_{R(J)}]$.
\end{de}
\noindent Let $L^{*}$ be defined by
\begin{equation}\label{LF}
        L^*:=\max_{j=1,\dots,\tilde{N}}\{L(J_j)\}.
\end{equation} The integer $L^*$ is well defined since a P. Hall family is totally ordered. Thus, there exists $J^*\in\{J_1,\dots,J_{\tilde{N}}\}$ such that $L^*=L({J^*})$. Then, define $R^*:=R({J^*})$ and set $m^{*}=\vert L^*\vert$. Let
$$\Lambda=\Lambda_L\cup \Lambda_R\hbox{ and }\bar{\Lambda}=\{1,\dots, \tilde{N}\}\setminus \Lambda,$$
with $\Lambda_L$ and $\Lambda_R$ defined by
\begin{eqnarray}
\Lambda_L&:=&\{j\in\{1,\dots,\tilde{N}\},\ \textrm{ such that } Y_{L(J_j)}\sim Y_L\},\\
\Lambda_R&:=&\{j\in\{1,\dots,\tilde{N}\},\ \textrm{ such that } Y_{L(J_j)}\sim Y_R\}.
\end{eqnarray}
Then, for all $j\in\Lambda$, there exists an injection function $$\sigma_j:\{1,\dots, M\}\rightarrow\{1,\dots, M\}$$ such that one has
\begin{eqnarray}
\tilde{f}_{J_j}^Y(\nu_1,\dots,\nu_M)
&=&\frac{\tilde{f}_{L(J_{j})}^Y(\nu_{\sigma_j(1)},\dots,\nu_{\sigma_j(m^*)})}{\sum_{i=1}^{m^{*}} \nu_{\sigma_j(i)}}\tilde{f}_{R(J_{j})}^Y(\nu_{\sigma_j(m^{*}+1)},\dots, \nu_{\sigma_j(M)}),\ \textrm{ if } j\in \Lambda_L ,\label{DD1}\\
\tilde{f}_{J_j}^Y(\nu_1,\dots,\nu_M)&=&\frac{\tilde{f}_{L(J_{j})}^Y(\nu_{\sigma_j(m^*+1)},\dots,\nu_{\sigma_j(M)})}{\sum_{i=m^*+1}^{M} \nu_{\sigma_j(i)}}\tilde{f}_{R(J_{j})}^Y(\nu_{\sigma_j(1)},\dots, \nu_{\sigma(m^*)}),\ \textrm{ if } j\in \Lambda_R. \label{DD2}
\end{eqnarray}

Note that, for all $j_1$ and $j_2$ in $\Lambda_L$, one has
$\{\nu_{\sigma_{j_1}(1)},\dots, \nu_{\sigma_{j_1}(m^*)}\}=\{\nu_{\sigma_{j_2}(1)},\dots, \nu_{\sigma_{j_2}(m^*)}\}.$ Denote by $\Xi_L$ the set of variables involved in $\tilde{f}^Y_{L(J_j)}$ with $j\in\Lambda_L$. A similar property holds for $\Lambda_R$. For all $j_1$ and $j_2$ in $\Lambda_R$, one has
$\{\nu_{\sigma_{j_1}(m^*+1)},\dots, \nu_{\sigma_{j_1}(M)}\}=\{\nu_{\sigma_{j_2}(m^*+1)},\dots, \nu_{\sigma_{j_2}(M)}\}.$ Denote by $\Xi_R$ the set of all variables occurring in $\tilde{f}^Y_{L(J_j)}$ with $j\in\Lambda_R$. Then one has $\Xi_L\cup\Xi_R=\{\nu_1,\dots,\nu_M\}.$ By abuse of notation, we re-write Eqs. $(\ref{DD1})$ and $(\ref{DD2})$ in the following form:
\begin{eqnarray}
\tilde{f}_{J_j}^Y(\nu_1,\dots,\nu_M)&=&\frac{\tilde{f}_{L(J_{j})}^Y(\Xi_L)}{\sum_{\tilde{\nu}_k\in\Xi_L}\tilde{\nu}_{k}}\tilde{f}_{R(J_{j})}^Y(\Xi_R),\ \textrm{ if } j\in \Lambda_L \label{DD1_bis};\\
\tilde{f}_{J_j}^Y(\nu_1,\dots,\nu_M)&=&\frac{\tilde{f}_{L(J_{j})}^Y(\Xi_R)}{\sum_{\tilde{\nu}_k\in\Xi_R} \tilde{\nu}_k}\tilde{f}_{R(J_{j})}^Y(\Xi_L),\ \textrm{ if } j\in \Lambda_R. \label{DD2_bis}
\end{eqnarray}Moreover, by the resonance condition $\sum_{i=1}^M\nu_i=0$, Eq. $(\ref{DD2_bis})$ becomes
\begin{equation}
\tilde{f}_{J_j}^Y(\nu_1,\dots,\nu_M)=\frac{\tilde{f}_{L(J_{j})}^Y(\Xi_R)}{-\sum_{\tilde{\nu}_k\in\Xi_L} \tilde{\nu}_k}\tilde{f}_{R(J_{j})}^Y(\Xi_L),\ \textrm{ if } j\in \Lambda_R. \label{DD2_bbis}
\end{equation}

We now prove that the family of rational functions $\displaystyle\{\tilde{f}_{J_j}^Y(\nu_1,\dots,\nu_M)\}_{j=1,\dots,\tilde{N}}$ is linearly independent over the hyperplane $\sum_{i=1}^M\nu_i=0$. The proof goes by induction over the length of the Lie brackets under consideration. For the brackets of length $1$, the result is obviously true. Assume that the result holds for all brackets of length smaller than $M-1$, $M\geq 2$.

Assume that there exist $\ell_j\in\mathbb{R}^{\tilde{N}}$ such that
\begin{equation}\label{LR}
        \sum_{j=1}^{\tilde{N}} \ell_j\tilde{f}_{J_j}^Y(\nu_1,\dots,\nu_M)=0,\ \textrm{ with } \sum_{i=1}^M\nu_i=0.
\end{equation}
One has
\begin{eqnarray}
        &&\sum_{j=1}^{\tilde{N}} \ell_j\tilde{f}_{J_j}^Y(\nu_1,\dots,\nu_M)
        =\sum_{j\in \Lambda} \ell_j\tilde{f}_{J_j}^Y(\nu_1,\dots,\nu_M)+\sum_{j\in\bar{\Lambda}} \ell_j\tilde{f}_{J_j}^Y(\nu_1,\dots,\nu_M)\label{LR1}\\
        &=&\sum_{j\in\Lambda_L}\ell_j\frac{\tilde{f}_{L(J_{j})}^Y(\Xi_L)}{\sum_{\tilde{\nu}_k\in\Xi_L}\tilde{\nu}_{k}}\tilde{f}_{R(J_{j})}^Y(\Xi_R)-\sum_{j\in\Lambda_R}\ell_j\frac{\tilde{f}_{L(J_{j})}^Y(\Xi_R)}{\sum_{\tilde{\nu}_k\in\Xi_L} \tilde{\nu}_k}\tilde{f}_{R(J_{j})}^Y(\Xi_L)+\sum_{j\in\bar{\Lambda}} \ell_j\tilde{f}_{J_j}^Y(\nu_1,\dots,\nu_M)=0.\nonumber
\end{eqnarray}
Multiplying Eq. $(\ref{LR1})$ by the factor $\displaystyle\sum_{\tilde{\nu}_k\in\Xi_L}\tilde{\nu}_k$, one gets
\begin{eqnarray}
\sum_{j\in\Lambda_L}\ell_j\tilde{f}_{L(J_{j})}^Y(\Xi_L)\tilde{f}_{R(J_{j})}^Y(\Xi_R)-\sum_{j\in\Lambda_R}\ell_j{\tilde{f}_{L(J_{j})}^Y(\Xi_R)}\tilde{f}_{R(J_{j})}^Y(\Xi_L)
+(\sum_{\tilde{\nu}_k\in\Xi_L}\tilde{\nu}_{k})\sum_{j\in\bar{\Lambda}} \ell_j\tilde{f}_{J_j}^Y(\nu_1,\dots,\nu_M)=0.\label{LR2}
\end{eqnarray}

\noindent Since $L^*$ is the maximal element among the left factors of Lie brackets of length $M$, the fraction $\tilde{f}_{J_j}^Y$ does not contain the factor $\displaystyle\sum_{\tilde{\nu}_k\in\Xi_L}\tilde{\nu}_k$ for all $j\in\bar{\Lambda}$. Therefore, on the hyperplane of $\mathbb{R}^{m^*}$ defined by $\displaystyle\sum_{\tilde{\nu}_k\in\Xi_L}\tilde{\nu}_{k}=0$, one has
\begin{equation}
\sum_{j\in\Lambda_L}\ell_j\tilde{f}_{L(J_{j})}^Y(\Xi_L)\tilde{f}_{R(J_{j})}^Y(\Xi_R)-\sum_{j\in\Lambda_R}\ell_j{\tilde{f}_{L(J_{j})}^Y(\Xi_R)}\tilde{f}_{R(J_{j})}^Y(\Xi_L)=0.\label{LR3}
\end{equation}
Fixing variables belonging to $\Xi_R$, Eq. $(\ref{LR3})$ is a linear combination of elements of the family $$\{\tilde{f}_{L(J_{j})}^Y(\Xi_L)\}_{j\in\Lambda_L}\cup\{\tilde{f}_{R(J_{j})}^Y(\Xi_L)\}_{j\in\Lambda_R}$$ associated with elements of length $m^*$ in the P. Hall family. By the inductive hypothesis, this family is linearly independent over the hyperplane of $\mathbb{R}^{m^*}$ defined by $\displaystyle\sum_{\tilde{\nu}_k\in\Xi_L}\tilde{\nu}_{k}=0$. We therefore obtain that
\begin{eqnarray}
\ell_j\tilde{f}_{R(J_{j})}^Y(\Xi_R)=0,&&\textrm{ for all }j\in\Lambda_L,\label{LR4}\\
\ell_j\tilde{f}_{L(J_{j})}^Y(\Xi_R)=0,&&\textrm{ for all }j\in\Lambda_R.\label{LR5}
\end{eqnarray}

Since Eqs. $(\ref{LR4})$ and $(\ref{LR5})$ hold true over the whole hyperplane of $\mathbb{R}^{M-m^*}$ defined by $\displaystyle\sum_{\tilde{\nu}_k\in\Xi_R}\tilde{\nu}_k=0$, one has $\ell_j=0$ for every $j\in\Lambda$. Therefore, Eq. $(\ref{LR1})$ becomes
\begin{equation}
\sum_{j\in\bar{\Lambda}} \ell_j\tilde{f}_{J_j}^Y(\nu_1,\dots,\nu_M)=0.
\end{equation}

Consider now the maximum left factor for $j\in\bar{\Lambda}$ and iterate the same reasoning used for Eq. $(\ref{LR})$. We deduce that $\ell_j=0$ for every $j\in\bar{\Lambda}$. Therefore, the family $\{\tilde{f}_{J_j}^Y(\nu_1,\dots,\nu_M)\}_{j=1,\dots,\tilde{N}}$ is linearly independent over the hyperplane $\sum_{i=1}^M\nu_i=0$ and this concludes the proof of Claim $\ref{Free_D}$.
\end{proof}

We are now in a position to proceed with the argument of Lemma \ref{Free_Family}. Let $X_{I}$ be an element of $\mathcal{E}_X(m_1,m_2)$, $M:=m_1+m_2$ and $N:=$Card $\mathcal{E}_X(m_1,m_2)$. Consider also another family of $M$ indeterminates $Y=\{Y_1,\dots,Y_M\}$ and let $H_Y$ be the \textit{P. Hall family} over $Y$. Finally, consider all the elements of the class $\mathcal{E}_Y(1,\dots,1)=\{Y_{J_1},\dots,Y_{J_{\tilde{N}}}\}$ in $H_Y$.

Let $\Pi$ be the algebra homomorphism from $L(Y)$ to $L(X)$ defined by
\begin{eqnarray}
        \Pi(Y_i)=X_1,\ &\textrm{ for }& i=1,\dots,m_1,\label{Proj1}\\
        \Pi(Y_i)=X_2,\ &\textrm{ for }& i=m_1+1,\dots,M.\label{Proj2}
\end{eqnarray}
Note that the map $\Pi$ is surjective from $\mathcal{E}_Y$ onto $\mathcal{E}_X$. Consider the following vector fields
\begin{equation*}
V_Y=\{v_1Y_1+\dots+v_MY_M\},
\end{equation*} where
\begin{equation}\label{Y-control}
v_i=\cos\omega_it,\textrm{ for }i=1\dots M-1, \textrm{ and } v_M=\cos(\omega_Mt+\varepsilon\frac{\pi}{2}),
\end{equation} with $\displaystyle\omega_M=\sum_{i=1}^{M-1}\omega_i$, and $\omega_i$ verifying the non-resonance conditions.

\noindent Then, the non-autonomous flow of $V_Y$ between $0$ and $2\pi$ is given by
\begin{eqnarray}
\overrightarrow{\exp}(V_Y)(0,2\pi)=e^{f_{J_1}^YY_{J_1}}\circ\cdots\circ e^{f_{J_{\tilde{N}}}^YY_{J_{\tilde{N}}}}\circ \prod_{J>J_1}e^{f_{J}^YY_J}.
\end{eqnarray}

\noindent Let us now apply $\Pi$ to $V_Y$, we get
\begin{equation}\label{VX}
\Pi(V_Y):=V^X=\{v_1\Pi(Y_1)+\dots+v_m\Pi(Y_m)\}=\{u_1X_1+u_2X_2\},
\end{equation} where
\begin{eqnarray}
u_1&=&\sum_{i=1}^{m_1}v_i=\sum_{i=1}^{m_1}\cos\omega_it,\\
u_2&=&\sum_{i=m_1+1}^mv_i=\sum_{i=m_1+1}^{m-1}\cos\omega_it+\cos(\omega_mt+\varepsilon\frac{\pi}{2}).
\end{eqnarray}
Then, the non-autonomous flow of $V_X$ between $0$ and $2\pi$ is given by
\begin{eqnarray}
\overrightarrow{\exp}(V_X)(0, 2\pi)=e^{f_{J_1}^Y\Pi(Y_{J_1})}\circ\cdots\circ e^{f_{J_{\tilde{N}}}^Y\Pi(Y_{J_{\tilde{N}}})}\circ \prod_{J>J_1}e^{f_{J}^Y\Pi(Y_J)}
=e^{\sum_{j=1}^{\tilde{N}}f_{J_j}^Y\Pi(Y_{J_j})}\circ\prod_{J>J_1}e^{\bar{f}_{J}^Y\Pi(Y_J)}.\label{Surj1}
\end{eqnarray}
We also know that
\begin{eqnarray}
\overrightarrow{\exp}(V_X)(0,2\pi)=e^{f_{I_1}^XX_{I_1}}\circ\cdots\circ e^{f_{I_N}^XX_{I_N}}\circ \prod_{I>I_1}e^{f_{I}^XX_I}
=e^{\sum_{j=1}^Nf_{I_j}^XX_{I_j}}\circ \prod_{I>I_1}e^{\bar{f}_{I}^XX_I}.\label{Surj2}
\end{eqnarray}

\noindent Recall that  $\Pi$ is surjective from $\mathcal{E}_Y(1,\dots,1)$ onto $\mathcal{E}_X(m_1,m_2)$. Therefore, by identifying Eqs. $(\ref{Surj1})$ and $(\ref{Surj2})$, we obtain that, for all $j=1,\dots,N$, $f_{I_j}^X$ is a linear combination of $f_{J_i}^Y$ with $i=1,\dots,\tilde{N}$, i.e.,
\begin{equation}
        f_{I_j}^X=\sum_{i=1}^{\tilde{N}}\alpha_i^jf_{J_i}^Y.
\end{equation}

\noindent Since the family $(f_{J_i}^Y)_{i=1,\dots,\tilde{N}}$ is linearly independent and the matrix ${A}:=(\alpha_i^j)_{i=1,\dots,N; j=1,\dots,\tilde{N}}$ is surjective, we conclude that the family $(f_{I_j}^X)_{j=1,\dots, N}$ is also linearly independent. This ends the proof of Lemma \ref{Free_Family}.\end{proof}

A consequence of Lemma $\ref{Free_Family}$ is the following.
\begin{coro}\label{Det}
With the above notations, the function {\em det}$A$ is not identically equal to zero on $\cap_{j=1}^{N}P_j$.
\end{coro}

\begin{proof}[Proof of Corollary \ref{Det}.]

For $j=1,\dots,N$, we define the vector $L_j$ by
$$L_j=\left(f_{I_j}^X(\omega_{11}^1,\dots,\omega_{21}^{m_2-1},\omega_{21}^{*}),\cdots,f_{I_j}^X(\omega_{1N}^1,\dots,\omega_{2N}^{m_2-1},\omega_{2N}^{*})\right)^T.$$
Assume that $\displaystyle\sum_{j=1}^N\ell_jL_j=0$ with $l_j\in\mathbb{R}$. Then, for all $i=1,\dots, N$, we have
\begin{equation}
\sum_{j=1}^N\ell_jf_{I_j}^X(\omega_{1i}^1,\dots,\omega_{1i}^{m_1},\omega_{2i}^1,\dots,\omega_{2i}^{m_2-1},\omega_{2i}^{*})=0.
\end{equation}
By Lemma \ref{Free_Family}, we have $l_j=0$ for $j=1,\dots,N$. Then, the family $(L_j)_{j=1,\dots,N}$ is linearly independent. We conclude that $\textrm{det}~A$ is not equal to zero. This ends the proof of Corollary \ref{Det}.
\end{proof}

We still need another technical lemma which guarantees that there exist integer frequencies such that Eq. $(\ref{NRMC})$ is satisfied and the matrix $A$ in Eq. $(\ref{Systeme_Inverse})$ is invertible.

\begin{lemma}\label{Integer}
There exists integer frequencies such that $(\ref{NRMC})$ is satisfied and $\textrm{det}~A$ is not equal to zero.
\end{lemma}

\begin{proof}[Proof of Lemma \ref{Integer}]
For $j=1,\dots,N$, we set
\begin{equation}\label{New_d}
        f_j(\omega_1,\dots,\omega_{m-1})=f_{I_j}^X(\omega_1,\dots,\omega_{m-1},\sum_{i=1}^{m-1}\omega_i),
\end{equation}
then, we have
\begin{eqnarray}
\textrm{det}A&=&\begin{vmatrix}
f_1(\omega_{11}^1,\dots,\omega_{21}^{m_2-1}),&\dots&,f_1(\omega_{1N}^1,\dots,\omega_{2N}^{m_2-1})\\
\vdots&\ddots&\vdots\\
f_N(\omega_{11}^1,\dots,\omega_{21}^{m_2-1}),&\dots&,f_N(\omega_{1N}^1,\dots,\omega_{2N}^{m_2-1})
\end{vmatrix}=\frac{P(\omega_{11}^1,\dots,\omega_{2N}^{m_2-1})}{Q(\omega_{11}^1,\dots,\omega_{2N}^{m_2-1})},
\end{eqnarray}where $P$ and $Q$ are two polynomials of $(m-1)N$ variables.

We first note that $Q$ never vanishes over integer frequencies. Assume, by contradiction, that $P$ is always equal to zero for integer frequencies verifying Eq. $(\ref{NRMC})$.
Consider $P$ as a polynomial in one variable $\omega_{2N}^{m_2-1}$, i.e.,
\begin{eqnarray}
P(\omega_{11}^1,\dots,\omega_{2N}^{m_2-1})=\sum_{j=0}^{M}P_j(\omega_{11}^1,\dots,\omega_{2N}^{m_2-2})(\omega_{2N}^{m_2-1})^j.\label{P}
\end{eqnarray}
Given integer frequencies $(\omega_{11}^1,\dots,\omega_{2N}^{m_2-2})$, if Eq. $(\ref{P})$ is not identically equal to zero, then this polynomial in the variable $\omega_{2N}^{m_2-1}$ most has a finite number of roots. However,  for a given choice of $(m-1)N-1$ first frequencies, there exist an infinite number of $\omega_{2N}^{m_2-1}$ verifying $(\ref{NRMC})$. Then, $P_j=0$ over all integer frequencies, and $P_M$ is not identically equal to zero. We note that all $P_j$ are polynomials of $(m-1)N-1$ variables. Proceeding by induction on the number of variables, it is easy to see that, at the end, we obtain a polynomial in the variable $\omega_{11}^1$ which is equal to zero over all integer $\omega_{11}^1$, and which is not identically equal to zero according to Corollary $\ref{Det}$. That contradiction ends the proof of Lemma $\ref{Integer}$.
\end{proof}

\subsubsection{General case: $m>2$}\label{more-general}

Notice that the proof of Theorem \ref{Finitude} does not really depend on the number of vector fields involved in the control system (\ref{CS}). Indeed, for $m>2$, if the control functions are linear combination of sinusoids with integer frequencies, then the state variables in the canonical form are also linear combinations of sinusoids so that the frequencies are $\mathbb{Z}-$linear combinations of the frequencies occurring in the control functions.  The proof is the same as that of Lemma \ref{FG}, up to extra notation. Since Lemma \ref{Free_D} depends only on the length of the Lie brackets, but not on the number of vector fields, the proof of Lemma \ref{Free_Family} does not depend on $m$, either. In order to prove a similar result for $m>2$, we just need to re-project Eqs. (\ref{Proj1}) and (\ref{Proj2}) to $m$ vector fields instead of $2$.

\subsection{Exact and sub-optimal steering law}\label{sec:sub-opt}

In this section, we explain how we can devise, from Proposition \ref{Finitude}, an \emph{exact} and \emph{sub-optimal} steering law (cf. Definition \ref{de:sub-opt}) $\ex_{m,r}$ for the approximate system, which is already in the canonical form and how $\ex_{m,r}$ can be incorporated into the \emph{global approximate steering algorithm} (cf. Section \ref{GASA}). Note that $\ex_{m,r}$ only depends on the number of controlled vector fields $m$ and on the maximum degree of nonholonomy $r$.

Recall that the components of $x\in\R^{\tn^r}$ are partitioned into equivalence classes $\{\mathcal{E}_x^1,\mathcal{E}_x^2,\dots,\mathcal{E}_x^{\widetilde{N}}\}$ according to Definition \ref{Equivalence-Class} in such a way that $\mathcal{E}_x^i\prec\mathcal{E}_x^j$ for $(i,j)\in\{1,\dots,\widetilde{N}\}^2$ and $i<j$. For every equivalence class $\mathcal{E}_x^j$, Proposition \ref{Finitude} and Subsection \ref{more-general} guarantee that we can choose frequencies such that the corresponding control matrix $A_j$ (cf. Definition \ref{de:control-matrix}) is invertible and the corresponding control function $u^j$ obtained by Eq. (\ref{CF1}) steers all the elements of $\mathcal{E}_x^j$ from an arbitrary initial value to the origin $0$ (see Remark \ref{rem:goal0}) without modifying any elements belonging to smaller classes. 

Let $x^{\textrm{initial}}\in\R^{\tn^r}$. Let $B_j:=A_j^{-1}$ and $N_j:=\textrm{Card}(\mathcal{E}_x^j)$, $j=1,\dots,\widetilde{N}$. For $x\in\R^{\tn^r}$, we will use $[x]_{i,\dots,k}$ with $1\leq i<k\leq\tn_r$ to denote the vector $(x_{i},\dots, x_k)$, and $\Vert x\Vert$ to denote the \emph{pseudo-norm} of $x$ defined by the \emph{free} weights (cf. Definition \ref{de:pseudonorm} and Definition \ref{de:free-weight}). We also define an intermediate function $\textrm{Position}(u)$ as follows: if System (\ref{cano-f5}) starts from $x=0$ and evolves under the action of $u$, then $\textrm{Position}(u)$ returns its position at $t=2\pi$. 

\begin{algorithm}
\caption{Exact Steering Law: $\ex_{m,r}(x^{\textrm{initial}})$}
\label{algo_ex}
\begin{algorithmic}[1]
\REQUIRE $B_1,\dots,B_{\widetilde{N}}$, and $N_1,\dots,N_{\widetilde{N}}$; 
\STATE $\lambda:=\Vert x^{\textrm{initial}}\Vert_0$; 
\STATE $x^{\textrm{new}}:=\delta_{0,\frac{1}{\lambda}}(x^{\textrm{initial}})$; 
\STATE $\hat{u}_{\textrm{norm}}:=0$; 
\STATE $i:=0$; 
\FOR{$j=1,\dots,\widetilde{N}$}
  \STATE $x:=[x^{\textrm{new}}]_{i+1,\dots,i+N_j}$; 
  \STATE $a^j:=B_j~x$; 
  \STATE construct $u^j$ from $a^j$ by Eq. (\ref{CF1}); 
  \STATE $x^{\textrm{new}}:=x^{\textrm{new}}+\textrm{Position}(u^j)$; 
  \STATE $\hat{u}_{\textrm{norm}}:=\hat{u}_{\textrm{norm}}\ast u^j$ (cf. Definition \ref{de:conca}); 
  \STATE $i=i+N_j$; 
\ENDFOR 
\RETURN $\hat{u}:=\lambda\hat{u}_{\textrm{norm}}$.
\end{algorithmic}
\end{algorithm}

\begin{prop}\label{prop:sub-optimal}
For every $x^{\textrm{\emph{initial}}}\in \R^{\tn^r}$, the input given by \\ $\ex_{m,r}(x^{\textrm{\emph{initial}}})$ steers System (\ref{cano-f5}) from $x^{\textrm{\emph{initial}}}$ to $0$ exactly. Moreover, there exists a constant $C>0$ such that
\begin{equation}\label{eq:sub-optimal}
\ell(\ex_{m,r}(x^{\textrm{\emph{initial}}}))\leq Cd(x^{\textrm{\emph{initial}}},0),\qquad \forall~x^{\textrm{\emph{initial}}}\in \R^{\tn^r},
\end{equation}where we use $d$ to denote the sub-Riemannian distance defined by the family $X$.
\end{prop}

\begin{proof}[Proof of Proposition \ref{prop:sub-optimal}]

The fact that the procedure described by the Lines $5-12$ in {Algorithm \ref{algo_ex}} produces an input $\hat{u}_{\textrm{norm}}$ steering System (\ref{cano-f5}) from $\delta_{0,\frac{1}{\lambda}}(x^{\textrm{initial}})$ to $0$ is a consequence of Proposition \ref{Finitude} and Subsection \ref{more-general}. We also note that, due to the homogeneity of System (\ref{cano-f5}), if an input ${u}$ steers (\ref{cano-f5}) from $x$ to $0$, then, for every $\lambda>0$, the input $\lambda{u}$ steers (\ref{cano-f5}) from $\delta_{0,\lambda}(x)$ to $0$.  Therefore, the input computed by $\ex_{m,r}(x^{\textrm{{initial}}})$ steers System (\ref{cano-f5}) from ${x^{\textrm{{initial}}}}$ to $0$.
Let us now show (\ref{eq:sub-optimal}). In the sequel, the application $\ex_{m,r}: x\to \ex_{m,r}(x)$ will be simply denoted by $\hat{u}: x\to \hat{u}(x)$. Let $S(0,1):=\{y, \Vert y\Vert_0=1\}$ and $x\in\R^{\tn^r}$. Then, there exists $x_{\textrm{norm}}\in S(0,1)$ such that $x=\delta_{0,\lambda}(x_{\textrm{norm}})$ with $\lambda:=\Vert x\Vert_0$. We have:
\begin{eqnarray*}
\ell(\hat{u}(x))=\ell(\lambda\hat{u}(x_{\textrm{norm}}))=\lambda\ell(\hat{u}(x_{\textrm{norm}}))
\leq\lambda\sup_{y\in S(0,1)}\ell(\hat{u}(y)).
\end{eqnarray*}

We also know that, since the sub-Riemannian distance $d(0,\cdot)$ from $0$ and the pseudo-norm $\Vert\cdot\Vert_0$ at $0$ are both homogeneous of degree $1$ with respect to the dilation $\delta_{0,t}(\cdot)$, there exists a constant $\widetilde{C}>0$ such that $\widetilde{C}\lambda\leq d(0,x)$. Since the application $y\to \hat{u}(y)$ is continuous from $S(0,1)$ to $\R^m$ and $S(0,1)$ is compact, then, $\sup_{y\in S(0,1)}\ell(\hat{u}(y))$ is bounded, thus the inequality (\ref{eq:sub-optimal}) holds true.
\end{proof}

The following theorem is a consequence of Proposition \ref{prop:sub-optimal} and Remark \ref{rem:form-approx}.

\begin{theo}\label{th:sub-optimal}
The function $\ex_{m,r}(\cdot)$ constructed by \emph{Algorithm \ref{algo_ex}} provides the approximate system $\mathcal{A}^X$ defined in Section \ref{bellaiche-construction} with a sub-optimal steering law.
\end{theo}

\begin{proof}[Proof of Theorem \ref{th:sub-optimal}]

It suffices to note that, for every $a\in\Omega$, $\mathcal{A}^X(a)$ has the same form (cf. Remark \ref{rem:form-approx}), thus defines the same sub-Riemannian distance ${d}$. Therefore, the inequality (\ref{eq:sub-optimal}) holds uniformly with respect to the approximate point $a$, and this terminates the proof of Theorem \ref{th:sub-optimal}.
\end{proof}

\begin{rem}
Frequencies choices and the construction of the corresponding control matrix $A_j$, as well as its inverse $B_j$, translate to off-line computations.
We note that Proposition $\ref{Finitude}$ only gives sufficient conditions to prevent resonance (by choosing widely spaced frequencies, cf. Eq. $(\ref{NRMC})$) and guarantee the invertibility of the corresponding matrix (by using a sufficiently large number of independent frequencies). These conditions tend to produce high frequencies while it is desirable to find smaller ones for practical use.  We can prove that two independent frequencies suffice to steer one component (cf. Section \ref{special_case}), and we \emph{conjecture} that $2N$ independent frequencies suffice to control one equivalence class of cardinal $N$ by producing an invertible matrix. One can implement a searching algorithm for finding the optimal frequencies for each equivalence class such that they prevent all resonances in smaller classes and produce an invertible matrix for the class under consideration. Proposition \ref{Finitude} guarantees the finiteness of such an algorithm. Moreover, one can construct, once for all, a table containing the choice of frequencies and the corresponding matrices for each equivalence class in the free canonical system.
\end{rem}

\begin{rem}
Recall that the key point in our control strategy consists in choosing suitable frequencies such that, during each $2\pi-$period, the corresponding input function displaces components of one equivalence class to their preassigned positions while all the components of smaller classes (according to the ordering in Definition \ref{Order-equivalence-class}) return at the end of this control period to the values taken at the beginning of the period. In order to achieve the previous task, special resonance conditions must be verified by the appropriate components, and these conditions must not hold for all the other smaller components (according to the ordering in Definition \ref{Order-equivalence-class}). Note that two categories of frequencies have been picked in Proposition \ref{Finitude}: the basic frequencies $\{\omega_{ij}^k\}$, and the resonance frequencies $\{\omega_{ij}^*\}$. Since frequencies occurring in the dynamics of the state variables are just $\mathbb{Z}-$linear
combinations of $\{\omega_{ij}^k\}\cup\{\omega_{ij}^*\}$, and the resonance frequencies $\{\omega_{ij}^*\}$ are chosen to be special $\mathbb{Z}-$linear combinations of basic  frequencies (resonance condition), the frequencies in the dynamics of the state variables are special $\mathbb{Z}-$linear combinations of $\{\omega_{ij}^k\}$.
\end{rem}

\begin{rem}
Once the frequencies and matrices are obtained, the on-line computation $\ex_{m,r}$ is only a series of matrix multiplications which can be performed on-line without any numerical difficulty.
\end{rem}

\begin{rem}\label{rem:CNS}
The Desingularization Algorithm presented in Section \ref{Sec-lifting} (see also Remarks \ref{rem:nilpotent} and \ref{rem:nilpotent-general}) together with Algorithm \ref{algo_ex} provides general nilpotent control systems with an exact steering method, which is also sub-optimal.
\end{rem}

\begin{rem}\label{rem:smooth-control}
We note that the inputs constructed in this section are piecewise $C^\infty$ during each time interval $[2i\pi, 2(i+1)\pi]$, for $i=1,\dots,\widetilde{N}-1$, but they are not globally continuous during the entire control period $[0,~2\widetilde{N}\pi]$, due to discontinuity at $t=2\pi,4\pi,\dots, 2(\widetilde{N}-1)\pi$. However, it is not difficult to devise (globally) continuous
inputs using interpolation techniques. We illustrate the idea with a simple example. Assume that we use $u^i=(u^i_1, u^i_2)$ and $u^j=(u_1^j, u_2^j)$ defined by
\begin{eqnarray*}
u_1^i(t)&=&\cos\omega_{1i}t,\\
u_2^i(t)&=&\cos\omega_{2i}t+a^i\cos(\omega_{2I}^*t+\varepsilon^i\frac{\pi}{2}), \qquad t\in [2(i-1)\pi,2i\pi],\\
u_1^{j}(t)&=&\cos\omega_{1j}t, \\
u_2^{j}(t)&=&\cos\omega_{2j}t+a^{j}\cos(\omega_{2j}^*t+\varepsilon^{j}\frac{\pi}{2}), \qquad t\in[2(j-1)\pi, 2j\pi],
\end{eqnarray*} to steer two consecutive classes $\mathcal{E}^i_x$ and $\mathcal{E}^j_x$ (i.e. $j=i+1$) which are both of cardinal equal to $1$.

If we require their concatenation $u^i\ast u^j$ to be continuous, i.e.
\begin{eqnarray}
u_1^{i}(2\pi)&=&u_1^{j}(2\pi),\label{continuity1}\\
u_2^{i}(2\pi)&=&u_2^{j}(2\pi),\label{continuity2}
\end{eqnarray}we can proceed as follows. 

For Eq. $(\ref{continuity1})$, it suffices to modify slightly $u_1^{j}$. We take \begin{equation}\label{continuity3}\widetilde{u}_1^{j}(t)=u_1^{i}(2\pi)\cos \omega_{1J}t.\end{equation}

For Eq. $(\ref{continuity2})$, we distinguish two cases:
\begin{itemize}
	\item if $\varepsilon^j=1$, we can take
	\begin{eqnarray}
	 \widetilde{u}_2^{j}(t)&=&u_2^i(2\pi)\cos\omega_{2j}t+a^j\cos(\omega_{2j}^*t-\frac{\pi}{2})\nonumber\\
	&=&u_2^i(2\pi)\cos\omega_{2j}t+a^j\sin\omega_{2j}^*t;
	\end{eqnarray}
	\item if $\varepsilon^j=0$, we add a frequency $\omega_c$ to $u_{2}^j$, large enough to avoid any additional resonances,
	\begin{equation}
	 \widetilde{u}_2^{j}(t)=\cos\omega_{2j}t+a^j\cos\omega_{2j}^*t+(u_2^i(2\pi)-a^j-1)\cos\omega_ct.
	\end{equation}
\end{itemize}

\noindent Let $\widetilde{u}^j:=(\widetilde{u}_1^j,~\widetilde{u}_2^j)$. Then, by construction, the new input $u^i\ast\widetilde{u}^j$ is continuous over the time interval $[2i\pi,2j\pi]$.

It is clear that this idea of interpolation by adding suitable frequencies can be used to construct continuous inputs over the entire control period $[0,2\widetilde{N}\pi]$.
In fact, by using more refined interpolations, one can get inputs of class $C^k$ for arbitrary finite integer $k$.

\end{rem}

\begin{rem}\label{rem:smooth-traj-nilpotent}
Using the idea presented in Remark \ref{rem:smooth-control} together with Remark \ref{rem:CNS}, it is easy to conclude that, for general nilpotent systems, the resulting trajectories are globally $C^1$ curves and the regularity does not depend on the time-parameterization of the trajectories.
\end{rem}


\appendix

For the sake of completeness, we give in this short appendix the proof of Theorem \ref{th1} together with some comments on Algorithm \ref{algo}.

\section{Proof of Theorem \ref{th1}}

We first note that Steps $1$ through $5$ in Algorithm \ref{algo} are straightforward.

Theorem \ref{theo-desingularization} guarantees that the Desingularization Algorithm (Section \ref{Sec-lifting}) provides us with a new family of vectors fields $\xi=\{\xi_1,\dots, \xi_m\},$ which is regular and free up to step $r$ with $r$ denoting the maximum value of the degree of nonholonomy of the original system $X=\{X_1,\dots, X_m\},$ on the corresponding compact set $\cv^c_{\cj_i}$. Then, we construct the approximate system $\mathcal{A}^{\xi}$ using the procedure presented in Section \ref{bellaiche-construction} and provide it with the sub-optimal steering law $\ex_{m,r}$ defined in Algorithm \ref{algo_ex}. The sub-optimality of $\ex_{m,r}$ is guaranteed by Theorem \ref{th:sub-optimal}. Therefore, by Theorem \ref{regular-uniform}, the LAS method $\ap$ associated with $\mathcal{A}^{\xi}$ and its steering law is uniformly contractive on the compact set $\cv^c_{\cj_i}\times \bar{B}_R(0)$. Then, by Theorem \ref{th:GlobConv},\\ $\gf~(\tx_0,\tx_1,e,\cv^c,\ap)$ provided by Algorithm \ref{algo_gf} terminates in a finite number of steps and stops at a point $\tx$ such that $d(\tx, \tx_1)<e$. Since there is a finite number of compacts to be explored, we conclude that Algorithm \ref{algo} terminates in a finite number of steps and steers the system $(\Sigma)$ from $x^{\textrm{initial}}$ to some point $x$ such that $d(x,x^{\textrm{final}})<e$. This ends the proof of Theorem \ref{th1}.

\section{About the control set}\label{rem:control-set}
Let $U\subset\R^m$ be any neighborhood of the origin. Then every trajectory of (\ref{CS}) corresponding to the inputs produced by Algorithm \ref{algo} can be time-reparameterized so that the resulting trajectory of (\ref{CS}) is associated with an input taking values in $U$.   

\section{Getting trajectories of class $C^1$ for the original control system}\label{rem:smooth-control2}
We can slightly modify Algorithm \ref{algo} to get trajectories of class $C^1$ for  the original control system $(\Sigma)$. This is equivalent to ask for continuous inputs produced by the algorithm. According to Remark \ref{rem:smooth-control}, inputs can be made continuous within each iteration step in Algorithm \ref{algo_gf} since they are computed based the nilpotent approximate system. By using the same interpolation technique as presented in Remark \ref{rem:smooth-control}, we can still produce inputs which remain continuous from one step to another in Algorithm \ref{algo_gf}. Therefore, trajectories of class  $C^1$ for the control system $(\Sigma)$ are obtained.

\bibliographystyle{amsalpha}

\end{document}